\documentclass[12pt]{scrartcl}    
\usepackage{a4} 
\usepackage{amsmath}      
\usepackage{mathtools}      
\usepackage{paralist}     
\usepackage{amssymb}   
\usepackage{amsfonts}   
\usepackage{mathrsfs}   
\usepackage{dsfont}
\usepackage{latexsym} 
\usepackage{xcolor}
\usepackage{bbm,exscale}
\definecolor{Myblue}{rgb}{0,0,0.6}  
\usepackage[colorlinks,citecolor=Myblue,linkcolor=Myblue,urlcolor=Myblue,pdfpagemode=None]{hyperref}
\usepackage{amsthm}
\usepackage{accents}
\usepackage[square,numbers,sort&compress]{natbib} 
\usepackage[all,cmtip]{xy}
\usepackage{ifthen} 
\usepackage{bbding}
\usepackage{stmaryrd}  
\usepackage{wasysym}
\usepackage{verbatim}
\usepackage{bbding} 
\usepackage{soul}  
\usepackage[yyyymmdd,hhmmss]{datetime}
\usepackage{booktabs}
\usepackage{enumitem}
\usepackage{color}
\usepackage{textcomp}
\usepackage{gensymb}  
\usepackage{pdfpages}
\usepackage{tikz}
\usepackage{tikz-cd}
\usepackage{tikz-3dplot}
\usepackage{pgfplots}
\pgfplotsset{width=7cm,compat=1.8}
\usetikzlibrary{decorations.pathreplacing}
\usetikzlibrary{decorations.markings}
\usetikzlibrary{patterns}
\usepgflibrary{shapes.geometric}
\usepackage{datetime2}
\usepackage[export]{adjustbox} 
\usepackage[scr=boondoxo,scrscaled=1.00]{mathalfa}
\usepackage{subcaption}
\usepackage{appendix}
\usepackage{stackrel}

\tikzset{
	string/.style={draw=#1, postaction={decorate}, decoration={markings,mark=at position .51 with {\arrow[color=#1]{>}}}},
	costring/.style={draw=#1, postaction={decorate}, decoration={markings,mark=at position .51 with {\arrow[draw=#1]{<}}}},
	ostring/.style={draw=#1, postaction={decorate}, decoration={markings,mark=at position .47 with {\arrow[draw=#1]{>}}}},
	ustring/.style={draw=#1, postaction={decorate}, decoration={markings,mark=at position .56 with {\arrow[draw=#1]{>}}}},
	oostring/.style={draw=#1, postaction={decorate}, decoration={markings,mark=at position .43 with {\arrow[draw=#1]{>}}}},
	uustring/.style={draw=#1, postaction={decorate}, decoration={markings,mark=at position .59 with {\arrow[draw=#1]{>}}}},
	directed/.style={string=blue!50!black}, 
	odirected/.style={ostring=blue!50!black}, 
	udirected/.style={ustring=blue!50!black}, 
	oodirected/.style={oostring=blue!50!black}, 
	uudirected/.style={uustring=blue!50!black},     
	redirected/.style={costring= blue!50!black},
	redirectedgreen/.style={costring= green!50!black},
	directedgreen/.style={string= green!50!black},
	redirectedlightgreen/.style={costring= green!65!black},
	directedlightgreen/.style={string= green!65!black},
	redirectedred/.style={costring= red!50!black},
	directedred/.style={string= red!50!black}%
}

\tikzset{-dot-/.style={decoration={
			markings,
			mark=at position 0.5 with {\fill circle (1.875pt);}},postaction={decorate}}}

\tikzset{
	Fdot/.style={circle, draw, fill, inner sep=0pt}, 
	Odot/.style={circle, draw, inner sep=0.1pt, minimum size=0.1cm}
}

\newcommand{\eqrefO}[1]{\hyperref[eq:O#1]{\text{(O#1)}}}
\newcommand{\eqrefT}[1]{\hyperref[eq:T#1]{\text{(T#1)}}}

\newcommand\tikzzbox[1]
{pic}%

\vfuzz \hfuzz
\makeatletter
\newcommand{\raisemath}[1]{\mathpalette{\raisem@th{#1}}}
\newcommand{\raisem@th}[3]{\raisebox{#1}{$#2#3$}}
\makeatother

\newcommand{\A}{\mathcal{A}}
\newcommand{\CC}{\mathcal{C}}

\newcommand{\W}{\mathcal{W}}

\newcommand{\D}{\mathds{D}}

\newcommand{\R}{\mathds{R}}
\newcommand{\Z}{\mathds{Z}}

\def\1{\ifmmode\mathrm{1\!l}\else\mbox{\(\mathrm{1\!l}\)}\fi}
\newcommand{\one}{\mathbbm{1}}
\newcommand{\be}{\begin{equation}}
  \newcommand{\ee}{\end{equation}}
\newcommand{\bes}{\begin{equation*}}
  \newcommand{\ees}{\end{equation*}}

\newcommand{\Hom}{\operatorname{Hom}}
\newcommand{\End}{\operatorname{End}}

\newcommand{\ev}{\operatorname{ev}}

\newcommand{\tev}{\widetilde{\operatorname{ev}}}
\newcommand{\coev}{\operatorname{coev}}
\newcommand{\tcoev}{\widetilde{\operatorname{coev}}}
\newcommand{\too}{\longrightarrow}

\newcommand{\del}{\partial}
\def\lra{\longrightarrow}

\newcommand{\Bord}{\operatorname{Bord}}

\newcommand{\Borddef}{\operatorname{Bord}^{\mathrm{def}}}

\newcommand{\Borddefn}[1] {\operatorname{Bord}^{\mathrm{def}}_{#1}}
\newcommand{\Bordribn}[1] {\operatorname{Bord}^{\mathrm{rib}}_{#1}}

\newcommand{\Borddefen}[1] {\widehat{\operatorname{Bord}}{}^{\mathrm{def}}_{#1}}

\newcommand{\Sphere}{\operatorname{Sphere}^{\mathrm{def}}}
\newcommand{\Disc}{\operatorname{Disc}^{\mathrm{def}}}

\newcommand{\Bordd[1]}{\operatorname{Bord}_{#1}^{\mathrm{def}}(\mathds{D})}
\newcommand{\Bordstr}{\textup{Bord}^{\textup{str}}}

\newcommand{\In}{\mathrm{Ins}}

\newcommand{\zz}{\mathcal{Z}}

\newcommand{\zzc}{\mathcal{Z}^{\mathcal C}}
\newcommand{\zza}{\mathcal{Z}_{\mathcal A}}

\newcommand{\zzhat}{\widehat{\mathcal{Z}}}
\newcommand{\zzhata}{\widehat{\mathcal{Z}}_{\A}}
\newcommand{\wa}{{\mathcal W_{\A}}}
\newcommand{\zzwa}{\widehat{\mathcal{Z}}_{\A}^\Gamma}
\newcommand{\tz}{\mathcal{T}_{\mathcal{Z}}}
\newcommand{\tzhat}{\mathcal{T}_{\widehat{\mathcal{Z}}}}

\newcommand{\Vect}{\textrm{Vect}}

\newcommand{\eps}{\varepsilon}

\newcommand\arxiv[2]      {\href{https://arXiv.org/abs/#1}{#2}}
\newcommand\doi[2]        {\href{https://dx.doi.org/#1}{#2}}

\allowdisplaybreaks

\deffootnote[1em]{1em}{1em}{\textsuperscript{\thefootnotemark}}

\theoremstyle{definition} 
\newtheorem{definition}{Definition}
\newtheorem{proposition}[definition]{Proposition}
\newtheorem{theorem}[definition]{Theorem}

\newtheorem{lemma}[definition]{Lemma}
\newtheorem{corollary}[definition]{Corollary}
\newtheorem{remark}[definition]{Remark}

\newtheorem{example}[definition]{Example}
\newtheorem{examples}[definition]{Examples}
\newtheorem{construction}[definition]{Construction}

\newtheorem{convention}[definition]{Convention}

\numberwithin{equation}{section}
\numberwithin{definition}{section}
\numberwithin{figure}{section}

\newcommand\void[1]{}

\definecolor{DarkViolet} {rgb}{0.580392,0.000000,0.827450}

\begin{document}

\title{%
Orbifold graph TQFTs%
}

\author{%
	Nils Carqueville$^*$ \quad
	Vincentas Mulevi\v{c}ius$^\#$ \quad 
	Ingo Runkel$^\#$ \\ [0.3cm]
	Gregor Schaumann$^\vee$ \quad 
	Daniel Scherl$^\#$	
	\\[0.5cm]
	\normalsize{\texttt{\href{mailto:nils.carqueville@univie.ac.at}{nils.carqueville@univie.ac.at}}} \\  %
	\normalsize{\texttt{\href{mailto:vincentas.mulevicius@uni-hamburg.de}{vincentas.mulevicius@uni-hamburg.de}}} \\  %
	\normalsize{\texttt{\href{mailto:ingo.runkel@uni-hamburg.de}{ingo.runkel@uni-hamburg.de}}}\\[0.1cm]	\normalsize{\texttt{\href{mailto:gregor.schaumann@uni-wuerzburg.de}{gregor.schaumann@uni-wuerzburg.de}}} \\  %
	\normalsize{\texttt{\href{mailto:daniel.scherl@uni-hamburg.de}{daniel.scherl@uni-hamburg.de}}} \\[0.5cm]  %
	\hspace{-1.2cm} {\normalsize\slshape $^*$Fakult\"at f\"ur Physik, Universit\"at Wien, Austria}\\[-0.1cm]
	\hspace{-1.2cm} {\normalsize\slshape $^\#$Fachbereich Mathematik, Universit\"{a}t Hamburg, Germany}\\[-0.1cm]
	\hspace{-1.2cm} {\normalsize\slshape $^\vee$Institut f\"{u}r Mathematik, Universit\"{a}t W\"{u}rzburg, Germany}
}

\date{}
\maketitle

\begin{abstract}
A generalised orbifold of a defect TQFT $\zz$ is another TQFT $\zz_\A$ obtained by performing a state sum construction internal to $\zz$.
As an input it needs a so-called orbifold datum $\A$ which is used to label stratifications coming from duals of triangulations and is subject to conditions encoding the invariance under Pachner moves.
In this paper we extend the construction of generalised orbifolds of $3$-dimensional TQFTs to include line defects. 
The result is a TQFT acting on $3$-bordisms with embedded ribbon graphs labelled by a ribbon category~$\wa$ that we canonically associate to~$\zz$ and~$\A$. 
We also show that for special orbifold data, the internal state sum construction can be performed on more general skeletons than those dual to triangulations. This makes computations with $\zz_\A$ easier to handle in specific examples.
\end{abstract}

\newpage 

\tableofcontents

\newpage

\section{Introduction and summary}

By a generalised orbifold -- or just orbifold for short -- of a topological quantum field theory~$\zz$, we mean a state sum construction internal to~$\zz$, as initiated in \cite{ffrs0909.5013, CR3, CRS1}. 
Here~$\zz$ must necessarily be a defect TQFT, i.\,e.\ a symmetric monoidal functor on a category of stratified and decorated bordisms. 
The defining conditions on the datum~$\A$ from which the orbifold theory $\zz_\A$ is constructed 
encode invariance under decompositions of bordisms in the state sum construction. 
In \cite{CRS1}, these decompositions were taken to be stratifications that are dual to triangulations. 

If~$\zz$ is the trivial defect TQFT, then its orbifolds recover conventional state sum constructions; 
in the 2-dimensional case (where~$\A$ is a $\Delta$-separable symmetric Frobenius $\Bbbk$-algebra) this is implicit in~\cite{DKR}, while in \cite{CRS3} it was shown that Turaev--Viro--Barrett--Westbury models are 3-dimensional orbifolds (in particular, every spherical fusion category gives rise to an orbifold datum~$\A$). 
Another class of examples comes from discrete group actions (which can be ``gauged'') on arbitrary defect TQFTs, see e.\,g.\ \cite{BCP2, CRS3} for detailed discussions of the 2- and 3-dimensional cases, and \cite{SchweigertWoike201802} for a more geometric approach.
Indeed, this embeds the original meaning of ``orbifold'' as gauging a discrete symmetry into the setting of generalised orbifolds we consider here.

There are orbifolds beyond the unification of state sum models and the gauging of symmetry groups. 
For example, based on results of \cite{CR3}, in \cite{CRCR, OEReck} 2-dimensional orbifolds of Landau--Ginzburg models were constructed, which uncovered new relations between homological invariants of isolated singularities.
It was necessary for these applications to have a universal construction of defects for the orbifold TQFT in terms of a  representation theory of orbifold data internal to the original theory. 
In three dimensions, examples of orbifolds for the 3-dimensional Reshetikhin--Turaev theory based on Ising-type categories
were constructed in \cite{MuleRunk2}, inverting the extension of the modular fusion category for $\mathfrak{sl}(2)$ at level 10 by its commutative algebra of type E${}_6$. 

The 3-dimensional examples just mentioned build on the following general construction. 
Let~$\mathcal C$ be a modular fusion category, and let~$\zzc$ be the associated defect TQFT of Reshetikhin--Turaev type described in \cite{CRS2}. 
Then the main result of \cite{MuleRunk} states that for every (simple, special)
orbifold datum~$\A$ for~$\zzc$, one obtains another modular fusion category~$\mathcal C_{\A}$. 
It was conjectured in \cite{MuleRunk} that 
\be 
\label{eq:orb=orb}
\zz^{\mathcal C_{\A}} \cong \big(\zzc \big)_{\A} \, , 
\ee 
i.\,e.\ the Reshetikhin--Turaev defect TQFT associated to~$\mathcal C_{\A}$ is isomorphic to the $\A$-orbifold of the theory associated to~$\mathcal C$. 

\medskip 

In the present paper we develop a general theory of (Wilson) line defects in 3-dimensional orbifold TQFTs. 
This will be applied in the companion paper \cite{CMRSS2} to prove the conjecture~\eqref{eq:orb=orb}.
One notable consequence of this is that Reshetikhin--Turaev theories close under orbifolds. 

We now outline our general constructions, highlighting the three main contributions of this paper, which may be of independent interest: 
(i) the construction of orbifolds from decompositions that are computationally easier to deal with than dual triangulations, namely so-called admissible skeleta; 
(ii) the construction of a ribbon category~$\wa$ 
	of Wilson lines
associated to any orbifold datum~$\A$; 
(iii) the construction of a TQFT on bordisms with $\wa$-labelled ribbon graphs. 
The latter is the orbifold graph TQFT which gives this paper its title.

\subsubsection*{Admissible skeleta} 

According to \cite{CRS1, CRS3} 
(and as reviewed in some detail in Section~\ref{subsec:OrbifoldTQFTs}), 
an orbifold datum~$\A$ of a 3-dimensional defect TQFT~$\zz$ consists of defect labels $\A_3, \A_2, \A_1, \A_0^+, \A_0^-$ as in 
\be 
\includegraphics[scale=1.0, valign=c]{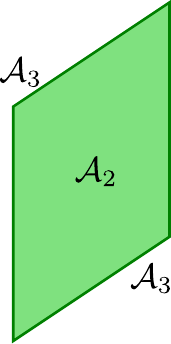}\, , \;\;
\includegraphics[scale=1.0, valign=c]{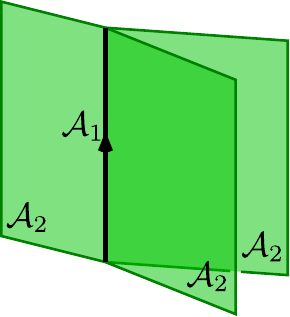}\, , \;\;
\includegraphics[scale=1.0, valign=c]{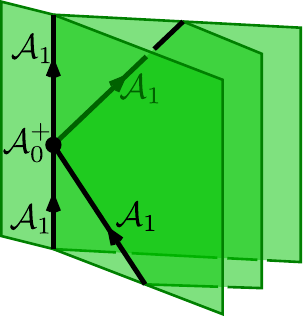}\, , \;\;
\includegraphics[scale=1.0, valign=c]{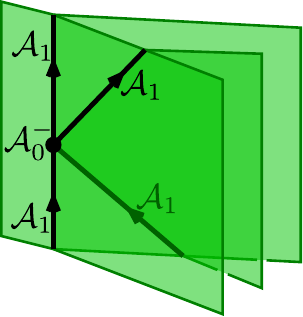}\, , 
\ee 
such that evaluating~$\zz$ on $\A$-decorated stratifications which are dual to suitably
oriented triangulations is independent of the choice of triangulation. 
Then the orbifold~$\zza$ is constructed as a colimit that arises from applying~$\zz$ to all these stratifications. 

We will show that instead of practically cumbersome stratifications dual to triangulations, for those~$\A$ which are ``special'' in the sense explained in Section~\ref{subsubsec:SOD}, one can compute~$\zza$ in terms of a simple type of stratification that we call ``admissible skeleta'' (which are fully oriented variants of the ``special skeleta'' of \cite{TVireBook}). 
These are stratifications where every 3-stratum is a ball and every point has one of the neighbourhoods listed in Figure~\ref{fig:skeleta}, see Definitions~\ref{def:Skeleta} and~\ref{def:admissible-skeleton} for details. 
We show that any two admissible skeleta can be related by three types of moves (bubble, lune, and triangle, BLT for short), see Figure~\ref{fig:BLTMoves}. 

In particular, every admissible skeleton can be consistently decorated with a special orbifold datum~$\A$, and the defining conditions on~$\A$ ensure invariance under BLT moves.
Then Theorem~\ref{thm:MainResultCRS1} explains how to construct~$\zza$ from admissible skeleta. 
As an example, note that an embedding $S^2 \subset S^3$ gives an admissible skeleton which is not dual to a triangulation.

\subsubsection*{A ribbon category of Wilson lines}

Every 3-dimensional defect TQFT~$\zz$ gives rise to a 3-category~$\tz$, see \cite{CMS}. 
The objects of~$\tz$ are interpreted as bulk theories, while 1-, 2-, and 3-morphisms are interpreted as surface, line, and point defects, respectively. 
In Sections~\ref{subsec:MonoidalCatFromDefectTQFT} and~\ref{subsec:RibbonCategoriesFromSOD} we will explain how this implies that for every orbifold datum~$\A$ for~$\zz$, one obtains a ribbon category~$\wa$. 
Objects of~$\wa$ are line defects~$X$ in $\A_2$-decorated surface defects which can cross $\A_1$-decorated line defects at point defects $\tau_1^X, \tau_2^X$, 
\be 
\includegraphics[scale=1.0, valign=c]{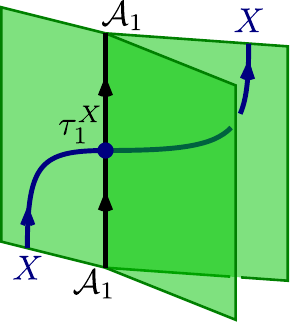} 
\, , \quad 
\includegraphics[scale=1.0, valign=c]{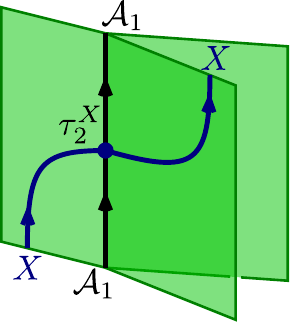} 
\, , 
\ee 
that are subject to the compatibility conditions in Figure~\ref{fig:CrossingIdentities}. 
We allow the line defects to have non-trivial ``internal structure''; for example, 
\be 
\label{eq:LineDefectWithInternalStructure}
\includegraphics[scale=1.0, valign=c]{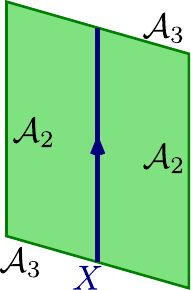} \quad = \,
\includegraphics[scale=1.0, valign=c]{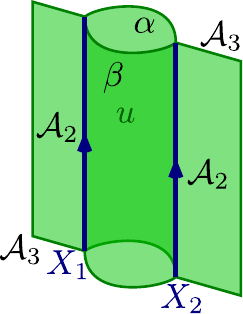}
\ee 
is a ``line defect'', where any (line; surface; bulk) defect labels $X_1,X_2$; $\alpha, \beta$; $u$ allowed by~$\zz$ may occur. 

Morphisms in~$\wa$ by definition have to intertwine with $\tau_1^X, \tau_2^X$, and it is straightforward to give~$\wa$ the structure of a rigid monoidal category. 
Moreover, the diagrams
\be 
\includegraphics[scale=1.0, valign=c]{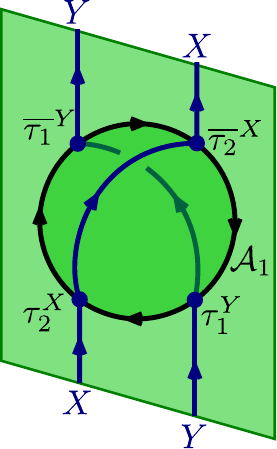}
\, , \quad
\includegraphics[scale=1.0, valign=c]{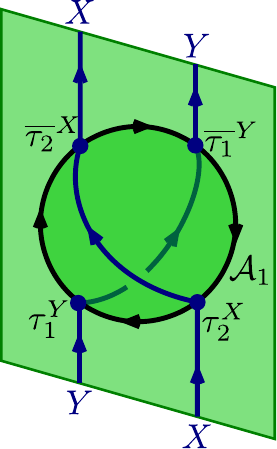}
\, , 
\ee 
when evaluated with a certain completion~$\zzhat$ of~$\zz$ (that can handle line defects as in~\eqref{eq:LineDefectWithInternalStructure}, see Section~\ref{subsec:RibbonCategoriesFromSOD} for details) endow~$\wa$ with a braiding (Proposition~\ref{prop:WA-ribbon-cat}). 

We view the ribbon category~$\wa$ as a natural algebraic invariant attached to the orbifold datum~$\A$ for~$\zz$.
In \cite{CMRSS2} we will show that if~$\zz$ is the Reshetikhin--Turaev defect TQFT of \cite{CRS2}, then~$\wa$ is equivalent to the modular fusion category~$\mathcal C_\A$ of \cite{MuleRunk}. 
There is however no reason for~$\wa$ to be semisimple in general.

\subsubsection*{Orbifold graph TQFTs}

Recall from \cite{turaevbook, TVireBook} that a graph TQFT is a symmetric monoidal functor on the bordism category $\Bordribn{3}(\mathcal C)$ with embedded ribbon graphs that are labelled by some fixed $\Bbbk$-linear
ribbon category~$\mathcal C$ for a field $\Bbbk$. 
Our main result is a lift of the orbifold TQFT $\zza \colon \Bord_3 \lra \Vect$ to a graph TQFT 
\be 
\label{eq:MainResult}
\zzwa \colon \Bordribn{3}(\wa) \lra \Vect 
\, , 
\ee 
where $\Vect$ denotes the category of 
	$\Bbbk$-vector spaces. 
To do so, we adapt the formalism of \cite{TVireBook} to ``represent'' every $\wa$-labelled ribbon graph in a given bordism~$M$ by pushing it into an $\A$-decorated admissible skeleton of~$M$, see Figure~\ref{fig:PushingRibbonGraphsAround} for an illustration. 
Analogously to how one finds that the choice of admissible skeleton is immaterial in the construction of~$\zza$, we prove (see Theorem~\ref{thm:MainResultForEulerComplete}) that the construction of~$\zzwa$ is independent of the choice of admissible skeleton and how precisely the $\wa$-labelled ribbon graph is pushed into it. 

\begin{figure}
\centering
\captionsetup[subfigure]{labelformat=parens}
\begin{subfigure}[b]{0.20\textwidth}
	\centering
	\includegraphics[scale=1.4, valign=c]{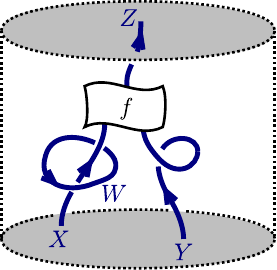}
	\caption{}
\end{subfigure}
\begin{subfigure}[b]{0.78\textwidth}
	\centering
	\includegraphics[scale=1.4, valign=c]{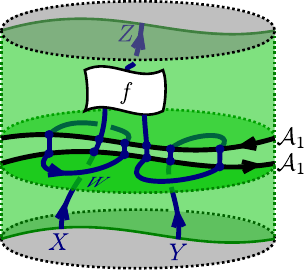} ,
	\includegraphics[scale=1.4, valign=c]{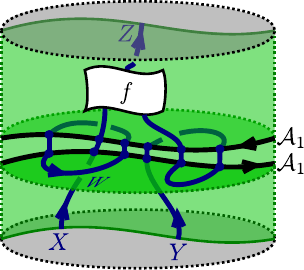}
		\hspace*{-7em}~
	\caption{}
\end{subfigure}
\caption{(a) A local patch of a bordism~$M$ with embedded $\wa$-labelled ribbon graph~$\mathcal R$. 
		 (b) A local patch of an admissible $\A$-decorated skeleton~$\mathcal S$ for~$M$ (where all green 2-strata implicitly carry a label~$\A_2$) together with two choices to represent~$\mathcal R$ in~$\mathcal S$, which are necessarily related by the moves in Figures~\ref{fig:BLTMoves} and~\ref{fig:omegaMoves}.}
\label{fig:PushingRibbonGraphsAround}
\end{figure}

\medskip 

The remainder of the present paper is organised as follows. 
In Section~\ref{sec:TopologicalPreliminaries} we introduce admissible skeleta, representations of ribbon graphs with respect to such skeleta, as well as the moves that connect them. 
Most of the technical details related to this discussion are contained in Appendix~\ref{app:A}. 
In Section~\ref{sec:DefectTQFTs} we briefly review 3-dimensional defect as well as 
graph TQFTs, and we construct a ribbon category of Wilson lines from any defect TQFT. 
In Section~\ref{sec:OrbGrphTQFTs}, after a short recollection of orbifold TQFTs, we define a ribbon category~$\wa$ associated to a special orbifold datum~$\A$, and then construct the orbifold graph TQFT~\eqref{eq:MainResult}.

\subsubsection*{Acknowledgements}

N.\,C.\ is supported by the DFG Heisenberg Programme. 
V.\,M.\ is partially supported by the DFG Research Training Group 1670.
I.\,R.\ is partially supported by the Cluster of Excellence EXC 2121.

\section{Topological preliminaries}
\label{sec:TopologicalPreliminaries}

In this section we set the topological stage for our constructions. 
Section~\ref{subsec:stratifications} collects our conventions for 3-dimensional stratified bordisms. 
In Section~\ref{subsec:Skeleta} we introduce a particular class of stratifications called ``admissible skeleta''. 
We show that these are related by the ``BLT moves'' of Figure~\ref{fig:BLTMoves}, which will feature prominently in later sections. 
Then a brief review of bordisms with embedded ribbon graphs in Section~\ref{sec:unlabelled-ribbon-bordisms} is followed by an account of how to represent ribbon graphs with respect to admissible skeleta in Section~\ref{subsuc:RibbonDiagramsOmegaMoves}, and how different such representations are related by the ``$\omega$-moves'' of Figure~\ref{fig:omegaMoves}. 

Our discussion here heavily draws from \cite{TVireBook}. 
The main novelty is that we carefully check that everything can be made admissibly oriented in our sense.

\subsection{Stratifications}
\label{subsec:stratifications}

We recall the stratifications used in \cite{CMS, CRS1}, to which we refer for more details.

\subsubsection{Stratified manifolds} 
\label{sssec:stratMan}

By an \textsl{$n$-dimensional stratified manifold} we mean an $n$-dimensional topological manifold $M$ (without boundary) together with a \textsl{stratification} $S$ of $M$, which is given by a filtration $\varnothing = F^{(-1)} \subset F^{(0)} \subset F^{(1)} \subset \ldots \subset F^{(n)} = M$ of topological spaces such that for each $j\in \{0,1,\ldots,n\}$, $S^{(j)} := F^{(j)}\setminus F^{(j-1)}$ has the structure of a smooth $j$-dimensional manifold (such that the smooth structure is compatible with the subspace topology). 
The connected components of $S^{(j)}$ are called \textsl{$j$-strata}. 
We denote the set of $j$-strata by $S_j$, and we ask each $S_j$ to be finite. 
For $s \in S_i$, $t \in S_j$ we require that whenever $s \cap \overline{t} \neq \varnothing$, then already $s \subset \overline{t}$. 
In this case necessarily $i<j$ and we say that $s$ and $t$ are \textsl{incident} to each other.
For any $x \in M$ we say that $x$ and $s$ are incident to each other if $x \in \bar{s}$. 

We denote by $S_j(x)$ the set of germs of $j$-strata around $x$, i.\,e.\ the inverse limit of the canonical maps $S_j^\eps(x) \too S_j^\delta(x)$ for $\eps < \delta$, where $S_j^\eps(x)$ is the set whose elements are intersections of $j$-strata and a ball of radius~$\eps$ around~$x$ (in some chart).

\begin{example}\label{ex:partTor}
  Consider the following stratification~$S$ of the 3-sphere which has two 3-strata (the interior of the coloured solid torus with one disc removed, as well as its complement in~$S^3$ -- which is also a torus, and not coloured in the picture), two 2-strata and one 1-stratum, with a chosen point~$x$ in the disc-shaped 2-stratum: 
  \be
  	\includegraphics[scale=1.0, valign=c]{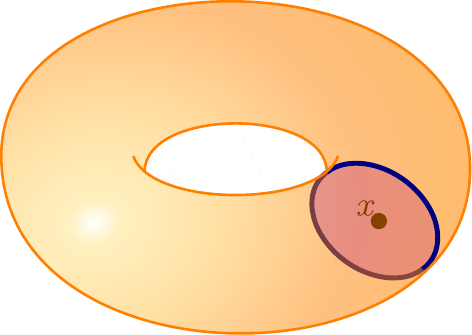}
  \ee
  Then $S_3(x)$ has two elements, even though $x$ is only incident to a single 3-stratum.
\end{example}

An \textsl{$n$-dimensional stratified manifold with boundary} is an $n$-dimensional topological manifold $M$ with boundary $\del M$ together with a filtration as above such that the interior of $M$ is a stratified manifold without boundary.
Furthermore we demand that each stratum $s$ satisfies $\del s = s \cap \del M$. 
It follows that $\del M$ canonically inherits the structure of an $(n-1)$-dimensional stratified manifold (without boundary).
Let us stress again that while the overall manifold~$M$ is topological, each stratum is a smooth manifold, cf.\ \cite[Footnote\,4]{CRS1}. 

A map $f\colon M \too M'$ of $n$-dimensional stratified manifolds with boundary is a continuous map that sends strata to strata and restricts to a smooth map on each stratum, and such that~$f$ restricts to a map $f\rvert_{\del M}\colon \del M \too \del M'$ that is a map of $(n-1)$-dimensional stratified manifolds. 

An \textsl{oriented stratified manifold} (possibly with boundary) is a stratified manifold (possibly with boundary) such that the underlying manifold~$M$ and all strata carry a prescribed orientation, such that each top-dimensional stratum carries the orientation induced from $M$. 
Maps of oriented stratified manifolds are defined as above, except that in each step we additionally require that the restriction to each stratum is orientation-preserving.

\subsubsection{Stratified bordisms}
\label{subsubsec:StratifiedBordisms}

An \textsl{$n$-dimensional oriented stratified bordism} is a tuple $M=(M,\Sigma_-, \Sigma_+, \varphi_-, \varphi_+)$, where $M$ is an $n$-dimensional compact oriented stratified manifold (possibly with boundary), $\Sigma_-$ and $\Sigma_+$ are $(n-1)$-dimensional compact closed oriented stratified manifolds and $\varphi_{\mp}$ are germs (in $\eps \in \R_{>0}$) of embeddings of stratified manifolds 
\be
\varphi^{\eps}_\mp \colon \Sigma_\mp \times [0,\eps) \too M \, .
\ee 
Each $\varphi^\eps_+$ is required to be orientation-preserving, and each $\varphi^\eps_-$ is orientation-reversing. 
Furthermore $\del M$ splits as a disjoint union 
\be
\del M = \varphi_-(\Sigma_-\times\{0\}) \sqcup \varphi_+(\Sigma_+\times\{0\}) \, .
\ee 
We refer to $\varphi_\mp(\Sigma_\mp) \times \{0\}$ as the \textsl{in-} and \textsl{out-boundary} of~$M$, respectively. 
The \textsl{source} of $M$ is $\Sigma_-$ and the \textsl{target} is $\Sigma_+$.
A morphism of stratified bordisms $M$ and $M'$ is a map $f\colon M \too M'$ of oriented stratified manifolds such that $f\varphi_{\mp}=\varphi'_{\mp}$.

There is a symmetric monoidal category of $n$-dimensional oriented stratified bordisms $\Bordstr_n$ as follows. 
Objects of $\Bordstr_n$ are $(n-1)$-dimensional oriented stratified closed manifolds. 
A morphism $\Sigma \too \Sigma'$ is an isomorphism class of $n$-dimensional oriented stratified bordisms with source $\Sigma$ and target $\Sigma'$. 
We will often use a bordism and its isomorphism class synonymously. 
Composition of morphisms is defined by gluing along common boundaries. 
Requiring germs of embeddings $\Sigma_{\mp}\times [0,\eps) \too M$ rather than just embeddings of the boundaries themselves ensures a canonical smooth structure on the strata of a composite bordism. 
The tensor product of $\Bordstr_n$ is given by disjoint union, and the symmetric braiding by mapping cylinders of the twist maps on disjoint unions.

\subsubsection{Defect bordisms}
\label{subsubsec:DefectBordisms}

\textsl{Defect bordisms} are oriented stratified bordisms satisfying an additional regularity condition, imposed by requiring the existence of certain local neighbourhoods around each point. 
The sets of \textsl{local neighbourhoods for $n$-dimensional defect bordisms} are denoted by $\mathcal{N}_n$. The elements of $\mathcal{N}_n$ are oriented stratified open manifolds of dimension $n$.
They are defined inductively for arbitrary~$n$ in \cite[Sect.\,2.2]{CRS1}.
Here we only give a brief discussion of the case $n=3$. 

\medskip

In order to define the local neighbourhoods for 3-dimensional defect bordisms, we first have to consider the 2-dimension case. 
There are three types of local neighbourhoods for 2-dimensional defect bordisms in $\mathcal{N}_2$: 
\be \label{eq:N2-neighbourhoods}
\includegraphics[scale=1.0, valign=c]{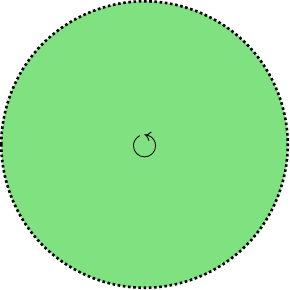} \, , \quad
\includegraphics[scale=1.0, valign=c]{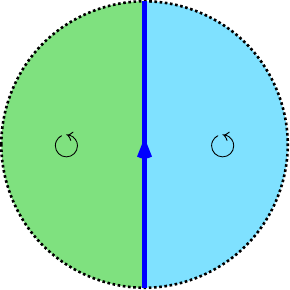} \, , \quad
\includegraphics[scale=1.0, valign=c]{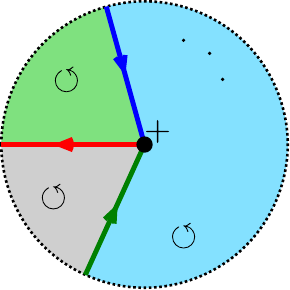} \, .
\ee
Note that there are infinitely many neighbourhoods of the third type, and any choice of orientation for the 0-stratum and the 1-strata is allowed.
A \textsl{defect 2-manifold} is an oriented stratified manifold, such that each point has a neighbourhood isomorphic (as an oriented stratified manifold) to an element of $\mathcal{N}_2$.
     A \textsl{defect 2-sphere} is a defect 2-manifold with  underlying  manifold is $S^2$. 
\medskip

A \textsl{3-dimensional defect bordism} is a 3-dimensional stratified bordism such that each point has a neighbourhood isomorphic (as an oriented stratified manifold) to one of the following list: 
\begin{enumerate}
	\item 
	Open cylinders $X \times (-1,1)$, where $X \in \mathcal{N}_2$, with orientations induced from~$X$ and the standard orientation of $(-1,1)$. For example:
	\be
	X =
	\includegraphics[scale=1.0, valign=c]{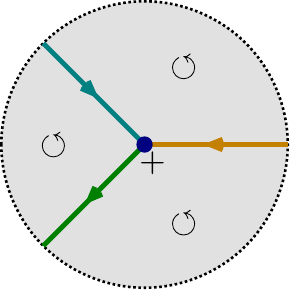}
	\in \mathcal{N}_2 \, , \quad
	X \times (-1,1) =
	\includegraphics[scale=1.0, valign=c]{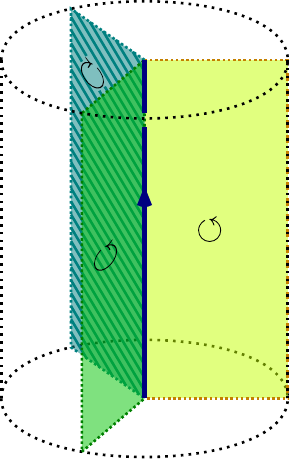} 
	\ee
	\item Open cones
	\[
	C(\Sigma) = (\Sigma \times [0,1)) / (\Sigma \times \{0\}) 
	\, ,
	\] 
	where $\Sigma$ is a defect 2-sphere. 
	Open cones have a natural structure of stratified manifolds with underlying manifold the open $3$-ball $B^{3} \subset \R^{3}$. 
	The cone point defines a 0-stratum at $0 \in \R^{3}$. Due to the choice of orientation~$\pm$ for the cone point each defect 2-sphere gives rise to two elements of $\mathcal{N}_{3}$. 
	An example for orientation ``$+$'' is:
	\be
	\includegraphics[scale=1.0, valign=c]{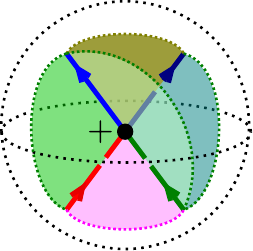}
	\ee
\end{enumerate}

We obtain the symmetric monoidal non-full subcategory $\Borddef_3$ of $\Bordstr_3$ whose objects
are compact closed defect 2-manifolds, and whose morphisms are given by only those stratified bordisms that locally look as specified above.

\subsection{Skeleta}
\label{subsec:Skeleta}

In this subsection we present a class of stratifications, called skeleta, that are well-suited for the procedure of ``orbifolding'' in Section~\ref{sec:OrbGrphTQFTs} below. 
Duals of triangulations form a proper subset of all skeleta, which in turn form a proper subset of the stratifications allowed for defect bordisms.

\begin{definition}\label{def:Skeleta}
	Let $M$ be an (unoriented) 3-manifold, possibly with boundary. A \textsl{skeleton} $S$ of $M$ is a stratification of $M$ that satisfies the following additional requirements.
	\begin{enumerate}[label={(\roman*)}]
		\item 
		\label{item:Skeleton1}
		Every 3-stratum is diffeomorphic to either an open 3-ball if it does not intersect $\del M$, or to an open half-ball otherwise.
		\item 
		Each $x \in M$ has a neighbourhood $B_x$ that is isomorphic (as a stratified manifold) to one from the list in Figure \ref{fig:skeleta}. 
		In each case $B_x$ is an open ball if $x \notin \del M$, and an open half-ball otherwise. 
	\end{enumerate} 
\end{definition}

\begin{figure}
	\begin{enumerate}[label={(\roman*)}]
		\item 
		If $x \in S^{(3)}$ then $B_x$ contains no 2- or lower strata.
		\item 
		If $x \in S^{(2)}$ and $x \notin \del M$ then $B_x$ contains a single 2-stratum and is given by
		\[
		\includegraphics[scale=1.0, valign=c]{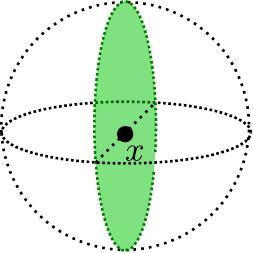}.
		\]
		\item 
		If $x \in S^{(2)}$ and $x \in \del M$ then $B_x$ is given by
		\[
		\includegraphics[scale=1.0, valign=c]{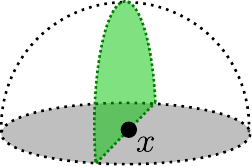}.
		\]
		\item 
		If $x \in S^{(1)}$ and $x \notin \del M$ then $B_x$ is given by 
		\[
		\includegraphics[scale=1.0, valign=c]{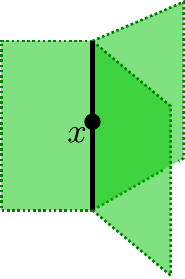}.
		\]
		\item 
		If $x \in S^{(1)}$ and $x \in \del M$ then $B_x$ is given by 
		\[
		\includegraphics[scale=1.0, valign=c]{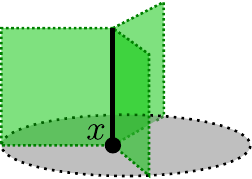}.
		\]
		\item 
		\label{item:Skeleta6}
		$S^{(0)} \cap \del M = \varnothing$ and if $x \in S^{(0)}$ then $B_x$ is given by 
		\[
		\includegraphics[scale=1.0, valign=c]{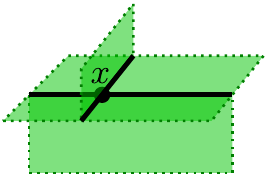}.
		\]
	\end{enumerate}
	\caption{List of allowed local neighbourhoods for skeleta~$S$ of a 3-manifold with boundary~$M$. 
		As in Section~\ref{sssec:stratMan}, $S^{(j)}$ denotes the union of all $j$-strata. 
		2-strata are depicted in green, 1-strata in black, and the boundary of~$M$ is grey (colour available online).} 
	\label{fig:skeleta}
\end{figure}
	
\begin{remark}\label{rem:3CellsBoundary}
	Condition~\ref{item:Skeleton1} of Definition~\ref{def:Skeleta} implies that no 3-stratum of a skeleton intersects both the in- and out-boundary of~$M$ non-trivially.
\end{remark}

If $M$ is oriented, then an \textsl{oriented skeleton} of~$M$ is a skeleton that is oriented as a stratification.  
In particular each 3-stratum carries the same orientation as~$M$, but there are no restrictions on the orientations of 2-, 1- and 0-strata.

Every stratification that is obtained as the Poincar\'e dual of a triangulation is a skeleton. 
An example of a skeleton that does not arise in this way is the (unoriented version of the) skeleton of $S^3$ in Example~\ref{ex:oriented_adm_skeleta}\,\ref{item:SphereSkeletonExample} below. 
An example of a stratification that is not a skeleton is the stratification of $S^3$ in Example~\ref{ex:partTor} 
(the uncoloured 3-stratum outside of the coloured solid torus is not a 3-ball).

\begin{remark}
In the terminology used in \cite[Sect.\,11.5.1]{TVireBook}, condition~\ref{item:Skeleta6} of Figure~\ref{fig:skeleta} means that each 0-stratum is a \textsl{special point}. 
The notion of skeleta given in \cite[Sect.\,11.2.1]{TVireBook} is more general than the one given here in that it allows for more diverse local situations than the ones specified in Figure~\ref{fig:skeleta}.
Our skeleta are the ``s-skeleta'' of loc.\ cit., except that we do not demand at least one special point in each connected component, and we allow for circles as 1-strata. 
\end{remark}

\subsubsection{Admissible skeleta}
\label{subsubsec:AdmissibleSkeleta}

For the remainder of Section~\ref{sec:TopologicalPreliminaries}, $M$ will denote an arbitrary but fixed oriented 3-manifold (possibly with boundary), and all skeleta are to be taken within $M$, if not explicitly stated otherwise. 
Recall from Section~\ref{sssec:stratMan} the definition of $S_j(x)$ as the set of germs of $j$-strata around a point $x \in M$. 

\begin{definition} \label{def:LocalOrder}
	A \textsl{local order} on a skeleton~$S$ is for each point $x \in M$ a choice of total order on $S_3(x)$ such that for any two points $x,y$ in the closure of a given $j$-stratum with $j\in\{0,1,2\}$, 
	the corresponding orders are compatible in the following sense: 
	whenever~$A$, $B$ are 3-strata incident with $x$ and $y$ that induce the elements~$a$ and~$b$ in $S_3(x)$ as well as~$a'$ and~$b'$ in $S_3(y)$, respectively, then $a<b \Longleftrightarrow a'<b'$.
\end{definition}

Note that if~$x$ is a point in a $j$-stratum of a skeleton~$S$, then $S_3(x)$ has precisely $4-j$ elements (but possibly fewer incident 3-strata). 
If $x,y\in M$ are both contained in the same stratum $s$, then\footnote{More precisely, there is a canonical isomorphism along which we identify the two sets.} $S_3(x)=S_3(y)$, and we define $S_3(s) := S_3(x)$ where $x \in s$ is arbitrary. 
Indeed, let $x,y \in s$. If $x$ is contained in some $B_y$ or $y$ is contained in some $B_x$ as in Figure \ref{fig:skeleta}, the claim $S_3(x)=S_3(y)$ is clear.
Otherwise, since $s$ is path-connected, we can consider a path~$\gamma$ from~$x$ to~$y$ that lies in~$s$, and transport the order along an open cover of~$\gamma$.
In light of this we can also define a local order on $S$ as a choice of total order on $S_3(s)$ for each stratum $s$ of $S$ such that for any two strata $s,t$ the induced orders on $S_3(s) \cap S_3(t)$ agree. 
Hence a local order on $S$ orders the germs of 3-strata of $S$ around each lower-dimensional stratum.

Any local order on a skeleton~$S$ turns it into an oriented skeleton by the following convention.

\begin{convention}\label{conv:indOrd}
	All 3-strata carry the orientation induced by the orientation of~$M$.
	The orientations for 2-strata are obtained via the right-hand rule:
	\be\label{eq:orientation-2strata}
		\includegraphics[scale=1.0,valign=c]{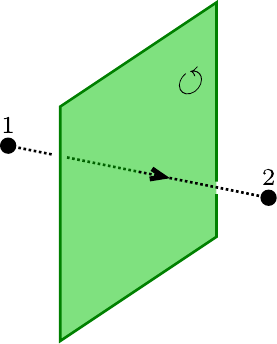}~,
	\ee
	where here and below the numbers on free-floating vertices indicate the local order on the ambient germs of 3-strata. 
	The orientations of 1- and 0-strata are determined as follows:
	\be\label{eq:orientation-10strata}
		\includegraphics[scale=1.0,valign=c]{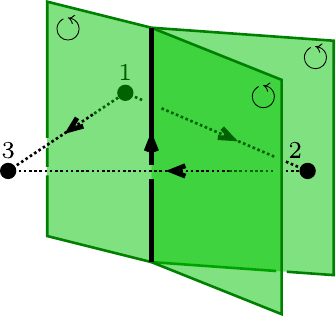} \, , \quad
		\includegraphics[scale=1.0,valign=c]{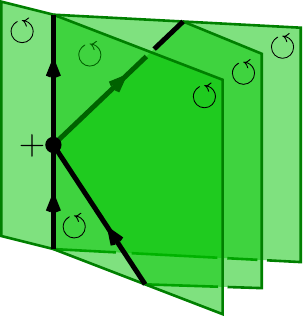} \, , \quad
		\includegraphics[scale=1.0,valign=c]{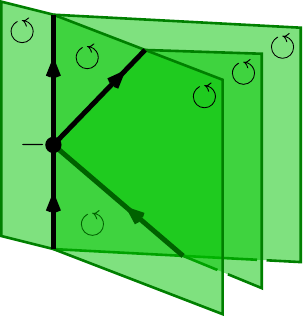} \, .
	\ee	
	As pictured above, by default we assume that 2-strata have the standard orientation of the paper/screen plane. 
	We indicate the opposite, i.\,e.\ clockwise, orientation by a stripy pattern, for example\footnote{ 
	These conventions are consistent with those in \cite{CRS2,CRS3,MuleRunk,MuleRunk2}, but they differ slightly from those in \cite{CRS1}, where the orientations of 2-strata are flipped.}
	\be
		\includegraphics[scale=1.0,valign=c]{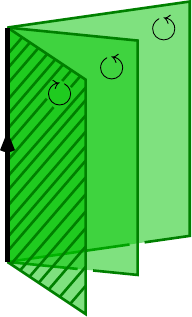}
		\, . 
	\ee 
\end{convention}

\begin{definition}\label{def:admissible-skeleton}
	An \textsl{admissible} skeleton of~$M$ is an oriented skeleton whose orientation is induced by a local order. 
	We denote the set of admissible skeleta of~$M$ by $\mathscr{S}(M)$.
\end{definition}

\begin{remark}\label{rem:ltovsorient}
	\begin{enumerate}[label={(\roman*)}]
		\item 
		\label{item:RemarkLocalOrder1}
		A local order is uniquely determined by the oriented skeleton that it induces. 
		In fact the local order can be recovered from only the orientations of 2-strata in the induced oriented skeleton via Convention \ref{conv:indOrd}. 
		Hence the datum of an admissible skeleton is the same as that of an unoriented skeleton together with a local order.
		\item 
		\label{item:RemarkLocalOrder2}
		If~$S$ is dual to a triangulation $T$ of~$M$, then Definition~\ref{def:LocalOrder} reduces to the notion of an \textsl{ordering} of a simplicial complex, see e.\,g.\ \cite[p.\,2]{JT}. 
		Such an ordering is given by a total order on the vertices of each 3-simplex of~$S$ such that the induced orders on shared faces agree. 
		By dualising Convention~\ref{conv:indOrd}, any ordering of a simplicial complex induces an orientation of all its simplices, in particular 1-simplices are oriented away from vertices of lower order. 
		If the orientation of~$T$ is induced by an ordering, then it can be uniquely recovered from just the orientations of 1-simplices.
		In turn the datum of an ordering on~$T$ is the same as that of an orientation of each 1-simplex of~$T$ such that no loops are formed around any single simplex. In this case we call $T$ an \textsl{admissible triangulation}.
	\end{enumerate}
\end{remark} 

\begin{examples}\label{ex:oriented_adm_skeleta}	
	\begin{enumerate}[label={(\roman*)}]
		\item 
		\label{item:SphereSkeletonExample}
		A local order on a skeleton $S$ does not necessarily induce a total order on the set $S_3$ of all 3-strata of $S$, as is illustrated by the following admissible skeleton of $S^3$:
		\be
		\includegraphics[scale=1.0,valign=c]{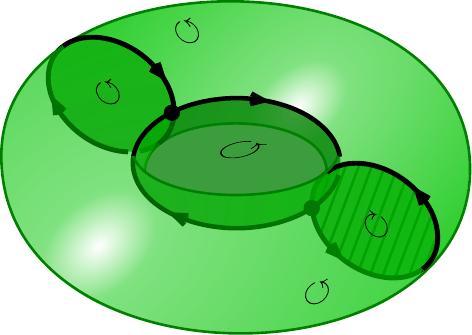} 
		\ee
		Indeed, for any given order on the two 3-strata making up the two halves of the solid torus, at least one of the two disc-shaped 2-strata separating them would have the wrong orientation. 
		\item 
        Using Remark \ref{rem:ltovsorient}\,\ref{item:RemarkLocalOrder1} it is straightforward to construct examples of oriented skeleta that are not admissible by simply choosing orientations for 2-strata that do not match any of those allowed by Convention \ref{conv:indOrd}. 
        For example, any oriented skeleton that locally looks as follows is not admissible:
		\be
			\includegraphics[scale=1.0,valign=c]{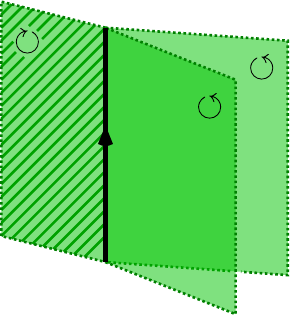} 
		\ee	
	\end{enumerate}
\end{examples}

\subsubsection{Local moves on skeleta}
\label{subsubsec:BLTmoves}

We now introduce local moves on oriented skeleta. 
We refer to these moves as \textsl{BLT moves} (short for bubble, lune, and triangle moves).
Our list of moves is a slight modification of the moves that are considered in \cite[Sect.\,11.3--11.4]{TVireBook}, and they are equivalent to the set of moves considered in \cite[Def.\,3.13]{CRS1}, see Lemma~\ref{lem:PachnerFromBLT} below. 
We show that all admissible skeleta are related by admissible BLT moves.

\begin{definition} Let $M$ be a stratified 3-manifold. 
	\begin{enumerate}
		\item 
		The \textsl{unoriented BLT move B, L or T} is given by the two stratified open 3-balls shown in Figure \ref{fig:BLTMoves}\,\ref{item:BubbleMoves}, \ref{item:LuneMoves} or \ref{item:TriangleMoves} respectively, considered up to isomorphism.
		\item 
		An \textsl{oriented BLT move $X$} consists of
		the unoriented BLT move $X$ with a choice of orientation of the two respective stratified 3-balls, such that the orientations of strata intersecting the boundaries of the balls agree. 
		We call an oriented BLT move \textsl{admissible} if the two respective stratified 3-balls are admissibly oriented. 
		\item 
		An \textsl{application of an unoriented BLT move $X$ to $M$} is the stratified 3-manifold $X(M)$ which consists of replacing an embedding of the open stratified 3-ball $B$ on the left of the move~$X$ 
		in $M$ with the stratified 3-ball on the right of $X$. 
		Analogously we define the application of the inverse of a BLT move and the application of an oriented BLT move.
	\end{enumerate}
\end{definition}

We remark that 
-- at least in the local neighbourhood shown in Figure \ref{fig:BLTMoves}\,\ref{item:LuneMoves} --
an application of a lune move splits up an oriented 2-dimensional region into two parts that consequently will have the same orientations in the target. 
This in turn puts a restriction on the orientations of 2-strata for when an inverse lune move can be applied.

\medskip 

\begin{figure}
	\begin{enumerate}[label={(\roman*)}]
		\item 
		\label{item:BubbleMoves}
		The \textsl{bubble moves} $B$ and their inverses:
		\[
		\includegraphics[scale=1.0, valign=c]{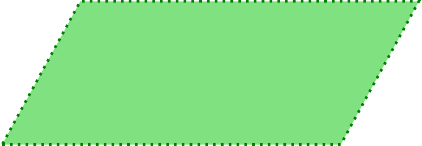} \stackrel[B^{-1}]{B}{\rightleftarrows}
		\includegraphics[scale=1.0, valign=c]{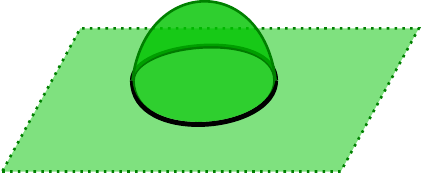}.
		\]
		\item 
		\label{item:LuneMoves}
		The \textsl{lune moves} $L$ and their inverses:
		\[
		\includegraphics[scale=1.0, valign=c]{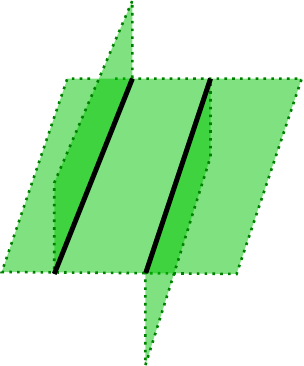} \stackrel[L^{-1}]{L}{\rightleftarrows}
		\includegraphics[scale=1.0, valign=c]{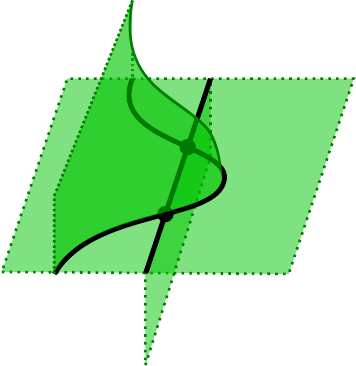}.
		\]
		\item 
		\label{item:TriangleMoves}
		The \textsl{triangle moves} $T$ and their inverses:
		\[
		\includegraphics[scale=1.0, valign=c]{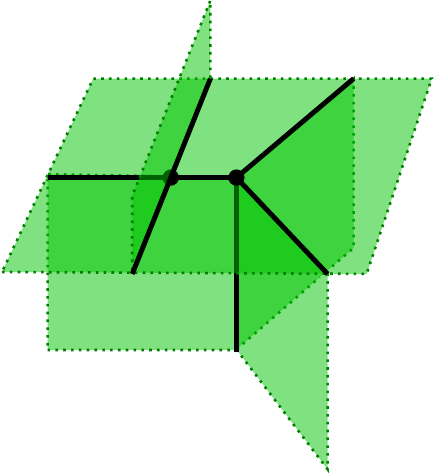} \stackrel[T^{-1}]{T}{\rightleftarrows}
		\includegraphics[scale=1.0, valign=c]{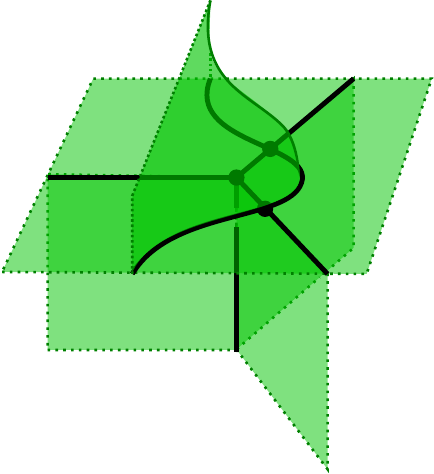}.
		\]    
	\end{enumerate}
	\caption{BLT moves without orientations.
	The dotted lines indicate where the 2-strata meet the boundary of the 3-ball in which they are embedded.
	}
	\label{fig:BLTMoves}
\end{figure}

In Theorem \ref{thm:ConnectSkeleta} below we show that any two admissible skeleta of $M$ that agree on $\del M$ can be transformed into one another by a finite sequence of admissible BLT moves.
Since for an admissible move we require the source and target to be admissibly oriented,  by Remark \ref{rem:ltovsorient}\,\ref{item:RemarkLocalOrder1} we only need to specify how the orientations of 2-strata are changed. 
One can check that this leaves us with a total of 32 admissible moves (up to isomorphisms of oriented stratified manifolds which do not necessarily fix the boundary): 
3 bubble moves, 9 lune moves and 20 triangle moves. 
Some examples of possible orientations are listed in Figure~\ref{fig:OrientedBLTMoves}.

We want to relate the BLT moves to two types of moves that are considered in \cite{CRS1}.
We call a set~$A$ of moves \textsl{stronger} than a set~$B$, if each of the moves of~$B$ is an application of a finite sequence of moves of~$A$. 
We say~$A$ is \textsl{equivalent} to~$B$ if~$A$ is stronger than~$B$ and~$B$ is stronger than~$A$.

We first consider the \textsl{admissible Pachner moves} which are the oriented Pachner moves whose source and targets are admissible triangulations, cf.\ Remark~\ref{rem:ltovsorient}\,\ref{item:RemarkLocalOrder2}, see also \cite{CRS1}.
Another set of moves that is considered in \cite[Def.\,3.13]{CRS1} are the \textsl{special orbifold data moves}, i.\,e.\ the moves relating the left- and right-hand sides of the identities in Figure~\ref{fig:SpecialOrbifoldData} below (without the $\A$-decorations).
Note that the latter consist of all 3 bubble moves, 6 of the 9 lune moves, and one triangle move. We consider all of these moves as moves between oriented skeleta. 

\begin{lemma}\label{lem:PachnerFromBLT}
	The admissible BLT moves are equivalent to the special orbifold data moves,	and both are stronger than the admissible Pachner moves. 
\end{lemma}

\begin{proof}
	By definition, the special orbifold data moves are a subset of the admissible BLT moves and consequently admissible BLT moves are stronger than special orbifold data moves. 
	In \cite[Prop.\,3.18]{CRS1} it is shown that the special orbifold data moves are stronger than the globally ordered Pachner moves. A slight modification of the arguments presented there shows the same for admissible Pachner moves.
	
	To verify that special orbifold moves are stronger than admissible BLT moves we observe that 
	\begin{itemize}
		\item 
		a T-move is the same as the 2-3 move in \eqrefO{1} (without its decoration), or one of the 19 variants thereof with a different admissible orientation; 
		\item 
		an L-move is the same as one of the six identities \eqrefO{2}--\eqrefO{7} in Figure~\ref{fig:SpecialOrbifoldData}, or one of the three variants thereof where the orientation of the additional 2-stratum on the right-hand sides of \eqrefO{4}--\eqrefO{6} is reversed; 
		\item 
		a B-move is the same as a bubble move in \eqrefO{8}. 
	\end{itemize}
	By \cite[Lem.\,3.15]{CRS1}, every admissible orientation for the 2-3 moves can be obtained from \eqrefO{1} and the six identities \eqrefO{2}--\eqrefO{7}.
	Moreover, the three L-moves which are not among \eqrefO{2}--\eqrefO{7} can be obtained by flipping the orientation of the new 2-stratum on the right-hand side: it follows from the proof of Lemma~\ref{lem:delsBLT} and Remark~\ref{rem:flip_orientation_flower} that this can be achieved only with T- and B-moves, and with the moves \eqrefO{2}--\eqrefO{7}. 
\end{proof}

\begin{figure}
\captionsetup[subfigure]{labelformat=empty}
\begin{subfigure}[b]{1.0\textwidth}
	\centering
	\includegraphics[scale=1.0, valign=c]{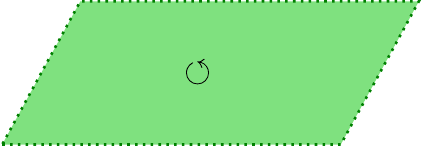} $\stackrel[B^{-1}]{B}{\rightleftarrows}$
	\includegraphics[scale=1.0, valign=c]{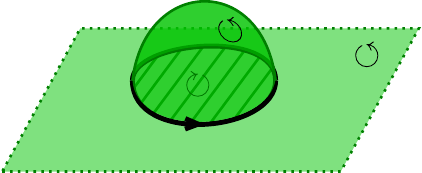},
	\caption{}
\end{subfigure}\\
\hspace{-25pt}
\begin{subfigure}[b]{1.1\textwidth}
\centering
\includegraphics[scale=1.0, valign=c]{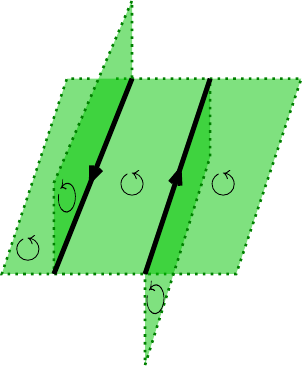} $\stackrel[L^{-1}]{L}{\rightleftarrows}$
\includegraphics[scale=1.0, valign=c]{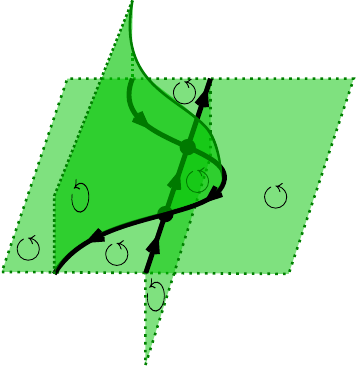}, 
\includegraphics[scale=1.0, valign=c]{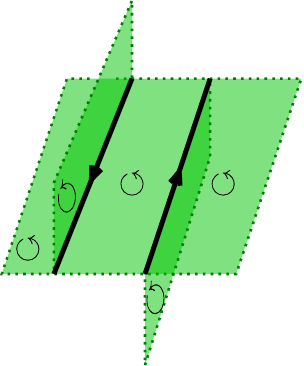} $\stackrel[L^{-1}]{L}{\rightleftarrows}$
\includegraphics[scale=1.0, valign=c]{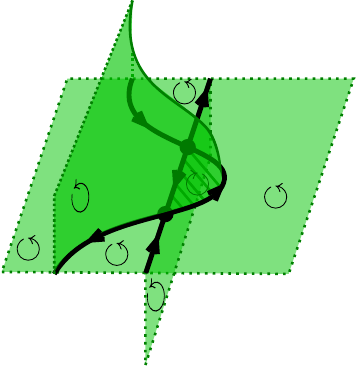},
\caption{}
\end{subfigure}\\
\begin{subfigure}[b]{1.0\textwidth}
	\centering
	\includegraphics[scale=1.0, valign=c]{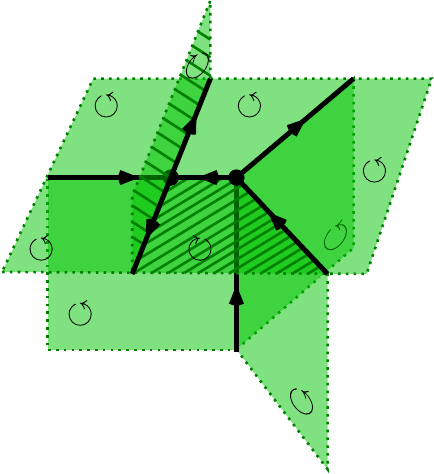} $\stackrel[T^{-1}]{T}{\rightleftarrows}$
	\includegraphics[scale=1.0, valign=c]{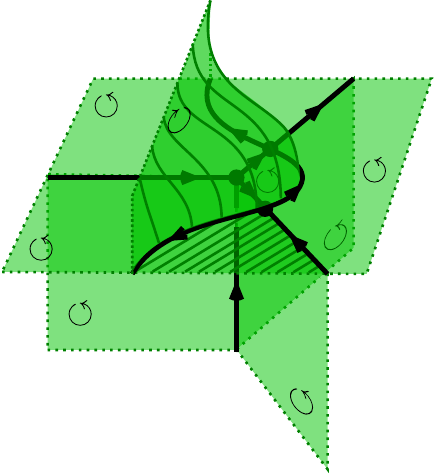},
	\caption{}
\end{subfigure}\\
\caption{Examples of BLT moves with orientations (of a total of 32). 
}
\label{fig:OrientedBLTMoves}
\end{figure}

We can now state the main result of this section.

\begin{theorem}
	\label{thm:ConnectSkeleta}
	Any two admissible skeleta that agree on $\del M$ are connected by a finite sequence of admissible BLT moves.
\end{theorem}

The strategy for the proof is as follows: First we show that any admissible skeleton can be refined to be dual to a triangulation. We then show that any two such skeleta are related by a sequence of Pachner moves that by Lemma~\ref{lem:PachnerFromBLT} can be obtained from admissible BLT moves.
Since we do not need the technical details of the proof of Theorem~\ref{thm:ConnectSkeleta}, which may be no surprise to the expert, we defer it to Appendix~\ref{app:A}. 
There it is explained how to make the construction of \cite{TVireBook} compatible with admissibility in every step.

\subsection{Bordisms with embedded ribbon graphs}
\label{sec:unlabelled-ribbon-bordisms}

In this section we review the category $\Bordribn{3}$ of 3-dimensional bordisms with embedded ribbon graphs. 
A variant of this category that includes labels from a modular fusion category~$\mathcal{C}$ is reviewed in Section~\ref{ssec:3dimGraphTQFT} below. 
For more details we refer to \cite[Sect.\,IV]{turaevbook}.

\medskip

A \textsl{ribbon punctured surface} is a compact oriented 2-dimensional manifold together with a finite set of \textsl{marked points} (or \textsl{punctures}) which are labelled by tuples of the form $(v,\varepsilon)$, where $v$ is a non-zero tangent vector, and $\varepsilon \in \{+1,-1\}$.
A diffeomorphism $\varphi\colon \Sigma \lra \Sigma'$ of ribbon punctured surfaces $\Sigma,\Sigma'$ is an orientation-preserving diffeomorphism of the underlying surfaces, mapping punctures to punctures such that if a puncture $p\in\Sigma$ is labelled by $(v,\varepsilon)$, then $\varphi(p)$ is labelled by $(d\varphi(v), \varepsilon)$.
The \textsl{orientation reversal} $-\Sigma$ of a punctured surface~$\Sigma$ is defined to be the ribbon punctured surface~$\Sigma$ except with opposite orientation, and if a puncture of~$\Sigma$ is labelled by~$(v,\varepsilon)$ then the corresponding puncture of $-\Sigma$ is labelled by $(v,-\varepsilon)$. 

By a \textsl{ribbon bordism} we mean a compact oriented $3$-dimensional bordism~$M$ together with an embedded ribbon graph~$R$, such that the loose strands of~$R$ meet $\partial M$ transversally. 
This induces the structure of a punctured surface on $\partial M$, whose punctures are in $\partial M \cap R$.
The punctures carry the labels $(v,\varepsilon)$, where~$v$ is the framing of the corresponding strand of~$R$ and $\varepsilon = +1$ if the strand is directed out of $\partial M$, and $-1$ otherwise.

Diffeomorphisms of ribbon bordisms by definition preserve embedded ribbon graphs, and their restrictions to the boundary are compatible with the parametrisation maps from the collared punctured surfaces to the boundary. 
Then the symmetric monoidal \textsl{category of ribbon bordisms} $\Bordribn{3}$ is obtained by a standard construction in analogy to the regular bordism category. 
Morphisms in $\Bordribn{3}$ are diffeomorphism classes of ribbon bordisms, but we usually will not make a notational distinction between these morphisms and their representatives.

\subsection[Ribbon diagrams and ${\omega}$-moves]{Ribbon diagrams and $\boldsymbol{\omega}$-moves}
\label{subsuc:RibbonDiagramsOmegaMoves}

Let~$M$ be a bordism, and let~$S$ be an admissible skeleton for~$M$.  
We adopt the nomenclature of \cite[Ch.\,14]{TVireBook}, to which we refer for a detailed discussion of the following notions. 

A \textsl{plexus} (Latin for ``braid'') is 
an (``abstract'', i.\,e.\ not embedded into $\R^n$) topological space which is 
made up of a finite number of oriented circles, oriented arcs, and coupons; arcs may meet only coupons, and only at pairwise distinct points at the coupon's top or bottom. 
Circles and arcs are collectively called \textsl{strands}. 

A \textsl{knotted plexus} in~$S$ (originally defined in \cite[Sect.\,14.1.2]{TVireBook}) is a local embedding~$\iota$ of a plexus~$d$ into~$S$, such that 
\begin{enumerate}
	\item 
	the coupons of~$d$ are embedded in~$S^{(2)}$ with their orientations preserved; 
	\item 
	if $\iota(d)$ has multiple points, then they are transversal double points of strands in~$S^{(2)}$ and they are (labelled as) either over-crossings or under-crossings; 
	\item 
	$\iota(d) \cap S^{(0)} = \varnothing$; 
	\item 
    $\iota(d) \cap \partial S$ consists only of endpoints of arcs in~$d$, and arcs meet $\partial S$ transversally; 
	\item 
	if a strand~$r$ of~$\iota(d)$ meets a 1-stratum~$L$ of~$S$ at a point~$w$, such that a neighbourhood of~$w$ is given by one of the four options
	\be 
		\hspace{-10pt}
		\includegraphics[scale=1.0, valign=c]{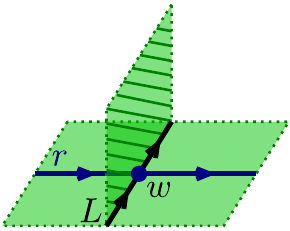}
		\includegraphics[scale=1.0, valign=c]{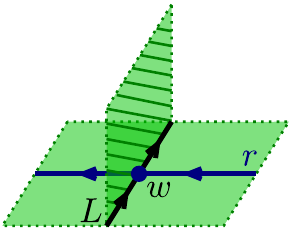} ~
		\includegraphics[scale=1.0, valign=c]{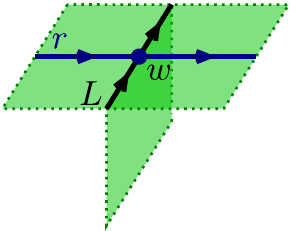}
		\includegraphics[scale=1.0, valign=c]{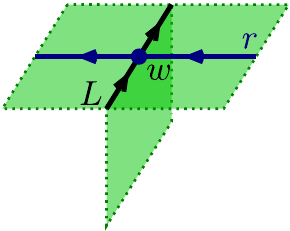}
		\, ,
	\ee 
	then the intersection point~$w$ is called a \textsl{positive switch}. 
	Without the above restriction on neighbourhoods, the intersection point is just called a \textsl{switch}.
\end{enumerate}
We usually refer to a knotted plexus $(d,\iota)$ simply by~$d$. 
Note that we depict strands in dark blue. 

The case of knotted plexi in admissible skeleta will be important for us: 
\begin{definition}
  	\label{def:AdmissibleRibbonDiagram1}
	Let~$M$ be a bordism. 
		An \textsl{admissible ribbon diagram} in~$M$ is a pair $(S,d)$, where~$S$ is an admissible skeleton of~$M$, and~$d$ is a knotted plexus in~$S$.
		An admissible ribbon diagram is \textsl{positive} if each of its switches is positive. 
\end{definition}

Recall that our starting point is an (unstratified) bordism~$M$, hence a smooth 3-manifold, and it makes sense to talk about tangent vectors at all points of $M$. 
To express the relation between ribbon graphs and embedded plexi, we will need in addition the notion of transversality on strata. For this reason, we will make the
\begin{quote}
\textbf{Assumption:}  
	In ribbon diagrams $(S,d)$ in~$M$, all strata of the skeleton~$S$ are smooth submanifolds of~$M$. (The image $\iota(d)$ is not required to be a smooth submanifold.)
\end{quote}
We can now define: 
A \textsl{framing}~$f$ of a positive admissible ribbon diagram is a function that continuously assigns a
direction $f(x)$ at~$\iota(x)$ in~$M$ to each $x\in d$ (hence double points $\iota(x) = \iota(y)$ for $x\neq y$ in~$d$ can have two different directions), 
such that 
(i) if~$x$ lies in a stratum~$t$, then $f(x)$ is transverse to~$t$, and for 2-strata~$t$, the orientation of~$t$ followed by the direction $f(x)$ agrees with the orientation of~$M$; 
(ii) if~$x$ lies in a coupon~$c$, then $f(x)$ is transverse to~$c$;
(iii) if $x\in\partial M$, then $f(x)$ is tangent to~$\partial M$. 
If an admissible ribbon diagram $(S,d)$ is positive, there exists a framing for~$d$, and any two framings that agree on $\partial M$ are isotopic relative to the boundary.

\medskip

From a positive admissible ribbon diagram $(S,d)$ we obtain a ribbon graph~$d^f$ in~$M$ as follows: 
Pick a framing~$f$ of~$d$ and slightly push the over-crossing strands of~$d$ at crossings in the direction of the framing~$f$, and then use~$f$ to provide the resulting graph with a ribbon structure. 
Two framings that agree on $\partial M$ give isotopic ribbon graphs. 

Conversely, we say that a positive admissible ribbon diagram $(S,d)$ \textsl{represents} a ribbon graph $R$ in $M$, if $R$ is isotopic to $d^f$.
The set of positive admissible ribbon diagrams in~$M$ that represent~$R$ is denoted $\mathscr S(M,R)$.  In the case $R=\varnothing$ this reduces to the set of admissible skeleta $\mathscr S(M)$ (cf.\ Definition~\ref{def:admissible-skeleton}). 

\begin{figure}
	\centering
	\begin{subfigure}[b]{0.2\textwidth}
		\centering
		\includegraphics[scale=1.2, valign=c]{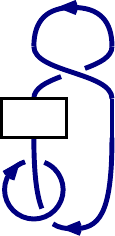}
		\caption{}
	\end{subfigure}
	\begin{subfigure}[b]{0.4\textwidth}
		\centering
		\includegraphics[scale=1.2, valign=c]{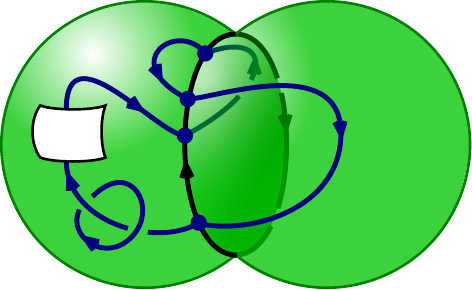}
		\caption{}
	\end{subfigure}
	\caption{A ribbon graph in $S^3$ (a) and an example of an admissible ribbon diagram representing it (b)}
	\label{fig:RibbonDiagram}
\end{figure}

\begin{remark}
The notion of ``positive admissible ribbon diagram'' is that of a ``positive ribbon diagram'' in the sense of \cite[Sect.\,14.2]{TVireBook}, but with local neighbourhoods as in Figure~\ref{fig:skeleta} and such that the orientations of 2-strata can be extended to an admissible orientation in the sense of Section~\ref{subsubsec:AdmissibleSkeleta}; if an admissible choice of orientations exists, then it is unique, cf.\ Remark~\ref{rem:ltovsorient}\,\ref{item:RemarkLocalOrder2}. 

We stress that a ``ribbon diagram'' (whether it is admissible, positive, or plain) always relates to a prescribed skeleton. 
Hence a ribbon diagram is not just an (embedded) string diagram, even though the phrase might suggest otherwise. 
\end{remark}

It is shown in \cite[Lem.\,14.1]{TVireBook} that every ribbon graph is representable by a positive ribbon diagram. 
The analogous result in our framework can be proven similarly: 

\begin{lemma}
	Every ribbon graph~$R$ in a bordism~$M$ can be represented by a positive admissible ribbon diagram, i.\,e.\ $\mathscr S(M,R) \neq \varnothing$. 
\end{lemma}
\begin{proof}
	We sketch the proof of \cite[Lem.\,14.1]{TVireBook} and point out how to adapt it for our purposes along the way. 
	
	Pick a tubular closed neighbourhood~$U$ of~$R$ in~$M$. 
	Pick a triangulation~$t$ of $M\setminus U^\circ$ and a total order on the vertices of~$t$, such that the induced admissible orientation on the dual~$t^*$ satisfies the condition that all 2-strata in $\partial U = U \setminus U^\circ$ are oriented by the normal pointing out of~$U$. 
	Next push~$R$ along its framing into $\partial U$, such that no coupon of the resulting knotted plexus~$d$ intersects a 0- or a 1-stratum of $t^*\cap\partial U$, no strand of~$d$ meets a 0-stratum of $t^*\cap\partial U$, and strands meet 1-strata of $t^*\cap\partial U$ only transversally. 
	
	Pick enough open meridional discs~$D_i$ of~$U$ such that their boundaries $\partial\overline D_i$ do not intersect coupons and 0-strata in $\partial U$, and which intersect strands and 1-strata in $\partial U$ only transversally, such that the complement of $\bigcup_i D_i$ in~$U$ is a disjoint union of 3-balls. 
	By declaring the boundaries $\partial\overline D_i$ to be new 1-strata and orienting them arbitrarily, 
	this lifts~$t^*$ to an admissible skeleton~$S$ of~$M$ (by adding the 2-strata~$D_i$ as well as the 1- and 0-strata in $\partial\overline D_i$, where 0-strata are intersection points with 1-strata of $t^* \cap \partial U$).
	In doing so, we endow every new disc-shaped 2-stratum~$D_i$ in~$U^\circ$ with the orientation dictated by that of $\partial\overline D_i$ and the orientations of the 2-strata in $\partial U$ adjacent to the meridian.
	Since every 2-stratum in $\partial U$ is oriented by the normal pointing out of~$U$, and since $d\subset \partial U$, every switch of the admissible ribbon diagram $(S,d)$ is positive. 
	Thus by construction, the positive admissible ribbon diagram $(S,d)$ represents the ribbon graph~$R$ in~$M$.
\end{proof}

\begin{figure}
	\captionsetup[subfigure]{labelformat=empty}
	\centering
	\begin{subfigure}[b]{0.48\textwidth}
		\centering
		\includegraphics[scale=1.0, valign=c]{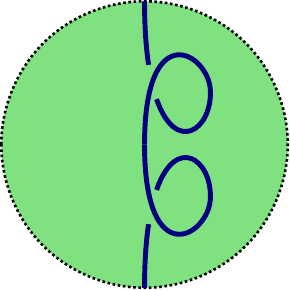} $\xrightarrow{\omega_1}$
		\includegraphics[scale=1.0, valign=c]{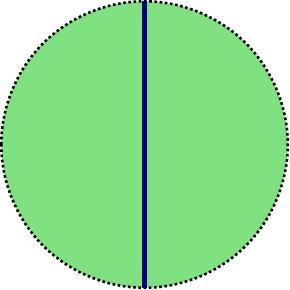}
		\caption{}
	\end{subfigure}
	\begin{subfigure}[b]{0.48\textwidth}
		\centering
		\includegraphics[scale=1.0, valign=c]{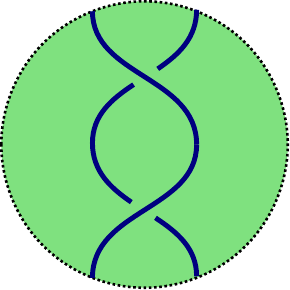} $\xrightarrow{\omega_2}$
		\includegraphics[scale=1.0, valign=c]{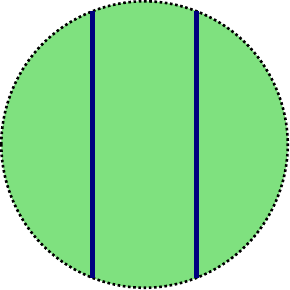}
		\caption{}
	\end{subfigure}\\
	\begin{subfigure}[b]{0.48\textwidth}
		\centering
		\includegraphics[scale=1.0, valign=c]{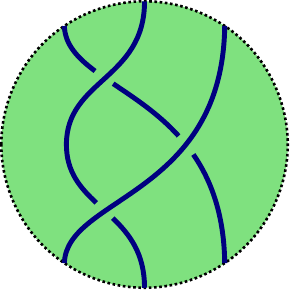} $\xrightarrow{\omega_3}$
		\includegraphics[scale=1.0, valign=c]{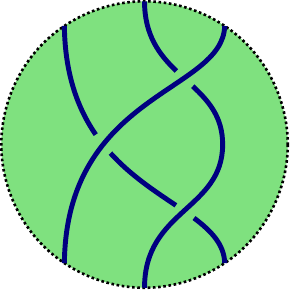}
		\caption{}
	\end{subfigure}
	\begin{subfigure}[b]{0.48\textwidth}
		\centering
		\includegraphics[scale=1.0, valign=c]{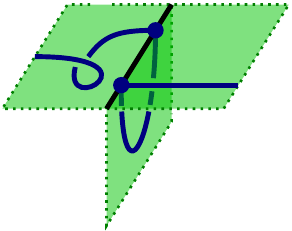} $\xrightarrow{\omega_4}$
		\includegraphics[scale=1.0, valign=c]{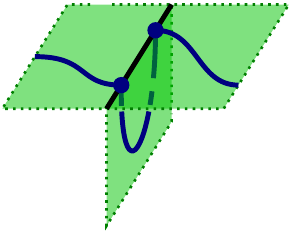}
		\caption{}
	\end{subfigure}\\
	\begin{subfigure}[b]{0.48\textwidth}
		\centering
		\includegraphics[scale=1.0, valign=c]{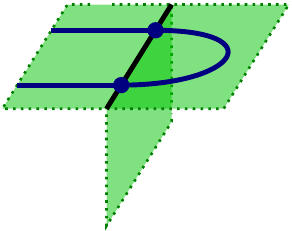} $\xrightarrow{\omega_5}$
		\includegraphics[scale=1.0, valign=c]{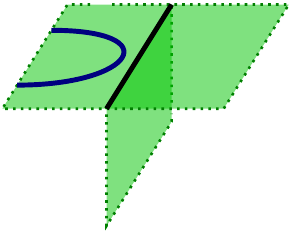}
		\caption{}
	\end{subfigure}
	\begin{subfigure}[b]{0.48\textwidth}
		\centering
		\includegraphics[scale=1.0, valign=c]{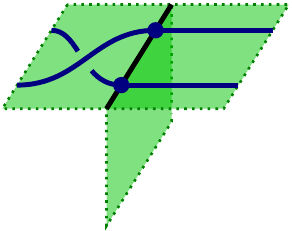} $\xrightarrow{\omega_6}$
		\includegraphics[scale=1.0, valign=c]{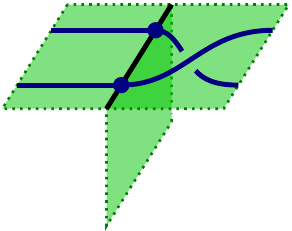}
		\caption{}
	\end{subfigure}\\
	\begin{subfigure}[b]{0.48\textwidth}
		\centering
		\includegraphics[scale=1.0, valign=c]{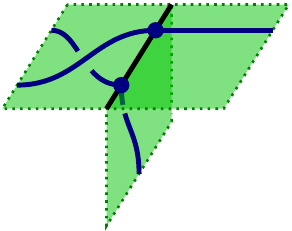} $\xrightarrow{\omega_7}$
		\includegraphics[scale=1.0, valign=c]{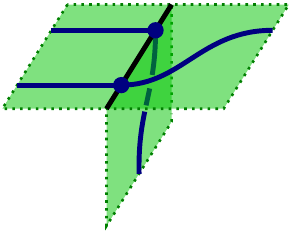}
		\caption{}
	\end{subfigure}
	\begin{subfigure}[b]{0.48\textwidth}
		\centering
		\includegraphics[scale=1.0, valign=c]{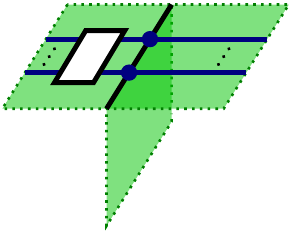} $\xrightarrow{\omega_8}$
		\includegraphics[scale=1.0, valign=c]{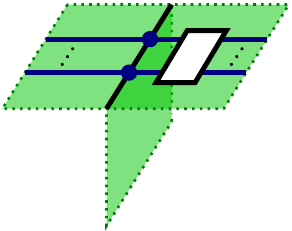}
		\caption{}
	\end{subfigure}\\
	\begin{subfigure}[b]{0.64\textwidth}
		\centering
		\includegraphics[scale=1.0, valign=c]{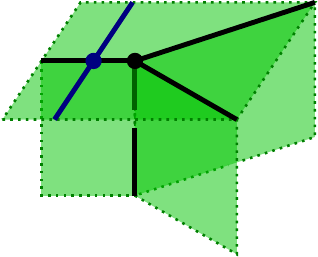} $\xrightarrow{\omega_9}$
		\includegraphics[scale=1.0, valign=c]{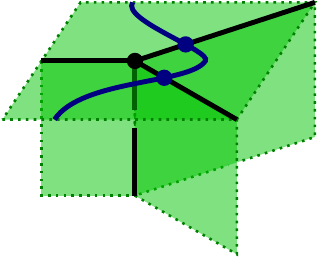}
		\caption{}
	\end{subfigure}\\
	\begin{subfigure}[b]{0.64\textwidth}
		\centering
		\includegraphics[scale=1.0, valign=c]{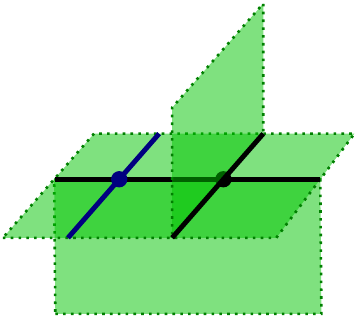} $\xrightarrow{\omega_{10}}$
		\includegraphics[scale=1.0, valign=c]{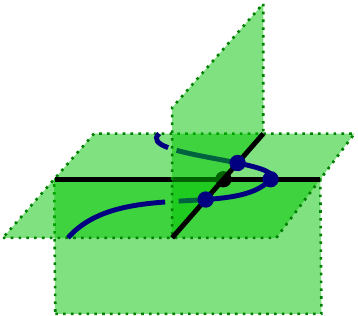}
		\caption{}
	\end{subfigure}
	\caption{$\omega$-moves}
	\label{fig:omegaMoves}
\end{figure}

By Theorem~\ref{thm:ConnectSkeleta}, any two admissible skeleta of a given bordism~$M$, i.\,e.\ any two elements of $\mathscr S(M)$, are related by a finite sequence of admissible BLT moves. 
Similarly, for a ribbon graph~$R$ in~$M$, any two positive admissible ribbon diagrams representing~$R$, i.\,e.\ any two elements of $\mathscr S(M,R)$, are related by a finite sequence of local moves between positive ribbon diagrams, namely those of type BLT or of type $\omega_1, \dots, \omega_{10}$ as in Figure~\ref{fig:omegaMoves}, and their inverses. 
In these pictures the knotted plexi have orientations that agree on both sides of every move, but orientations are not depicted. 
Neither are the orientations of the strata of the admissible skeleta in Figure~\ref{fig:omegaMoves} depicted. 
This means that there is one $\omega$-move for each choice of (admissible) orientation; e.\,g.\ there are $2\cdot 2 = 4$ moves of type~$\omega_4$.\footnote{To see this, pick any orientation of the leftmost 2-stratum in the figure (two choices), by positivity of the switches this fixes the orientations of the remaining two 2-strata on each side; the orientation of the 1-stratum is then determined by admissibility; pick any orientation of the strand (two choices).}

We collectively refer to BLT moves as \textsl{$\omega_0$-moves}, and we observe that the moves of types $\omega_1,\omega_2,\omega_3$ are framed Reidemeister moves. 
The moves $\omega_4, \omega_5, \dots, \omega_{10}$ appear in \cite[Sect.\,14.3]{TVireBook} (where those corresponding to our moves~$\omega_9$ and~$\omega_{10}$ are denoted $\omega_{9,0,1}$ and $\omega_{9,1,0}$, respectively) for the case of skeleta whose 0- and 1-strata are unoriented. 
Moreover, any two such positive ribbon diagrams representing the same ribbon graph in a given bordism are related by $\omega$-moves \cite[Lem.\,14.2\,\&\,Thm.\,14.4]{TVireBook}. 
The analogous result is proven similarly in our setting: 

\begin{proposition}
	\label{prop:omegaMoves}
	Let~$R$ be a ribbon graph in a bordism~$M$. 
	Any two elements in $\mathscr S(M,R)$ that agree on $\partial M$ are related by a finite sequence of moves of type $\omega_0, \omega_1, \dots, \omega_{10}$. 
\end{proposition}
\begin{proof}	
	Only the moves of type~$\omega_0$, which by definition are admissible BLT moves, change the underlying admissible skeleta, while not affecting the knotted plexi of ribbon diagrams. 
	Hence when restricting to ribbon diagrams which differ only away from their knotted plexi, the statement follows from Theorem~\ref{thm:ConnectSkeleta}. 
	
	Recall that in \cite{TVireBook}, the notion of skeleton comes with orientations for 2-strata, while 0- and 1-strata do not carry orientations (contrary to our setting). 
	Accordingly, the original variant of $\omega$-moves in \cite[(14.1)]{TVireBook}, to which we refer here as $\omega^{\textrm{TV}}$-moves, is between positive ribbon diagrams without orientations for 0- and 1-strata.
	In Sections 14.4--14.7 of loc.\ cit., it is shown that any two positive ribbon diagrams representing $(M,R)$ are related by $\omega^{\textrm{TV}}$-moves. 
	Note that elements of $\mathscr S(M,R)$ are positive admissible ribbon diagrams.
	Moreover, ``our'' $\omega$-moves in Figure~\ref{fig:omegaMoves} are $\omega^{\textrm{TV}}$-moves between positive ribbon diagrams which are endowed with an admissible orientation for all strata. 
	
	It follows that Proposition~\ref{prop:omegaMoves} holds if in the proofs of \cite{TVireBook}, we can restrict to $\omega^{\textrm{TV}}$-moves which lift to $\omega$-moves. 
	This is indeed the case: whenever a new 2-stratum appears in the construction of \cite[Sect.\,14.4--14.7]{TVireBook} (i.\,e.\ when ``attaching a bubble'', cf.\ Lemma~14.7 and Figure~14.13 of loc.\ cit.), there is a choice of orientation for this 2-stratum, and upon close inspection we notice that one of these choices is compatible with a (unique) choice of orientations for the new 0- and 1-strata which makes the entire positive ribbon graph admissibly oriented. 
\end{proof}

\section{Defect TQFTs}
\label{sec:DefectTQFTs}

In this section we first review the notions of 3-dimensional defect bordisms and defect TQFTs from \cite{CMS,CRS1}, and that to every defect TQFT~$\zz$ there is a naturally associated 3-category~$\tz$. 
After a brief reminder on coloured ribbon bordisms and graph TQFTs, we then present a construction that produces new line defect labels from~$\tz$, which can be thought of as a completion procedure on defect data. 
This will be important in Section~\ref{subsec:RibbonCategoriesFromSOD}, where we will construct a canonical ribbon category from orbifold data and completed defect data.

\subsection{Review of 3-dimensional defect TQFT}
\label{subsec:Review3dDefectTQFT}

A 3-dimensional defect TQFT is by definition a symmetric monoidal functor $\zz \colon \Bordd[3] \longrightarrow \Vect$, where $\D$ are 3-dimensional defect 
data, and $\Bordd[3]$ is the symmetric monoidal category of 3-dimensional defect bordisms decorated with defect data $\D$. 
We start by recalling the relevant definitions.

\subsubsection{Defect data and defect bordisms}

A list of \textsl{3-dimensional defect data $\D=(D_{3},D_{2},D_{1},s,t,f)$} consists of \cite[Def.\,2.6]{CMS}
\begin{enumerate}
  \item 
  three sets $D_{3},D_{2},D_{1}$,
  \item  
  \textsl{source} and \textsl{target} maps $s,t \colon D_{2} \longrightarrow D_{3}$,
  \item 
  and a \textsl{folding} map $f \colon D_{1} \longrightarrow [\Sphere_{1}(\D)]$.
\end{enumerate}
Here $\Sphere_{1}(\D)$ is the set  of all \textsl{defect circles}: an element $S \in \Sphere_{1}(\D)$  is a  stratified oriented circle $S^{1}$ whose 1-strata are decorated with elements in $D_{3}$, and whose $0$-strata are decorated with pairs $(\alpha, \pm)$, $\alpha \in D_{2}$, subject to the condition that the 1-strata oriented away from (resp.\ towards) an $(\alpha,+)$-decorated 0-stratum are decorated by $t(\alpha)$ (resp.\ $s(\alpha)$), while for an $(\alpha,-)$-decorated 0-stratum the decorations are swapped. 
The bracket around $\Sphere_{1}(\D)$ signals the set of equivalence classes of such stratified decorated circles, where $S$ and $S'$ are equivalent if they are related by a decoration-preserving isomorphism of stratified manifolds. 
Thus the remaining information of a class $[S] \in [\Sphere_{1}(\D)]$ is just the cyclic set of compatible decorations on the 0-strata, which is the point of view taken in~\cite[Def.\,2.6]{CMS}. 
We extend the map~$f$ to $f \colon  D_{1} \times \{\pm\} \longrightarrow [\Sphere_{1}(\D)]$ by setting $f((x,+))=f(x)$ and $f((x,-))= f(x)^{\mathrm{rev}}$ for $x \in D_1$, where the reverse of a defect circle is the defect circle with the orientation of all strata reversed. 

For 3-dimensional defect data $\D$, there is a symmetric monoidal category \textsl{$\Bordd[3]$ of 3-dimensional decorated defect bordisms}, see \cite[Def.\,2.4]{CRS1}. 
By definition, a morphism in $\Bordd[3]$ is a morphism in $\Borddefn{3}$ (cf.\ Section~\ref{subsubsec:DefectBordisms}) together with a decoration by $\D$: each $j$-stratum is labelled with an element of $D_{j}$ for $j\in\{1,2,3\}$, such that the decoration is compatible with the maps $s,t,f$, namely that the 3-strata adjacent to an $\alpha$-labelled 2-stratum are labelled by $s(\alpha)$ and $t(\alpha)$, and the labels of 2-strata adjacent to a 1-stratum~$L$ are read off of $f(L)$. 
Similarly, objects $\Bordd[3]$ are objects of $\Borddefn{3}$ together with a label in $D_{k+1}$ for each $k$-stratum, $k\in\{0,1,2\}$, and the decorations induced at the boundary of a morphism in $\Bordd[3]$ must match with the decorations of the source and target objects. 

\begin{definition}
\label{def:DefectTQFT}
A \textsl{3-dimensional defect TQFT with defect data $\D$} is a symmetric monoidal functor
\be 
\zz \colon \Bordd[3] \lra \Vect \, . 
\ee 
\end{definition}

Examples of defect TQFTs can be obtained from anomaly-free modular fusion categories $\CC$: 
As explained in \cite{CRS2}, the Reshetikhin--Turaev TQFT associated to~$\CC$ lifts to a defect TQFT 
\be 
\label{eq:RTdefectTQFT}
\zzc \colon \Borddefn{3}(\D^\CC) \lra \Vect \, , 
\ee 
where~$D^\CC_2$ consists of $\Delta$-separable symmetric Frobenius algebras in~$\CC$, and~$D^\CC_1$ consists of certain multi-modules. 
(If~$\CC$ does have an anomaly, Reshetikhin--Turaev theory is instead defined on an ``extended'' defect bordism category $\Borddefen{3}(\D^\CC)$, see \cite{CRS2} for details.)

\subsubsection{The 3-category associated to a defect TQFT}
\label{subsubsec:3catForDefectTQFT}

For each defect TQFT $\zz \colon \Bordd[3] \lra \Vect$, there is an associated 3-category $\tz$, see \cite[Sect.\,3.3--3.4]{CMS}. 
More precisely, $\tz$ is a ``Gray category with duals'', which is analogous to the fact that every 2-dimensional defect TQFT gives rise to a pivotal 2-category as explained in \cite{DKR}. 

We refer to \cite{CMS} for the detailed construction of~$\tz$ as well as all relevant definitions. 
For our purposes here it suffices to recall that 
\begin{itemize}
	\item 
	objects of~$\tz$ are elements of~$D_3$, pictured as labelling an (oriented yet otherwise structure-less) 3-cube; 
	\item 
	1-morphisms are (equivalence classes of) parallel $D_2$-labelled planes inside a 3-cube;
	\item 
	2-morphisms are (equivalence classes of) decorated stratified 3-cubes~$X$ that are cylinders over string diagrams of the pivotal pre-2-category freely generated by~$\D$, such that $j$-strata of~$X$ are labelled by elements of~$D_j$;
	\item 
	3-morphisms are elements of the vector spaces that~$\zz$ assigns to defect spheres. 
\end{itemize}

For illustration, note that
\be
\alpha = \includegraphics[scale=1.0, valign=c]{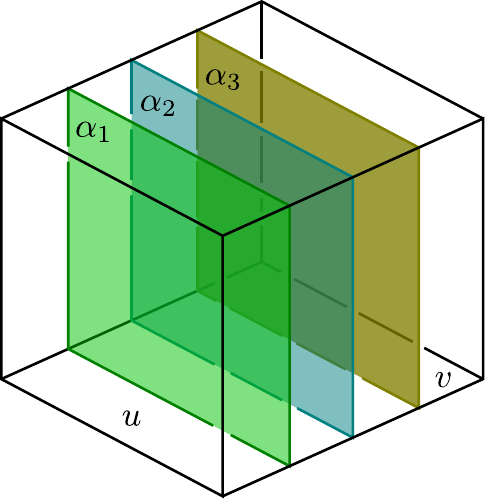}
~\text{ and }~
X = \includegraphics[scale=1.0, valign=c]{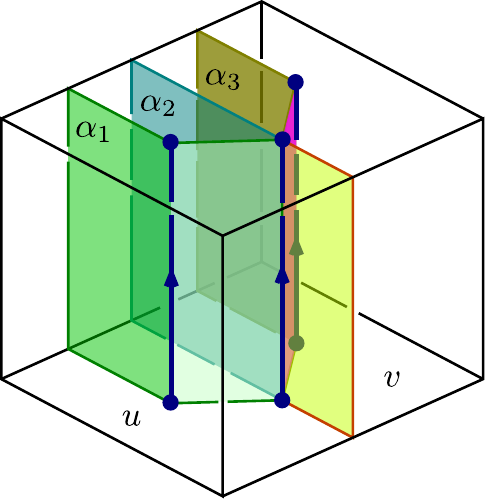}
\ee
represent a 1-morphism $\alpha\colon u \lra v$ and a 2-morphism~$X$ with $t(X) = \alpha$, respectively, and only parts of the decorations are shown. 
Note that contrary to generic $(\infty,3)$-categories, 3-morphisms in the 3-category~$\tz$ form vector spaces which do not carry any further homotopical information.

\subsection{Review of 3-dimensional graph TQFT} 
\label{ssec:3dimGraphTQFT}

The relation between unlabelled and labelled ribbon bordism categories is analogous to the relation between the defect bordism categories $\Borddef_3$ and $\Bordd[3]$. 
Indeed, recall from Section~\ref{sec:unlabelled-ribbon-bordisms} the unlabelled ribbon bordism category $\Bordribn{3}$, and  let~$\CC$ be a $\Bbbk$-linear \textsl{ribbon category} for a field $\Bbbk$, i.\,e.\ a $\Bbbk$-linear braided pivotal category whose left and right twists coincide, see e.\,g.\ \cite[Sect.\,3.3]{TVireBook}. 
Objects of the labelled ribbon bordism category $\Bordribn{3}(\CC)$ are objects $\Sigma \in \Bordribn{3}$ together with a label $X_i\in\CC$ for every marked point~$p_i$ of~$\Sigma$. 
Morphisms in $\Bordribn{3}(\CC)$ are morphisms $(M,R)$ in $\Bordribn{3}$ as in Section~\ref{sec:unlabelled-ribbon-bordisms}, where in addition each strand and coupon of the ribbon graph~$R$ is (compatibly) labelled with an object and morphism in~$\CC$, respectively. 

We will consistently use calligraphic Roman letters for such $\CC$-coloured ribbon graphs~$\mathcal R$, and non-calligraphic letters for the underlying ribbon graphs~$R$. 
For more details we refer to \cite[Sect.\,15.2.1]{TVireBook}, where $\Bordribn{3}(\CC)$ is denoted $\textrm{Cob}_3^\CC$. 

\begin{definition}
A \textsl{graph TQFT} over a $\Bbbk$-linear ribbon category~$\CC$ is a symmetric monoidal functor 
\be 
\Bordribn{3}(\CC) \lra \Vect \, ,
\ee 
where $\Vect$ denotes the symmetric monoidal category $\Bbbk$-vector spaces. 
\end{definition}

\begin{remark}
\label{rem:RTandD0}
\begin{enumerate}[label={(\roman*)}]
\item 
For an anomaly-free modular fusion category~$\CC$, the Resheti\-khin--Turaev construction \cite{turaevbook} produces a graph TQFT 
\be 
	\label{eq:RTgraphTQFT}
	\zz^{\textrm{RT},\CC} \colon \Bordribn{3}(\CC) \lra \Vect \, . 
\ee 
In \cite{CMRSS2} we will apply the results of the present paper to combine~\eqref{eq:RTgraphTQFT} and the defect TQFT~\eqref{eq:RTdefectTQFT} to construct an ``orbifold graph TQFT'' (as introduced in Section~\ref{subsec:OrbifoldGraphTQFTs} below), including the case of anomalous~$\CC$.  
\item 
\label{item:D0completion}
Recall from \cite[Sect.\,2.4]{CRS1} that the $D_0$-completion $\zz^\bullet \colon \Borddef_3(\D^\bullet) \lra \Vect$ is a canonical extension for any defect $\zz \colon \Bordd[3] \lra \Vect$ to a symmetric monoidal functor on defect bordisms that may also have decorated 0-strata in addition to $j$-strata for $j\geqslant 1$. 
Just as in the case without 0-strata (Definition~\ref{def:DefectTQFT}), we refer to~$\zz^\bullet$ as a defect TQFT.
\end{enumerate}
\end{remark}

\subsection{The line completion for defect TQFTs}
\label{subsec:MonoidalCatFromDefectTQFT}

Given a 3-dimensional defect TQFT $\zz$, we would like to be able to use the $k$-morphisms of the associated 3-category $\tz$ as $(3-k)$-dimensional defects. 
Here we will describe a completion of the defect data~$\D$ which implements this for $k=2$.

The intuitive picture is as follows: a morphism $g$ in $\tz$ is a certain stratified cube; if a stratum in a bordism is decorated by $g$ one would like to ``replace the corresponding stratum by the cube representing $g$''. 
However, the local neighbourhoods around the strata in a bordism are modelled by spheres and their 
cones and cylinders. 
Hence to avoid making additional choices, in this section we first ``complete'' the defect data~$\D$ of~$\zz$ to new defect data $\widehat{\D}$ which have additional line defect labels modelled on defect discs; the line defect label $X \in \widehat{D}_{1}$ in~\eqref{eq:LineDefectWithInternalStructure} is an example of such an additional label. 
In a second step we will then see how the label set $\widehat{D}_{1}$ indeed corresponds to certain 2-morphisms of~$\tz$. 
	
The purpose of this section is to make precise the idea of ``tensoring line defect labels''; this is a prerequisite of the construction (in Section~\ref{subsec:RibbonCategoriesFromSOD} below) of the ribbon categories~$\wa$ attached to~$\zz$. 

\medskip 

We start by defining \textsl{decorated defect 2-manifolds} (possibly with boundary) as stratified 2-manifolds with local neighbourhoods in $\mathcal{N}_{2}$ for points in the interior,  whose 2-strata are decorated by $D_{3}$, 1-strata are decorated by $D_{2}$, and 0-strata by $D_{1}$, such that the local neighbourhoods are compatible with the decorations as allowed by the maps $s,t,f$ of~$\D$. 
We denote by $\Disc(\D)$ (resp.\ $\Sphere_{2}(\D)$) the $\D$-decorated stratified 2-manifolds with underlying manifold being discs (resp.\ spheres). 
For example, we have  
\be 
\includegraphics[scale=1.0, valign=c]{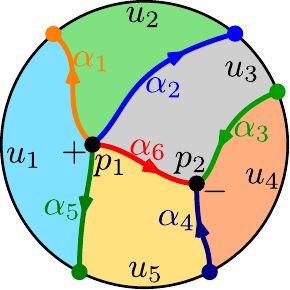} 
\in \Disc(\D)
\, , \qquad 
\includegraphics[scale=1.0, valign=c]{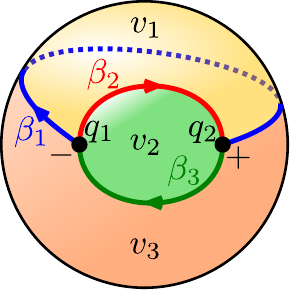} 
\in \Sphere_{2}(\D) \, , 
\ee 
where $u_i, v_j\in D_3$, $\alpha_i, \beta_j \in D_2$, $p_i, q_j\in D_1$ with adjacency of the strata as illustrated. 
Note that closed decorated defect 2-manifolds are precisely the objects of $\Borddef_3(\D)$. 

In particular we have a map $C\colon D_{1} \times \{\pm\}\longrightarrow [\Disc(\D)]$, mapping $(x,\varepsilon) \in  D_{1} \times \{\pm\}$ to the equivalence class of the cone $[Cf((x,\varepsilon))]$ with the 0-stratum corresponding to the cone point decorated by~$x$, for example
\be
\label{eq:starlikedisc}
C(x,+) = \includegraphics[scale=1.0, valign=c]{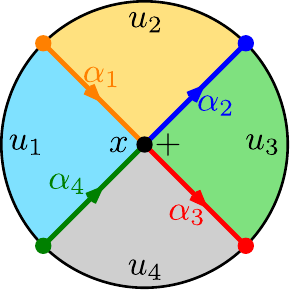} \, , 
\ee
where $\alpha_i \in D_2$, $u_i \in D_3$, and the adjacent strata of $x\in D_1$ are as indicated.

\begin{definition}
Let $\D$ be a list of 3-dimensional defect data. 
The \textsl{line defect completion of~$\D$} is the list of 3-dimensional defect data~$\widehat{\D}$ consisting of
\begin{enumerate}
	\item 
	the sets $\widehat{D}_{3}=D_{3}$ and $\widehat{D}_{2}=D_{2}$ with the maps $s,t$ from $\D$,
	\item 
    the set $\widehat{D}_{1}=\Disc(\D)$,
	\item 
	the map $\widehat{f}=\partial\colon \widehat{D}_{1} \longrightarrow [\Sphere_{1}(\D)] $ which assigns to $X\in \widehat{D}_{1}$ the isomorphism class represented by the boundary of~$X$.
\end{enumerate}
\end{definition}

Note that here we do not consider isomorphism classes for the elements in $\widehat{D}_{1}$.

\subsubsection{The line defect completion of a defect TQFT}

For a given defect TQFT  $\zz \colon \Bordd[3] \longrightarrow \Vect$ we will define a ``line defect completed'' defect TQFT
$\widehat{\zz} \colon \Borddefn{3}(\widehat{\D}) \longrightarrow \Vect$ by defining a symmetric monoidal \textsl{insertion functor} 
\be 
\In \colon \Borddefn{3}(\widehat{\D}) \longrightarrow \Borddefn{3}(\D) \, . 
\ee 
First, by shrinking the local neighbourhoods in the definition of a defect manifold, we can assume that these specify for every object $\Sigma \in \Borddefn{3}(\widehat{\D})$ for each 0-stratum $p\in\Sigma$ a closed neighbourhood $N_p$ and an isomorphism $\varphi^{\Sigma}_p \colon N_p \longrightarrow CS_{p}$, where~$S_{p}$ is the boundary of the specified local neighbourhood (i.\,e.\ of the specified element in~$\mathcal{N}_{2}$, see~\eqref{eq:N2-neighbourhoods}) at~$p$, and $CS_{p}$ is its cone.

Similarly, for a morphism~$M$ in $\Borddefn{3}(\widehat{\D})$ we can choose for each 1-stratum~$L$ of~$M$ with corresponding defect circle~$S_{L}$ a tubular closed neighbourhood~$N_L$ with a specified isomorphism
$\varphi_L$, which is either 
\be 
\varphi_L \colon N_L \longrightarrow CS_{L} \times [0,1]
\quad \text{ or } \quad 
\varphi_L \colon N_L \longrightarrow CS_{L} \times S^1
\, , 
\ee 
depending on whether $L$ meets the boundary of $M$ or not. 
In the first case the neighbourhood $N_L$ is required to restrict to the already chosen neighbourhood of the corresponding 0-stratum on the boundary.

Now we define for an object $\Sigma \in \Borddefn{3}(\widehat{\D})$ the object 
 \begin{equation}
  \label{eq:repl-bdy}
  \In(\Sigma) = \Big( \Sigma \setminus \bigcup_{p \in \Sigma_0}N_p\Big) \cup_{\varphi_p} X_p
\end{equation}
of $\Borddefn{3}(\D)$, where $p$ runs over all 0-strata of $\Sigma$ and $X_p \in \widehat{D}_{1}=\Disc(\D)$ is 
the decoration at the 0-stratum $p$.
That is, we remove the neighbourhoods~$N_p$ and glue in the discs $X_p$ instead. 
Different choices of neighbourhoods lead to isomorphic functors $\In$, here we fix one such choice for each $\Sigma$.

To define $\In$ on morphisms in the case where~$L$ with decoration $X_{L}$ meets the boundary of~$M$, we set
\begin{equation}
  \label{eq:repl-L}
  \In_L(M) = \big( M \setminus N_L \big) \cup_{\varphi_L} \big(X_{L} \times [0,1]\big) \, . 
\end{equation}
That is, we insert the cylinder over the defect disc $X_{L}$ in place of (a cylindrical neighbourhood of) $L$. 
In the other case, where~$L$ forms a circle in $M$, we set 
$\In_L(M) = (M \setminus N_L) \cup_{\varphi_L} (X_{L} \times S^1)$. 
Finally we define
\be 
\In(M)=\In_{L_m} \ldots \In_{L_1}(M) \, , 
\ee 
where $L_1, \ldots, L_m$ are all 1-strata of $M$. 
Clearly this is independent of the order of $L_1, \ldots, L_m$ and defines for  $M\colon \Sigma \longrightarrow \Sigma'$ a morphism  $\In(M) \colon \In(\Sigma) \longrightarrow \In(\Sigma')$ which does not depend on the  choices of closed neighbourhoods $N_{L}$. 
The functor $\In$ is symmetric monoidal by construction, and we thus obtain: 

\begin{definition}
\label{def:LineDefectCompletionOfZ}
  Let  $\zz \colon \Bordd[3] \longrightarrow \Vect$  be a defect TQFT. The \textsl{line defect completion} of $\zz$ is the defect TQFT 
  \be 
  \widehat{\zz} := \zz\circ \In \colon \Borddefn{3}(\widehat{\D}) \longrightarrow \Vect \, . 
  \ee 
\end{definition}

Recall from Section~\ref{subsubsec:3catForDefectTQFT} that to any 3-dimensional defect TQFT~$\zz$ there is an associated 3-category~$\tz$. 

\begin{proposition}
\label{proposition:tzhat}
We have an equivalence of Gray categories with duals: 
\be 
\tz \cong \mathcal T_{\widehat{\zz}} \, .
\ee 
\end{proposition}

\begin{proof}
  To see this, we apply the insertion functor to the cubes that correspond to morphisms in $\tzhat$. 
  More precisely, consider the functor $\In \colon \tzhat \longrightarrow \tz$ defined as follows. 
  It is the identity on objects and 1-morphisms. 
  
  To define $\In$ on 2-morphisms, first pick for each $Y \in \widehat{D}_{1}$ a square around~$Y$ and extend~$Y$ to a progressive diagram $\mathrm{prog}(Y)$ in the square.\footnote{In case there is a horizontal 1-stratum in $Y$, first pass to a choice of isomorphic progressive defect disc.}
   
  For a 2-morphism~$X$ in $\tzhat$ define $\In(X)$ by first picking a cube that represents~$X$, then insert  for  each 1-stratum with decoration $Y$ the corresponding progressive diagram $\mathrm{prog}(Y)$, where we pick the local neighbourhoods of the  1-strata small enough  to ensure  that their projections to the $x$-axis do not overlap (here we use the conventions of \cite[Sect.\,3.1.2]{CMS}). 
  After passing again to equivalence classes we obtain a well-defined 2-morphism of $\tz$.
  
  For the 3-morphisms we use that by \cite[Sect.\,3.3]{CMS}, the 3-morphisms in $\tzhat$ and in $\tz$ are obtained by applying $\zzhat$ and $\zz$, respectively, to defect spheres. 
  By definition, the corresponding defect spheres for $\Hom_{\tzhat}(X,X')$ and $\Hom_{\tz}(\In(X), \In(X'))$ match and we can identify the 3-morphisms.

  To see that $\In$ is an equivalence of Gray categories, it suffices to show that it is essentially surjective on 2-morphisms. 
  This is the case since each 2-morphism~$X$ of~$\tz$ gives a 2-morphism $\iota(X)$ of $\tzhat$, using the obvious inclusion $D_{1} \longrightarrow \widehat{D}_{1}$, which lifts to a functor $\iota \colon \Borddefn{3}(\D) \lra \Borddefn{3}(\widehat\D)$. 
  Thus $\In(\iota(X))=X$, showing that $\In$ is an equivalence. 
  Moreover, $\In$ is obviously compatible with the duals. 
\end{proof}

\subsubsection{A 2-category associated to line defect completion}
\label{sec:bicat-from-wideh}
By the general construction of the 3-category $\tz$ for a defect TQFT $\zz$, for all $u,v \in D_{3}$ there is a full sub-2-category $\tz^{1}(u,v)$ of $\tz(u,v)$ whose objects form the set $\{ \alpha \in D_{2} \,|\, s(\alpha)=u, \, t(\alpha)=v \}$. 
Thus, the 1-morphisms of $\tz^{1}(u,v)$ correspond almost to elements of $\widehat{D}_{1}$: 
the difference is that the elements of $\widehat{D}_{1}$ are neither 3-cubes (but defect 2-discs) nor isomorphism classes.
These are minor differences, however we want to use the elements of $\widehat{D}_{1}$ directly as morphisms of a 2-category for our construction in Section \ref{subsec:RibbonCategoriesFromSOD}. 
Thus we now mimic the construction of~$\tz$ to define 2-categories $\W(u,v)$ whose 1-morphisms are precisely elements of $\widehat{D}_{1}$. 
Then we show that $\W(u,v)$ is equivalent to $\tz^{1}(u,v)$.
 
\medskip 

We start with two operations for $\widehat{D}_{1}$. 
For $X \in \widehat{D}_{1}$, we denote by $X^{*} \in \widehat{D}_{1}$ the decorated 2-disc which is obtained by reversing the orientations of all strata of $X$.
Second, for $X, Y \in \widehat{D}_{1}$ with matching boundary in the sense that  $\partial X = \partial (Y^{*})$, 
we can consider $X \circ Y \in \Sphere_{2}(\D)$, which by definition is the defect 2-sphere that is obtained from gluing the 2-disc~$Y$ on top of the 2-disc~$X$ along their common boundary. 

For fixed elements $u,v \in D_{3}$, consider $\alpha,\beta \in D_{2}$ with $s(\alpha)=t(\beta)=u$ and $t(\alpha)=s(\beta)=v$, and set 
\be\label{eq:starlikecircle}
S_{\alpha,\beta} := \includegraphics[scale=1.0, valign=c]{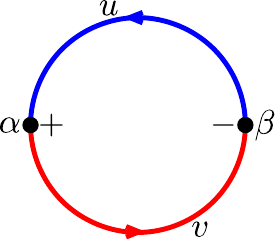} 
\, \in \Sphere_{1}(\D)
\, .
\ee
A choice of $S_{\alpha,\beta}$ defines the lower arrow in the following pullback diagram of sets:
\begin{equation}
  \label{eq:pullbackW1}
  \begin{tikzcd}
    \mathcal D(\alpha,\beta) \ar{d} \ar{r} &\widehat{D}_{1} \ar{d}{ \widehat{f}} \\
    \ast \ar[r,"S_{\alpha,\beta}"'] & \Sphere_{1}(\D) 
  \end{tikzcd}
\end{equation}
Thus $\mathcal D(\alpha,\beta)$ consists of the elements of $\widehat{D}_{1}$ with specified boundary.
We call~$\alpha$ the source of an element of $\mathcal D(\alpha,\beta)$ and $\beta$ the target.
In particular we can consider for $X,Y \in \mathcal D(\alpha,\beta)$ the defect 2-sphere $X^{*} \circ Y$.

\begin{lemma}
Let $\zz \colon \Bordd[3] \longrightarrow \Vect$ be a defect TQFT. 
For all $u,v \in D_{3}$ there is an associated linear pivotal 2-category  $\W(u,v)$  such that   
\begin{enumerate}
\item the objects of $\mathcal W(u,v)$ form the set $\{ \alpha \in D_{2} \,|\, s(\alpha)=u, \, t(\alpha)=v \}$,
\item the set of 1-morphisms of $\W(u,v)$ from $\alpha$ to $\beta$ is $\mathcal D(\alpha,\beta)$ as in~\eqref{eq:pullbackW1},
\item for $X,Y \in \mathcal D(\alpha,\beta)$, the set of 2-morphisms is
  \begin{equation}
    \label{eq:2-morph}
    \Hom_{\W(u,v)}(X,Y)=\zz(X^{*} \circ Y) \, . 
  \end{equation}
\end{enumerate}
\end{lemma}
\begin{proof}
  The proof is essentially contained in the proof of \cite[Thm.\,3.13]{CMS}, albeit in a slightly different setting.  We will need some details on the construction of $\W(u,v)$ later, thus we recall the main ingredients.  
  All compositions of 2-morphisms in $\W(u,v)$ are canonically obtained from $\zz$ by evaluating $\zz$ on defect 3-balls with 3-balls in the interior removed: 
  For 1-morphisms $X,Y,Z \in \mathcal D(\alpha,\beta)$, the vertical composition of 2-morphisms is a linear map
  \begin{equation}
    \label{eq:2}
    	 \Hom_{\W(u,v)}(Y,Z) \otimes_\Bbbk \Hom_{\W(u,v)}(X,Y) \lra \Hom_{\W(u,v)}(X,Z) \, ,
  \end{equation}
  which is given as $\zz(B(X,Y,Z))$ with the bordism 
  \be 
  B(X,Y,Z)\colon (Y^{*} \circ Z) \sqcup (X^{*} \circ Y) \longrightarrow X^{*} \circ Z 
  \ee 
  in $\Borddefn{3}(\D)$ 
  defined as follows. 
 The decorated 1-sphere $S_{\alpha,\beta}$ from  \eqref{eq:starlikecircle} gives the cylinder   $CS_{\alpha,\beta} \times [0,1]$ over the cone with cone point~0. 
 Remove the solid cylinder $B_{1/2}(0) \times [0,1]$ from the interior, then glue
  	 $
  	 (\frac{1}{2} X \times [0,\frac{1}{5}]) \cup (\frac{1}{2} Y 
  	 \times [\frac{2}{5},\frac{3}{5}]) \cup (\frac{1}{2} Z \times 
  	 [\frac{4}{5},1])
  	 $ 
 back in and collapse the outer boundary $S_{\alpha,\beta} \times [0,1]$ as well as the inner boundaries $\frac{1}{2}S_{\alpha,\beta} \times [\frac{1}{5},\frac{2}{5}]$ and  $\frac{1}{2}S_{\alpha,\beta} \times [\frac{3}{5},\frac{4}{5}]$ to obtain the ball $B(X,Y,Z)$ with two inner balls removed.
 As a result, the parallel 1-morphisms  and their 2-morphisms form categories $D(\alpha,\beta)$ (with $\textrm{Ob}(D(\alpha,\beta)) = \mathcal D(\alpha,\beta)$) with units  given by the value of $\zz$ on solid balls, see \cite[Sect.\,3.3]{CMS}. 

 The horizontal composition consists of linear functors $\otimes \colon D(\beta,\gamma) \times D(\alpha,\beta) \longrightarrow D(\alpha,\gamma)$ as follows:
 For $X \in D(\beta,\gamma)$ and $Y \in D(\alpha,\beta)$, the  object $X \otimes Y \in D(\alpha,\gamma)$ is defined to be the 2-disc which is obtained by placing the rescaled discs $\frac{1}{3}X$ and $\frac{1}{3} Y$
 next to each other in the disc $CS_{\alpha,\gamma}$. 
 
 To define $\otimes$ on morphisms we again use a 3-ball with two inner 3-balls removed that is defined similarly to the  case of the vertical composition and apply the  functor $\zz$, see \cite[Sect.\,3.3]{CMS}.

 Since $\zz$ is invariant under isotopies of bordisms, all axioms of a 2-category follow directly. The duals in $\W(u,v)$ are obtained as in \cite[Sect.\,3.4]{CMS}, i.\,e.\ the dual of a 1-morphism~$X$ is $X^{*}$. 
\end{proof} 

The equivalence $\In \colon \tzhat \longrightarrow \tz$ of Gray categories with duals from Proposition~\ref{proposition:tzhat} restricts to an equivalence $\W(u,v) \longrightarrow \tz^{1}(u,v)$: 
  
\begin{lemma}
For all $u,v \in D_{3}$, we have an equivalence of pivotal 2-categories: 
\be 
\W(u,v) \cong \tz^{1}(u,v) \, . 
\ee 
\end{lemma}

\section{Orbifold graph TQFTs}
\label{sec:OrbGrphTQFTs}

In this section we review orbifold TQFTs~$\zza$ of 3-dimensional defect TQFTs~$\zz$ as introduced in \cite{CRS1}, and we formulate their construction in terms of decorated skeleta (Section~\ref{subsec:OrbifoldTQFTs}). 
To any special orbifold datum~$\A$ for~$\zz$, we construct an associated ribbon category~$\wa$ (Section~\ref{subsec:RibbonCategoriesFromSOD}), generalising the construction of \cite{MuleRunk} for Reshetikhin--Turaev theories to arbitrary defect TQFTs. 
Finally, we lift~$\zza$ to an orbifold graph TQFT $\zzwa$, which acts on bordisms with embedded $\wa$-coloured ribbon graphs (Section~\ref{subsec:OrbifoldGraphTQFTs}). 
A variant of this result which is useful for applications is described in Appendix~\ref{app:ConstructionForNonEulerCompletedCase}.

\subsection{Orbifold TQFTs}
\label{subsec:OrbifoldTQFTs}

We start by recalling the notion of 3-dimensional orbifold TQFTs from \cite[Sect.\,3.4]{CRS1} and then
explain how this construction can be generalised and computationally simplified in terms of admissible skeleta in the case of special orbifold data.

\subsubsection{Special orbifold data}
\label{subsubsec:SOD}

Let $\zz\colon \Borddefn3(\mathds{D}) \lra \Vect$ be a defect TQFT as reviewed in Section~\ref{subsec:Review3dDefectTQFT}, with defect data $\mathds{D}=(D_3,D_2,D_1,s,t,f)$. 

\begin{definition}\label{def:special-orb-data}
A \textsl{special orbifold datum}~$\A$ for~$\zz$ consists of 
\begin{itemize}
	\item 
	an element $\A_3 \in D_3$, 
	\item 
	an element $\A_2 \in D_2$ with $s(\A_2) = \A_3 = t(\A_2)$, 
	\item 
	an element $\A_1 \in D_1$ with $f(\A_1) = (\A_2,+) \times (\A_2,+) \times (\A_2,-)$, 	
	\item 
	elements $\A_0^+ \in \zz({S}_{\A}^2)$ and $\A_0^- \in \zz(({S}_{\A}^2)^{\text{rev}})$, 
\end{itemize}
where 
\be 
{S}_{\A}^2 
=
\includegraphics[scale=1.0, valign=c]{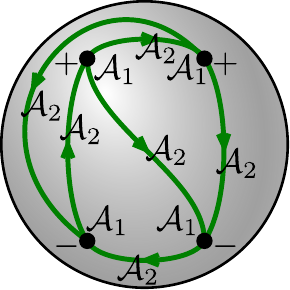} 
\, , \qquad 
({S}_{\A}^2)^{\text{rev}}
=
\includegraphics[scale=1.0, valign=c]{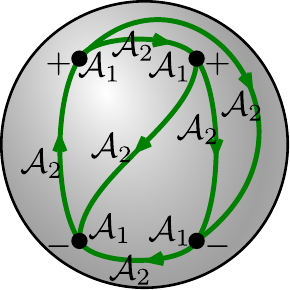} 
\ee 
are $\A$-decorated defect spheres, and in particular objects in $\Bordd[3]$. 
The tuple $\A = (\A_3,\A_2,\A_1,\A_0^\pm)$ satisfies the identities depicted in Figure~\ref{fig:SpecialOrbifoldData}, where it is understood that~$\zz$ is applied to the defect balls displayed  on either side of the equal sign.
\end{definition}

In drawing the defect balls in Figure~\ref{fig:SpecialOrbifoldData} we used that 
$\A$-decorated defect bordisms locally look as follows (where all 2-strata are oriented counterclockwise in the paper/screen frame, and an $\A_0^{\eps}$-decorated 0-stratum has orientation~$\eps$): 
\be
\label{eq:3dSpecialOrbifoldData}
\includegraphics[scale=1.0, valign=c]{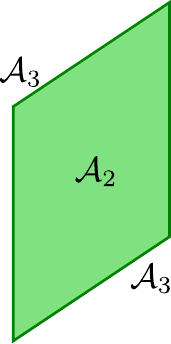} \, , \;\;
\includegraphics[scale=1.0, valign=c]{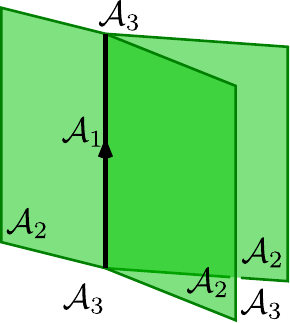} \, , \;\;
\includegraphics[scale=1.0, valign=c]{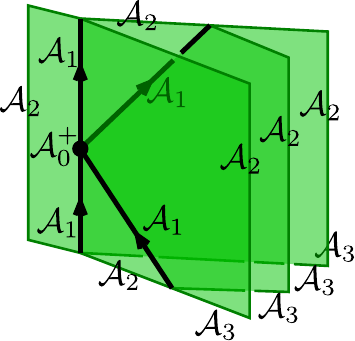} \, , \;\;
\includegraphics[scale=1.0, valign=c]{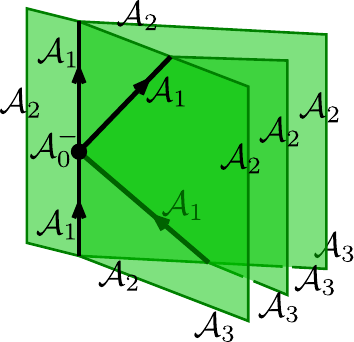} \, .
\ee
Note that $\A_0^\pm$-decorated 0-strata are interpreted as small defect 3-balls around them removed, with the resulting linear map after evaluation with~$\zz$ applied to (tensor products of the vectors) $\A_0^\pm$; this is made precise in terms of the $D_0$-completion mentioned in Remark~\ref{rem:RTandD0}\,\ref{item:D0completion}. 

\begin{figure}
	\captionsetup[subfigure]{labelformat=empty}
	\centering
	\vspace{-60pt}
	\begin{subfigure}[b]{0.5\textwidth}
		\centering
		\includegraphics[scale=0.85, valign=c]{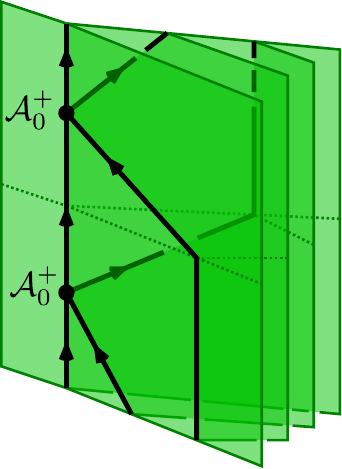} $=$
		\includegraphics[scale=0.85, valign=c]{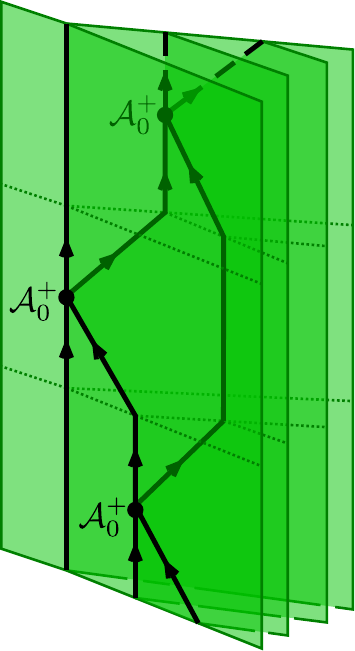}
		\caption{}
		\label{eq:O1}
	\end{subfigure}\raisebox{8em}{(O1)}\\
	\vspace{-15pt}
	\hspace{-60pt}
	\begin{subfigure}[b]{0.53\textwidth}
		\centering
		\includegraphics[scale=0.85, valign=c]{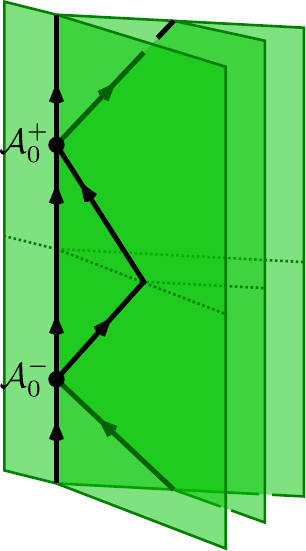} $=$
		\includegraphics[scale=0.85, valign=c]{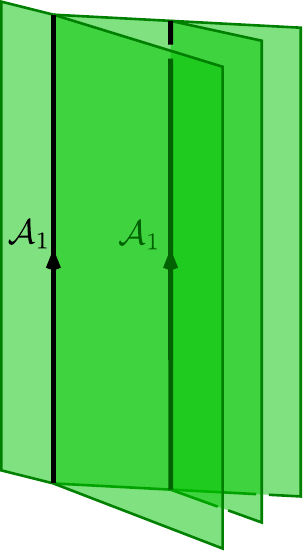}
		\caption{}
		\label{eq:O2}
	\end{subfigure}\hspace{-2em}\raisebox{6.5em}{(O2)}
	\hspace{-15pt}
	\begin{subfigure}[b]{0.53\textwidth}
		\centering
		\includegraphics[scale=0.85, valign=c]{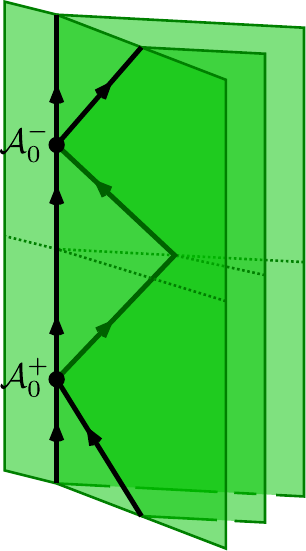} $=$
		\includegraphics[scale=0.85, valign=c]{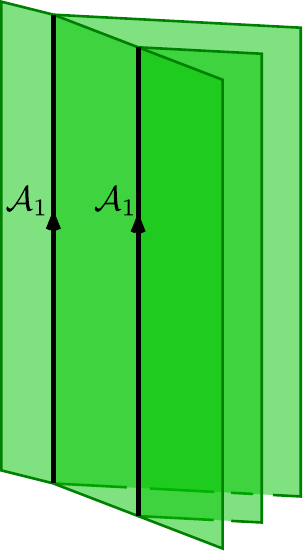}
		\caption{}
		\label{eq:O3}
	\end{subfigure}\hspace{-2em}\raisebox{6.5em}{(O3)}\\  
	\vspace{-15pt}
	\hspace{-60pt}
	\begin{subfigure}[b]{0.53\textwidth}
		\centering
		\includegraphics[scale=0.85, valign=c]{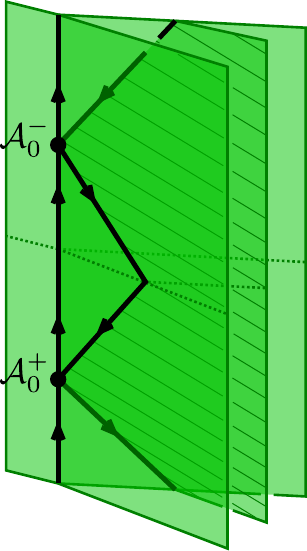} $=$
		\includegraphics[scale=0.85, valign=c]{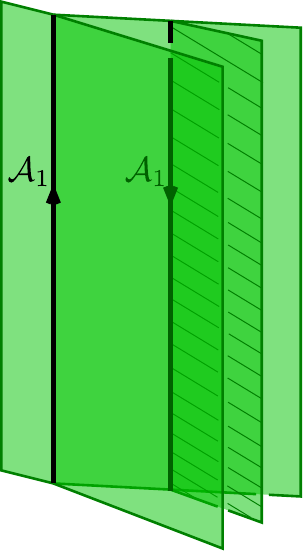}
		\caption{}
		\label{eq:O4}
	\end{subfigure}\hspace{-2em}\raisebox{6.5em}{(O4)}
	\hspace{-15pt}
	\begin{subfigure}[b]{0.53\textwidth}
		\centering
		\includegraphics[scale=0.85, valign=c]{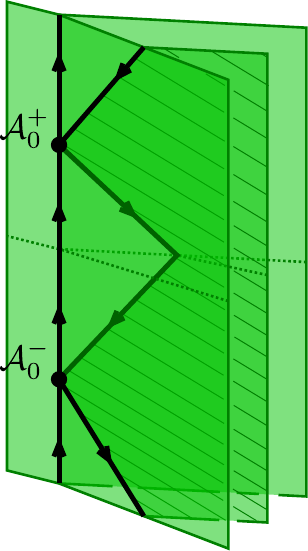} $=$
		\includegraphics[scale=0.85, valign=c]{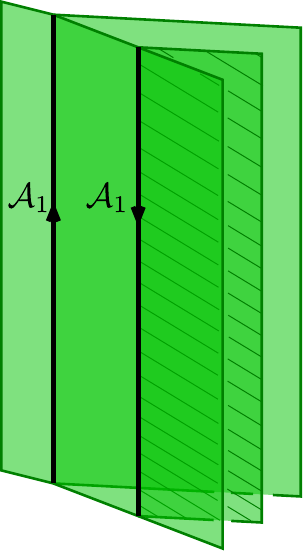}
		\caption{}
		\label{eq:O5}
	\end{subfigure}\hspace{-2em}\raisebox{6.5em}{(O5)}\\
	\vspace{-15pt}
	\hspace{-60pt}
	\begin{subfigure}[b]{0.53\textwidth}
		\centering
		\includegraphics[scale=0.8, valign=c]{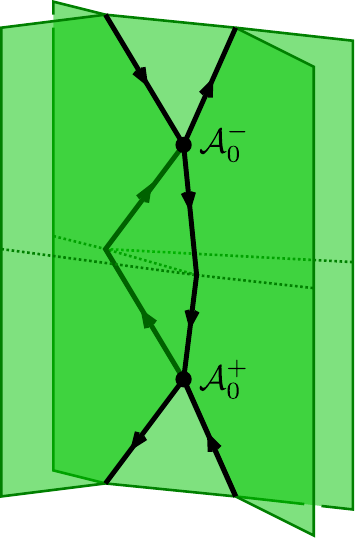} $=$
		\includegraphics[scale=0.8, valign=c]{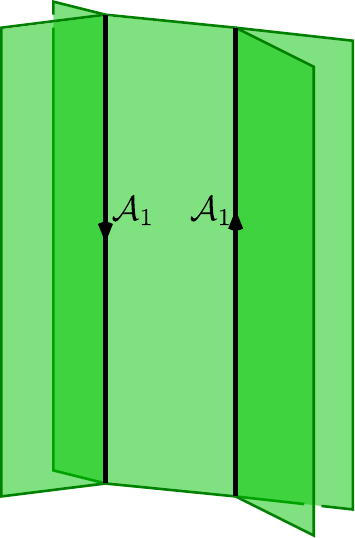}
		\caption{}
		\label{eq:O6}
	\end{subfigure}\hspace{-1.5em}\raisebox{6em}{(O6)}
	\hspace{-15pt}
	\begin{subfigure}[b]{0.53\textwidth}
		\centering
		\includegraphics[scale=0.8, valign=c]{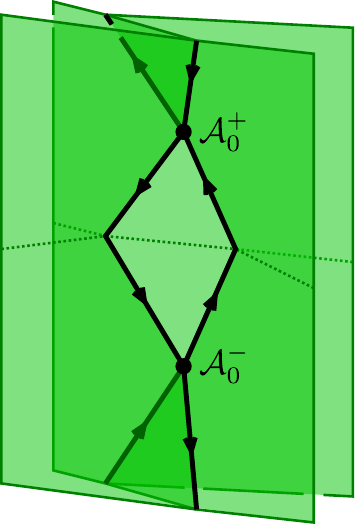} $=$
		\includegraphics[scale=0.8, valign=c]{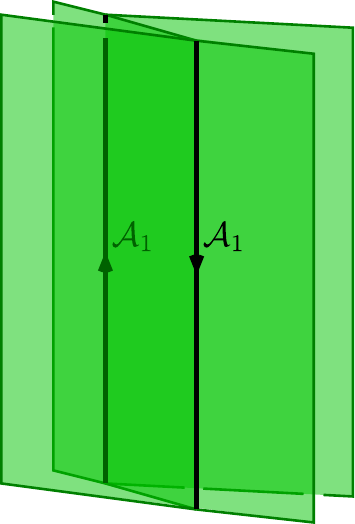}
		\caption{}
		\label{eq:O7}
	\end{subfigure}\hspace{-1.5em}\raisebox{6em}{(O7)}\\
	\vspace{-15pt}
	\hspace{-60pt}
	\begin{subfigure}[b]{1.0\textwidth}
		\centering
		\includegraphics[scale=0.8, valign=c]{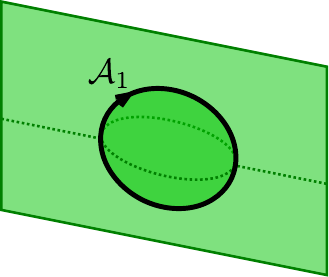} $=$
		\includegraphics[scale=0.8, valign=c]{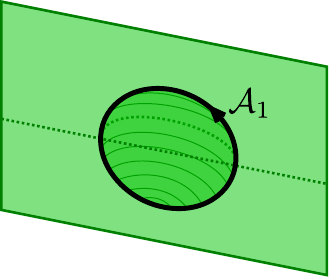} $=$
		\includegraphics[scale=0.8, valign=c]{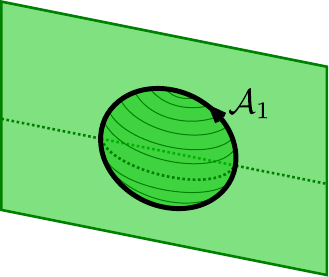} $=$
		\includegraphics[scale=0.8, valign=c]{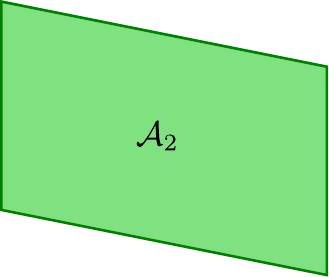}
		\caption{}
		\label{eq:O8}
	\end{subfigure}\hspace{-2em}\raisebox{4em}{(O8)}
	\vspace{-15pt}
	\caption{Defining conditions on special orbifold data~$\A$ for a defect TQFT~$\zz$; only the $\A_j$-labels for $j$-strata of lowest dimension are shown.
	Each picture represents a defect ball viewed as a bordism from~$\varnothing$ to the boundary, and the identities hold only after application of~$\zz$. }
	\label{fig:SpecialOrbifoldData}
\end{figure}

\begin{remark}
Note that the underlying stratified 3-balls of the defect bordisms depicted in~\eqref{eq:3dSpecialOrbifoldData} are, when read from left to right, Poincar\'{e} dual to a 1-simplex, a 2-simplex, and two 3-simplices (all oriented),
cf.~\eqref{eq:orientation-2strata} and \eqref{eq:orientation-10strata}.
Accordingly, general ``orbifold data'' for~$\zz$ are defined in \cite{CRS1} as above, but with the defining conditions in Figure~\ref{fig:SpecialOrbifoldData} replaced by (the larger number of) conditions arising from the Poincar\'{e} duals of oriented versions of the 2-3 and 1-4 Pachner moves (relating triangulations of 3-dimensional bordisms). 
In fact, orbifold data for any $n$-dimensional defect TQFT are defined for arbitrary $n\in\Z_+$ in \cite[Def.\,3.5]{CRS1} in terms of invariance under $n$-dimensional oriented Pachner moves. 

Special orbifold data are special cases of 3-dimensional orbifold data. 
From here on we will only consider special orbifold data.
\end{remark}

\subsubsection{Decorated skeleta}
\label{subsubsec:OrbifoldsFromSkeleta}

We now move to consider admissible skeleta that are decorated with special orbifold data.
We will see that special orbifold data algebraically encode invariance under admissible BLT moves. 

\begin{definition}
	\label{def:AdecoratedSkeleta}
	Let~$\zz$ be a defect TQFT, and let $\A = (\A_3, \A_2, \A_1, \A^\pm_0)$ be a special orbifold datum for~$\zz$. 
	An \textsl{$\mathcal A$-decorated skeleton} $\mathcal S$ of a bordism~$M$ in $\Bord_3$ is an admissible skeleton~$S$ of~$M$ together with a decoration as follows: 
	\begin{itemize}
		\item 
		each 3-stratum of~$S$ is decorated by~$\A_3$, 
		\item 
		each 2-stratum of~$S$ is decorated by~$\A_2$, 
		\item 
		each 1-stratum of~$S$ is decorated by~$\A_1$, 
		\item 
		for $\varepsilon \in \{+,-\}$, each $\varepsilon$-oriented 0-stratum of~$S$ is decorated by~$\A^\varepsilon_0$. 
	\end{itemize}
\end{definition}

One similarly arrives at the notion of \textsl{$\A$-decorated BLT moves}: these are local changes of $\A$-decorated skeleta (and hence of $\A$-decorated defect bordisms) whose effect on the underlying admissible skeleta are precisely the admissible BLT moves of Section~\ref{subsubsec:BLTmoves}. 
The evaluation with a TQFT is invariant under these moves, and special orbifold data precisely encode this invariance: 

\begin{proposition}
	\label{prop:SODBLT}
	Let $\zz\colon \Borddefn3(\mathds{D}) \lra \Vect$ be a defect TQFT, and let $\A = (\A_3, \A_2, \A_1, \A^\pm_0)$ be a list of defect labels for~$\zz$ that can decorate admissible skeleta of all bordisms in $\Bord_3$ to obtain morphisms in $\Borddefn3(\mathds{D})$. 
	Then~$\A$ is a special orbifold datum for~$\zz$ iff~$\zz$ applied to $\A$-decorated BLT moves gives identities in $\Vect$. 
\end{proposition}

\begin{proof}
	In Lemma \ref{lem:PachnerFromBLT} we showed the equivalence of undecorated BLT moves and undecorated special orbifold data moves.
	From this the equivalence of the corresponding decorated moves between decorated skeleta follows immediately.
\end{proof}

\subsubsection{Definition of orbifold TQFTs}
\label{subsubsec:OrbifoldTQFTs}

Let $\zz \colon \Borddefn3(\mathds{D}) \lra \Vect$ be a defect TQFT, and let~$\A$ be a special orbifold datum for~$\zz$. 
Given an $\A$-decorated skeleton~$\mathcal S$ of a bordism 
\be 
M \colon \Sigma \lra \Sigma'
\ee
in $\Bord_3$, we obtain a new defect bordism $F(\mathcal S)$, which we call a \textsl{foamification} of~$M$ represented by $\mathcal S$. 
Viewed as a morphism in $\Borddefn{3}(\D)$, the defect bordism $F(\mathcal S)$ has source and target objects which are $\D$-decorated defect surfaces whose 1- and 0-strata are labelled by~$\A_2$ and~$\A_1$, respectively (corresponding to 2- and 1-strata of~$M$ which end on~$\partial M$). 
We denote these source and target objects $F(\Sigma, \mathcal G)$ and $F(\Sigma', \mathcal G')$, respectively, where $\mathcal G, \mathcal G'$ are the decorated skeleta of $\Sigma, \Sigma'$ induced from~$\mathcal S$. 
Thus 
\be
\label{eq:FoamificationMapForBeginners}
F(\mathcal S) \colon F(\Sigma, \mathcal G) \lra F(\Sigma', \mathcal G')
\ee 
in $\Borddefn{3}(\D)$. 

By definition, the evaluation of~$\zz$ on an $\A$-decorated skeleton~$\mathcal S$ is given by $\zz(F(\mathcal S))$. 
In particular, $\zz$ can be evaluated on defect 3-balls around each side of the identities in Figure~\ref{fig:SpecialOrbifoldData}, and doing so gives identities in $\Vect$. 
This in turn implies that if $\partial M = \varnothing$, we have $\zz(F(\mathcal S)) = \zz(F(\mathcal S'))$ for any other $\A$-decorated skeleton~$\mathcal S'$ of~$M$. 
Hence setting 
\be 
\zza(M) := \zz(F(\mathcal S)) \qquad \text{if } \partial M = \varnothing
\ee 
does not depend on the choice of $\A$-decorated skeleton~$\mathcal S$ of~$M$, thanks to Proposition~\ref{prop:SODBLT} and Theorem~\ref{thm:ConnectSkeleta}. 

To prepare for the definition of~$\zza$ on bordisms with nonempty boundary, let~$\Sigma$ be an object in $\Bord_3$. 
Any choice of $\A$-decorated skeleton~$\mathcal S$ of the cylinder $C_\Sigma := \Sigma \times [0,1]$ gives rise to a linear map 
\be 
\Phi_{\mathcal G}^{\mathcal G'} := \zz\big( F(\mathcal S) \big) \colon 
\zz \big( F(\Sigma,\mathcal G) \big) \lra \zz \big( F(\Sigma,\mathcal G') \big) \, , 
\ee 
where $\mathcal G, \mathcal G'$ are the decorated skeleta of~$\Sigma$ induced by~$\mathcal S$ as in~\eqref{eq:FoamificationMapForBeginners}. 
By definition of special orbifold data, the linear maps $\Phi_{\mathcal G}^{\mathcal G'}$ do not depend on the choice of $\A$-decorated skeleta~$\mathcal S$ in the interior of~$C_\Sigma$, and for arbitrary decorated admissible skeleta $\mathcal G, \mathcal G', \mathcal G''$ of~$\Sigma$ we have 
\be
\Phi_{\mathcal G'}^{\mathcal G''} \circ \Phi_{\mathcal G}^{\mathcal G'} = \Phi_{\mathcal G}^{\mathcal G''} \, . 
\ee 
In particular, the maps $\Phi_{\mathcal G}^{\mathcal G}$ are idempotents. 

\begin{construction}
	\label{constr:OrbifoldTQFT}
	Let~$\mathcal A$ be a special orbifold datum for a defect TQFT~$\zz$. 
	The \textsl{$\A$-orbifold TQFT}  
	\be 
	\zza \colon \Bord_3 \lra \Vect
	\ee  
	is defined as follows: 
	\begin{enumerate}
		\item 
		For an object $\Sigma \in \Bord_3$, we set 
		\be 
		\zza(\Sigma) = \textrm{colim} \big\{ \Phi_{\mathcal G}^{\mathcal G'} \big\} , 
		\ee 
		where $\mathcal G, \mathcal G'$ range over all admissible $\A$-decorated skeleta of~$\Sigma$. 
		\item 
		For a morphism $M \colon \Sigma \lra \Sigma'$ in $\Bord_3$, we set $\zza(M)$ to be 
		\be 
		\label{eq:ZAonMorphisms}
		\begin{tikzcd}
		\zza(\Sigma) \arrow[hook]{r} 
		& \zz \big( F(\Sigma, \mathcal G) \big) \arrow{rr}{\zz( F(\mathcal S) )}
		& 
		& \zz\big(  F(\Sigma', \mathcal G') \big) \arrow[two heads]{r}
		& \zza(\Sigma') \, , 
		\end{tikzcd}  
		\ee 
		where~$\mathcal S$ is an arbitrary $\A$-decorated skeleton representing~$M$, the first map is obtained from the universal property of the colimit, and the last map is part of the data of the colimit. 
	\end{enumerate}
\end{construction}

It is straightforward to verify from Proposition~\ref{prop:SODBLT} and Theorem~\ref{thm:ConnectSkeleta} that the definition of $\zza(M)$ in~\eqref{eq:ZAonMorphisms} does not depend on the choice of admissible $\A$-decorated skeleton~$\mathcal S$. 
Moreover, by construction the state spaces of~$\zza$ are isomorphic to the images of the idempotents, 
\be 
\zza(\Sigma) \cong \textrm{Im} \big( \Phi_{\mathcal G}^{\mathcal G} \big) \, . 
\ee 

\begin{theorem}
	\label{thm:MainResultCRS1}
	Let~$\mathcal A$ be a special orbifold datum for a defect TQFT~$\zz$. 
	Then $\zza \colon \Bord_3 \lra \Vect$ as in Construction~\ref{constr:OrbifoldTQFT} is a symmetric monoidal functor. 
\end{theorem}
\begin{proof}
In light of the discussion in Section~\ref{subsubsec:SOD}, the proofs of Thm.~3.10 and Prop.~3.18 in \cite{CRS1} generalise to the case of $\A$-decorated skeleta whose underlying stratifications are not Poincar\'{e} duals of triangulations. 
\end{proof}

\begin{remark}
It is typically easier to evaluate~$\zza$ in terms of $\A$-decorated skeleta which are not Poincar\'{e} duals of triangulations. 
For example, instead of computing the invariant $\zza(S^3)$ from a triangulation of the 3-sphere (which involves at least five tetrahedra), one can choose the $\A$-decorated skeleton~$\mathcal S$ consisting only of an embedded $\A_2$-decorated 2-sphere that divides~$S^3$ into two $\A_3$-decorated 3-balls.  
This skeleton has no 1-strata and no 0-strata. 
\end{remark}

\subsection{Ribbon categories associated to special orbifold data} 
\label{subsec:RibbonCategoriesFromSOD}

Let $\zz \colon \Borddefn{3}(\D) \lra \Vect$ be a defect TQFT, and let~$\A$ be a special orbifold datum for the completed TQFT~$\widehat\zz$ of Definition~\ref{def:LineDefectCompletionOfZ}. 
In this section we describe a ribbon category~$\wa$ that is naturally associated to~$\zz$ and~$\A$. 
As will be explained in more detail in \cite{CMRSS2}, for a TQFT~$\zz$ of Reshetikhin--Turaev type associated to a modular fusion category~$\mathcal C$, our~$\wa$ is equivalent to the category of Wilson lines~${\mathcal C_{\A}}$ introduced in \cite{MuleRunk}. 

\medskip 

Recall the 3-category~$\tz$ (reviewed in Section~\ref{subsubsec:3catForDefectTQFT}), the 2-categories $\mathcal W(u,v)$ associated to a pair of labels $u,v \in D_3$ in Section~\ref{subsec:MonoidalCatFromDefectTQFT}, and the line-completed TQFT~$\zzhat$ (Definition~\ref{def:LineDefectCompletionOfZ}). 
We write 
\be
\mathcal W := \End_{\mathcal W(\A_3,\A_3)}(\A_2)
\ee 
for the monoidal category of endomorphisms of~$\A_2$. 

\begin{definition}
  \label{definition:WA}
The category~$\wa$ is defined as follows: 
\begin{itemize}
	\item 
	Objects of~$\wa$ are tuples $\mathcal X = (X,\tau_1^X, \tau_2^X)$, with $X\in\mathcal W$, and 
	\begin{align}
	\label{eq:Tcrossings}
	\tau_1^X & \in \zzhat \left(  
	\includegraphics[scale=1.0, valign=c]{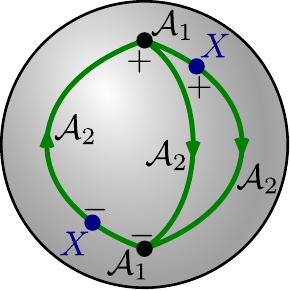} 
                   \right) 
	, 
	&& 
	\tau_2^X  \in \zzhat \left( 
	\includegraphics[scale=1.0, valign=c]{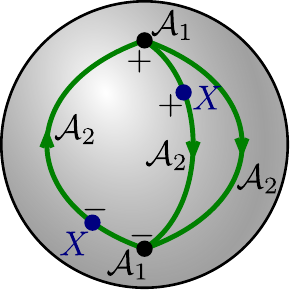}  
           \right) 
	,
	\end{align}
	are vectors, referred to as \textsl{crossings}, which correspond to 3-isomorphisms
	\be\label{eq:tau_i_crossing-def-1}
	\includegraphics[scale=1.0, valign=c]{pic_WAdef_Diagram_tau1.pdf} 
	\, , \qquad 
	\includegraphics[scale=1.0, valign=c]{pic_WAdef_Diagram_tau2.pdf} 
	\ee 
	in the 3-category~$\tz$. 
	Their inverses are denoted
	\be\label{eq:tau_i_crossing-def-2}
	\includegraphics[scale=1.0, valign=c]{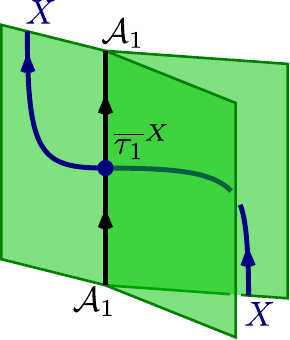}  
	\, , \qquad 
	\includegraphics[scale=1.0, valign=c]{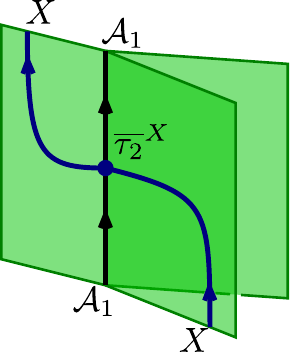}  
	\, , 
	\ee 
	and the crossings have to satisfy the identities in Figure~\ref{fig:CrossingIdentities} when~$\zzhat$ is applied to both respective sides, each viewed as a defect 3-ball. 
	(That $\tau_i^X$ and $\overline\tau_i^X$ are each others' inverse is expressed in \eqrefT{4} and \eqrefT{5}.)	
	\item 
	A morphism $\mathcal X \lra \mathcal X'$ in~$\wa$ is a morphism $f\colon X \lra X'$ in~$\mathcal W$ such that
	\begin{align}
	\zzhat\left( 
	\includegraphics[scale=1.0, valign=c]{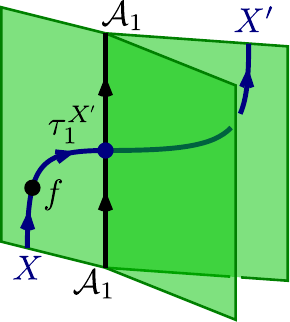}  
	\right) & = 
		\zzhat\left( 
	\includegraphics[scale=1.0, valign=c]{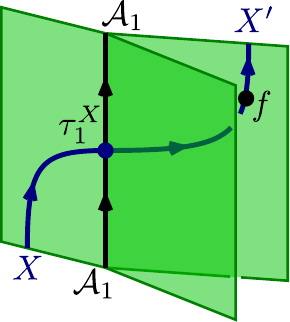} 
	\right) 
	, 
	\label{eq:M1}
	&\\ 
	\zzhat\left( 
	\includegraphics[scale=1.0, valign=c]{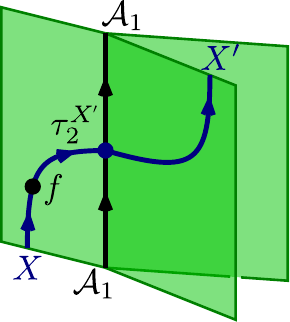} 
	\right) & = 
	\zzhat\left( 
	\includegraphics[scale=1.0, valign=c]{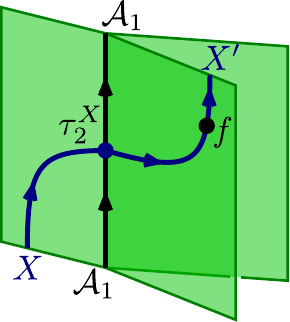} 
	\right) . 
	\label{eq:M2}
	\end{align}
	\item 
	Composition and identities in~$\wa$ are as in~$\mathcal W$. 
\end{itemize}
\end{definition}

\begin{figure}
	\captionsetup[subfigure]{labelformat=empty}
	\centering
	\hspace{-25pt}
	\makebox[1.1\textwidth]{
	\begin{subfigure}[b]{0.51\textwidth}
		\noindent
		\includegraphics[scale=1.0, valign=c]{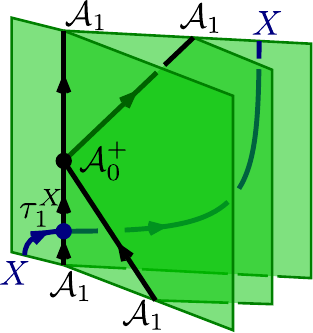} $=$
		\includegraphics[scale=1.0, valign=c]{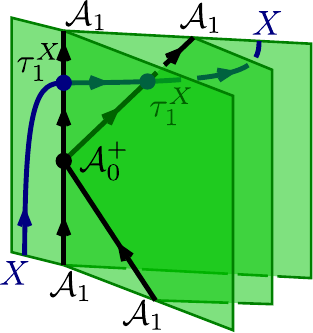}
		\caption{}
		\label{eq:T1}
	\end{subfigure}\hspace{-1em}\raisebox{5.5em}{(T1)}
	\begin{subfigure}[b]{0.54\textwidth}
		\noindent
		\includegraphics[scale=1.0, valign=c]{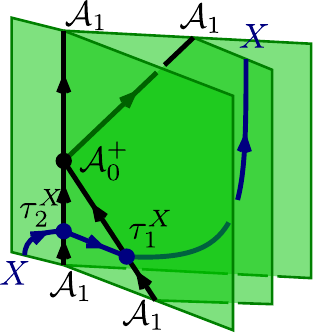} $=$
		\includegraphics[scale=1.0, valign=c]{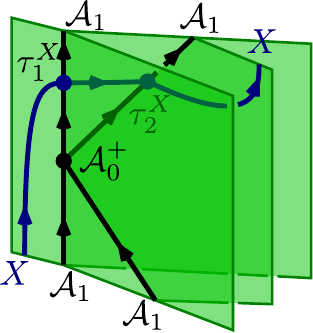}
		\caption{}
		\label{eq:T2}
	\end{subfigure}\hspace{-2em}\raisebox{5.5em}{(T2)}}\\
	\vspace{-10pt}
	\begin{subfigure}[b]{0.6\textwidth}
		\centering
		\includegraphics[scale=1.0, valign=c]{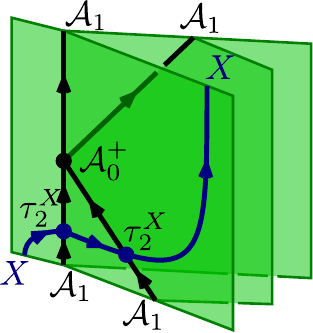} $=$
		\includegraphics[scale=1.0, valign=c]{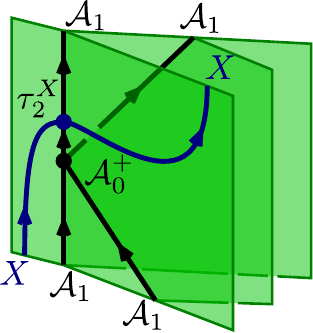}
		\caption{}
		\label{eq:T3}  
	\end{subfigure}\hspace{-2em}\raisebox{5.5em}{(T3)}\\
	\vspace{-15pt}
	\begin{subfigure}[b]{0.95\textwidth}
		\centering
		\includegraphics[scale=1.0, valign=c]{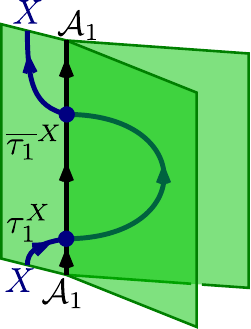} $=$
		\includegraphics[scale=1.0, valign=c]{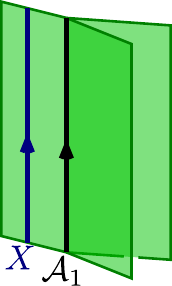}  , \quad
		\includegraphics[scale=1.0, valign=c]{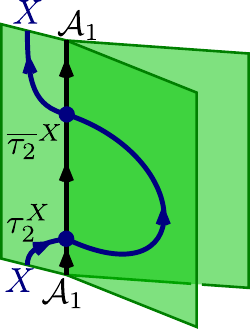} $=$
		\includegraphics[scale=1.0, valign=c]{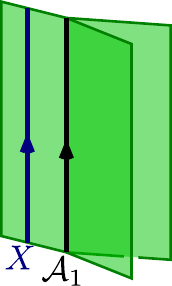}
		\caption{}
		\label{eq:T4}
	\end{subfigure}\hspace{-1em}\raisebox{5.5em}{(T4)}\\
	\vspace{-15pt}
	\begin{subfigure}[b]{0.95\textwidth}
		\centering
		\includegraphics[scale=1.0, valign=c]{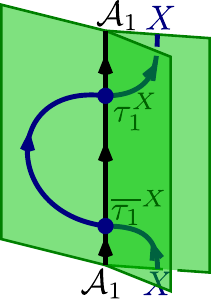} $=$
		\includegraphics[scale=1.0, valign=c]{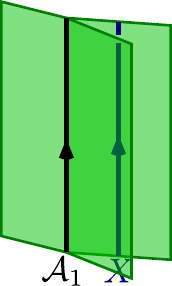}  , \quad
		\includegraphics[scale=1.0, valign=c]{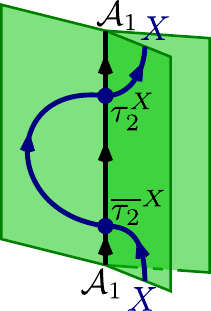} $=$
		\includegraphics[scale=1.0, valign=c]{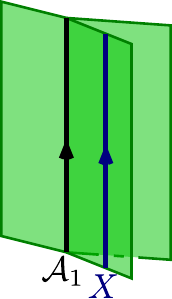}
		\caption{}
		\label{eq:T5}
	\end{subfigure}\hspace{-1em}\raisebox{5.5em}{(T5)}\\
	\vspace{-15pt}
	\begin{subfigure}[b]{0.95\textwidth}
		\centering
		\includegraphics[scale=1.0, valign=c]{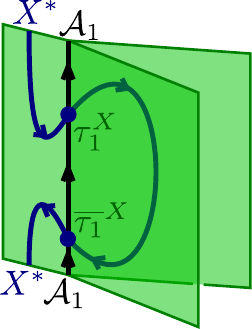} $=$
		\includegraphics[scale=1.0, valign=c]{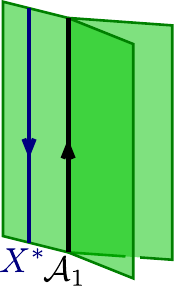}  , \quad
		\includegraphics[scale=1.0, valign=c]{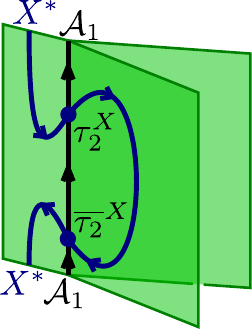} $=$
		\includegraphics[scale=1.0, valign=c]{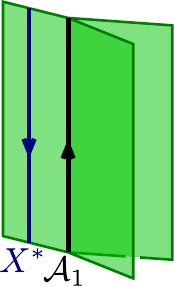}
		\caption{}
		\label{eq:T6}
	\end{subfigure}\hspace{-1em}\raisebox{5.5em}{(T6)}\\
	\vspace{-15pt}
	\begin{subfigure}[b]{0.95\textwidth}
		\centering
		\includegraphics[scale=1.0, valign=c]{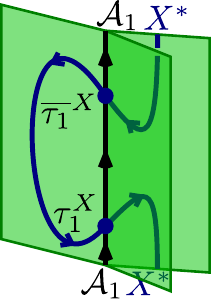} $=$
		\includegraphics[scale=1.0, valign=c]{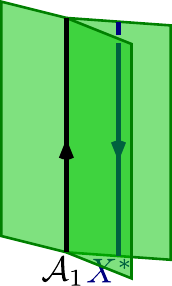}  , \quad
		\includegraphics[scale=1.0, valign=c]{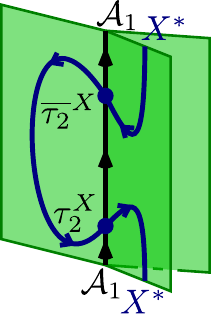} $=$
		\includegraphics[scale=1.0, valign=c]{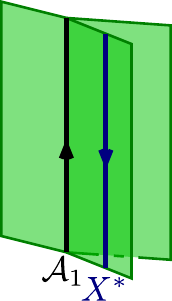}
		\caption{}
		\label{eq:T7}
	\end{subfigure}\hspace{-1em}\raisebox{5.5em}{(T7)}
	\vspace{-15pt}
	\caption{Defining conditions for objects in~$\wa$ (after application of $\widehat\zz$, i.\,e.\ in the 3-category $T_{\widehat{\zz}}\cong\tz$).}
	\label{fig:CrossingIdentities}
\end{figure}

\medskip 

We endow~$\wa$ with a rigid monoidal structure as follows: 
\begin{itemize}
	\item 
	The tensor product of~$\mathcal X$ with~$\mathcal Y$ in~$\wa$ is 
	\be 
	\big(X, \tau_1^X, \tau_2^X\big) \otimes_\A \big(Y, \tau_1^Y, \tau_2^Y\big) 
		= 
		\big(X\otimes Y, \tau_1^{X,Y}, \tau_2^{X,Y}\big) , 
	\ee 
	where~$\otimes$ on the right-hand side denotes the tensor product of~$\mathcal W$, and the crossings are 
	\begin{align}
	\tau_1^{X,Y} & = 
	\includegraphics[scale=1.0, valign=c]{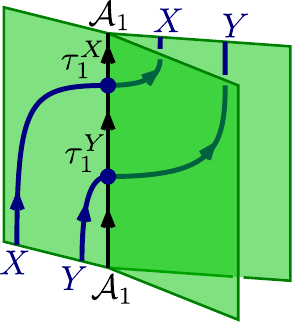} 
	\, , \quad 
	\tau_2^{X,Y}  = 
	\includegraphics[scale=1.0, valign=c]{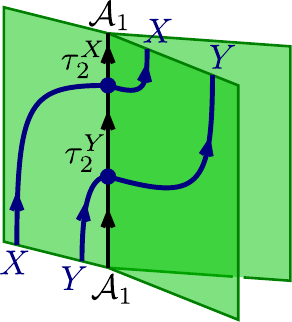} 
	\, . 
	\end{align}
	\item 
	The unit object of~$\wa$ is $\one_{\wa} = (\one, \tau_1^\one, \tau_2^\one)$, where~$\one$ is the unit object of~$\mathcal W$, and $\tau_1^\one, \tau_2^\one$ are obtained from the unitors of $\A_1$ in~$\tz$. 
	\item 
	Associators and unitors in~$\wa$ are as in~$\mathcal W$. 
	\item 
	The dual of an object $(X, \tau_1^X, \tau_2^X)$ in~$\wa$ is $(X^*, \tau_1^{X^*}, \tau_2^{X^*})$, where $X^* \in \mathcal W$ is the dual of $X\in\mathcal W$, and 
	\begin{align}
	\tau_1^{X^*} & = 
	\includegraphics[scale=1.0, valign=c]{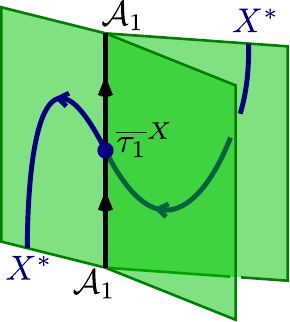} 
	\, , \quad 
	\tau_2^{X^*}  = 
	\includegraphics[scale=1.0, valign=c]{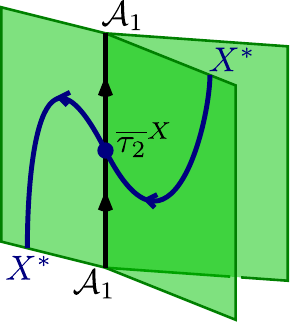}  
	\, , 
	\end{align}
	while the adjunction morphisms in~$\wa$ are those of~$\mathcal W$. 
\end{itemize}

In the following we will sometimes omit the orientations of $1$-strata when they are clear from the context.

\begin{proposition}\label{prop:WA-ribbon-cat}
	The rigid monoidal category~$\wa$ is pivotal, and together with the braiding morphisms 
	\begin{align}
		c_{\mathcal X, \mathcal Y} & = 
		\zzhat \left( 
		\includegraphics[scale=1.0, valign=c]{pic_WAdef_cXY.pdf}
		\right) 
		, \quad 
		c_{\mathcal X, \mathcal Y}^{-1} = 
		\zzhat \left(
		\includegraphics[scale=1.0, valign=c]{pic_WAdef_cXYinv.pdf}
		\right)  
		, 
	\end{align}
	it becomes a ribbon category.
\end{proposition}

Before giving the proof, let us recall that, as usual in a ribbon category, the ribbon twist of $\mathcal X \in \wa$ and its inverse can be defined in terms of dualities and the braiding as follows:
\begin{align}
\label{eq:TwistInWA}
\theta_{\mathcal X} := 
\zzhat \left(
\includegraphics[scale=1.0, valign=c]{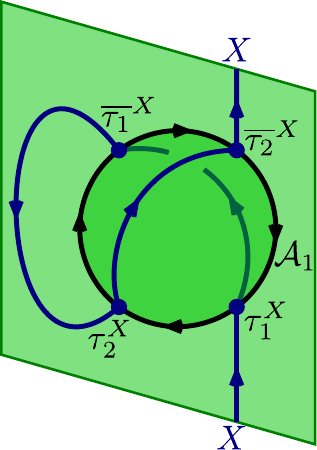}
\right) =
\zzhat \left(
\includegraphics[scale=1.0, valign=c]{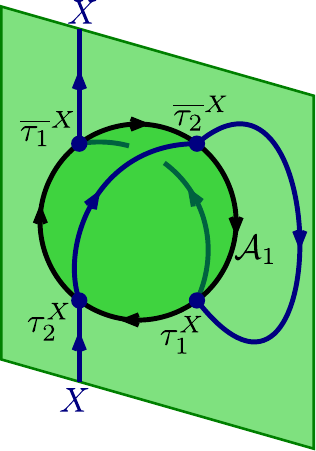}
\right) ,\\
\label{eq:TwistInvInWA}
\theta_{\mathcal X}^{-1} := 
\zzhat \left(
\includegraphics[scale=1.0, valign=c]{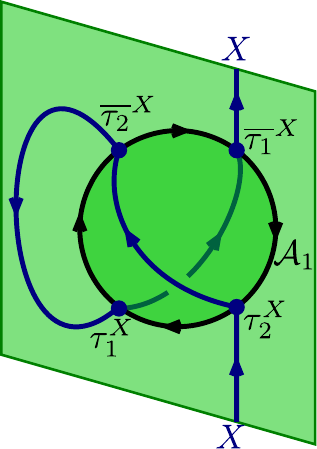}
\right) =
\zzhat \left(
\includegraphics[scale=1.0, valign=c]{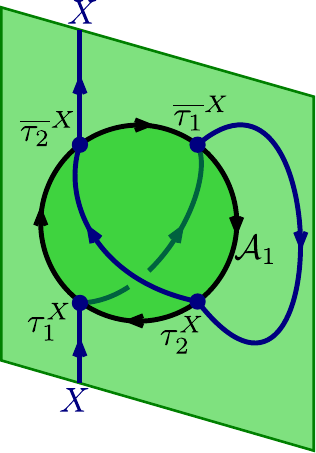}
\right) .
\end{align}

\begin{proof}[Proof of Proposition~\ref{prop:WA-ribbon-cat}]
The argument that $\wa$ is pivotal and braided is as in \cite[Sect.\,3.2]{MuleRunk} (with the parameters $\psi_i, \omega_i$ of loc.\ cit.\ set to~1, and with all string diagrams replaced by the corresponding 3-dimensional diagrams, to which~$\zzhat$ is applied).
However, the proof of the ribbon property in \cite{MuleRunk} used a shortcut that relies on semisimplicity (see Remark~3.7\,(ii) there), so here we need to add an extra calculcation.
 	
	First note that as in \cite[Lem.\,3.4]{MuleRunk}, for every $\mathcal X, \mathcal Y \in \wa$ we have
	\be
	\label{eq:aux_braiding_ids} 
	\includegraphics[scale=1.0, valign=c]{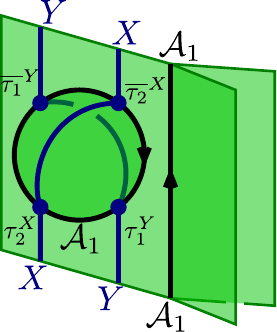} =
	\includegraphics[scale=1.0, valign=c]{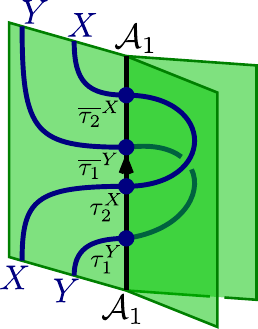}
	\, , \quad 
	\includegraphics[scale=1.0, valign=c]{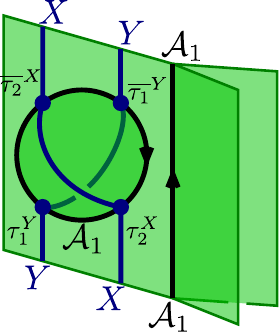} =
	\includegraphics[scale=1.0, valign=c]{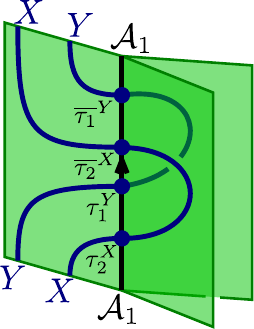}
	\ee
	To show that a pivotal braided category is ribbon, one needs to check that the left and right twists agree, i.\,e.\  that the equality in \eqref{eq:TwistInWA} holds.
	One has:
	\be
	\label{eq:WA-ribbon_proof_step1}
	\includegraphics[scale=1.0, valign=c]{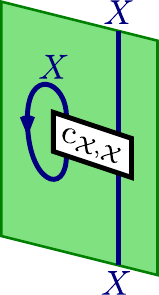} =
	\includegraphics[scale=1.0, valign=c]{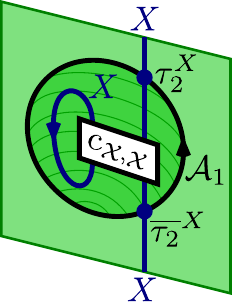} =
	\includegraphics[scale=1.0, valign=c]{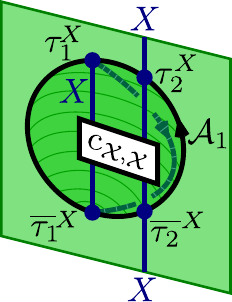} \, ,
	\ee
	where in the first equality we used \eqrefO{8}, \eqrefT{5} and \eqref{eq:M2} to create a bubble and move the coupon on it, and in the second equality we used \eqrefT{4} to move the $X$-strand onto the $2$-stratum to the back 
	(note that both this $2$-stratum and the strand lying in it have opposite to paper/screen plane orientation, hence the stripy pattern).
	Next, using the auxiliary identities \eqref{eq:aux_braiding_ids} together with \eqrefT{5} and \eqrefT{6} one gets:
			\begin{align}
			\includegraphics[scale=1.0, valign=c]{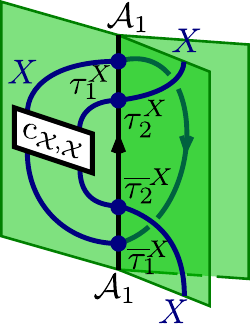} 
			& =
			\includegraphics[scale=1.0, valign=c]{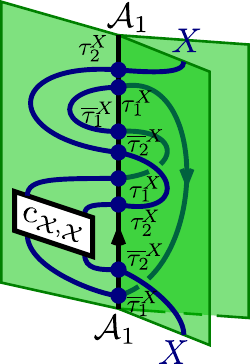} =
			\includegraphics[scale=1.0, valign=c]{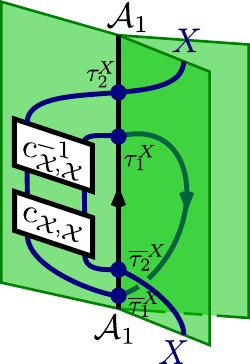} \nonumber 
			\\
			& = 
			\includegraphics[scale=1.0, valign=c]{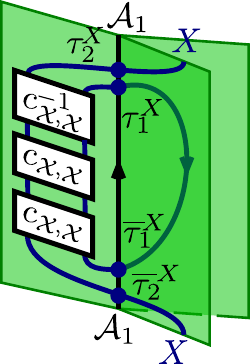} =
			\includegraphics[scale=1.0, valign=c]{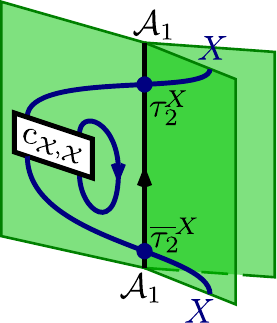}\, ,
			\label{eq:WA-ribbon_proof_step2}
			\end{align}
	which upon substituting back to \eqref{eq:WA-ribbon_proof_step1} yields the desired result. 
\end{proof}

\begin{remark}
  \label{rem:WAinGrayCat}
  \begin{enumerate}[label={(\roman*)}]
  \item \label{item:W-internal}
  The notion of a special orbifold datum can be formulated internal to an arbitrary Gray category with duals~$\mathcal T$, see \cite[Sect.\,4.2]{CRS1}. 
  In the case that $\mathcal T = \mathcal T_\zz$ from Section~\ref{subsubsec:SOD}, this reproduces Definition~\ref{def:special-orb-data}. 
  Our construction of a ribbon category~$\wa$ generalises to any special orbifold datum~$\A$ in a Gray category with duals~$\mathcal T$, by interpreting all the above 3-dimensional diagrams as Gray diagrams (see \cite{BMS,CMS}) of~$\mathcal T$.
  \item 
  It is illustrative to consider the following simplified version 
  (ignoring the duals) of the categorical construction in \ref{item:W-internal}:  Let~$\mathcal T$ be the delooping of the symmetric monoidal 2-category $\operatorname{Cat}$ with the cartesian product as monoidal structure. 
  That is, $\mathcal T$ has only one object, the 1-morphisms are categories, 2- and 3-morphisms are functors and natural transformations.
  Consider data  $\A = (\A_3,\A_2,\A_1,\A_0^\pm)$ as in Section~\ref{subsubsec:SOD}, subject to the axioms
  \eqrefO{1}--\eqrefO{3}. 
  This corresponds to a non-unital monoidal category $\A$.
  Now the analogue of $\mathcal W_{\A}$ has as objects tuples  $\mathcal X = (X,\tau_1^X, \tau_2^X)$ as in Definition~\ref{definition:WA}, subject to the axioms \eqrefT{1}--\eqrefT{5}. 
  Writing $\otimes$ for the monoidal product of $\A$, an object $F$ of $\mathcal W_{\A}$ is thus a functor $F\colon \A \lra \A$ with coherent isomorphisms $F(a \otimes b) \cong a \otimes F(b) \cong F(a) \otimes b$, i.\,e.\ $\mathcal W_{\A} \cong \operatorname{Fun}_{\A,\A}(\A,\A)$ is the category of bimodule endofunctors of~$\A$. 
  This is automatically unital, and in the case that~$\A$ has a unit object, it coincides with the Drinfeld centre of~$\A$. 
  In the case of spherical fusion categories, the full ribbon equivalence is proved by (tedious) direct computation in \cite[Thm.\,4.2]{MuleRunk}.
\end{enumerate}
\end{remark}

\subsection{Orbifold graph TQFTs}
\label{subsec:OrbifoldGraphTQFTs}

Let~$\zz$ be a defect TQFT, let~$\A$ be a special orbifold datum for the completed TQFT~$\widehat\zz$ of Section~\ref{subsec:MonoidalCatFromDefectTQFT}, and recall the associated ribbon category~$\wa$ of Section~\ref{subsec:RibbonCategoriesFromSOD}. 
In this section we extend the orbifold TQFT $\zzhata \colon {\Bord}_3 \lra \Vect$ of Construction~\ref{constr:OrbifoldTQFT} to a graph TQFT 
\be 
\zzwa \colon \Bordribn{3}(\wa) \lra \Vect
\ee 
on bordisms with embedded $\wa$-labelled ribbon graphs. 
To this end we use ribbon diagrams as in Section~\ref{subsuc:RibbonDiagramsOmegaMoves} to represent (uncoloured) ribbon graphs and then decorate them using the data from $\wa$.
To show that the construction of~$\zzwa$ is independent of the choice of such representations, we prove invariance under  $\omega$-moves of ribbon diagrams.

\subsubsection{Decorated ribbon diagrams}

Let $M$ be a bordism and $R$ an embedded ribbon graph in~$M$. 
In Section~\ref{subsuc:RibbonDiagramsOmegaMoves} the set $\mathscr S(M,R)$
of positive admissible ribbon diagrams in~$M$ that represent~$R$ was introduced.
By design, elements of  $\mathscr S(M,R)$ can be decorated using an orbifold datum $\A$ and the ribbon category $\wa$. 
This is formalised in Definition~\ref{def:AdecoratedRibbonDiagrams} below, which can be viewed as a generalisation of Definition~\ref{def:AdecoratedSkeleta} to non-trivial ribbon graphs. 

As in Section~\ref{subsec:Review3dDefectTQFT}, for a ribbon category~$\mathcal C$ we denote $\mathcal C$-coloured ribbon graphs by calligraphic letters like~$\mathcal R$ (whose underlying uncoloured ribbon graphs are then denoted~$R$). 
Accordingly, we write $\mathscr S(M,\mathcal R)$ for the set of $\mathcal C$-coloured positive admissible ribbon diagrams representing a $\mathcal C$-coloured ribbon graph~$\mathcal R$ in~$M$. 
An element $(S, \mathscr{d}) \in \mathscr S(M,\mathcal R)$ consists of an admissible skeleton~$S$ and a $\mathcal C$-coloured knotted plexus~$\mathscr d$, whose underlying uncoloured knotted plexus~$d$ obtains its $\mathcal C$-colouring from~$\mathcal R$. 
It follows that $\mathscr S(M,\varnothing) = \mathscr S(M)$. 

In the present setting, the ribbon category $\mathcal C$ used to colour $\mathcal R$ is given by $\wa$.

\begin{definition}
\label{def:AdecoratedRibbonDiagrams}
Let $\zz\colon \Borddefn3(\mathds{D}) \lra \Vect$ be a defect TQFT, and let~$\A$ be a special orbifold datum for~$\zzhat$. 
An \textsl{$\mathcal A$-decorated ribbon diagram} $(\mathcal S, \mathscr d)$ of a morphism $(M,\mathcal R)$ in $\Bordribn{3}(\wa)$ is an element $(S, \mathscr{d}) \in \mathscr S(M,\mathcal R)$ together with a decoration as follows: 
\begin{enumerate}[label={(\roman*)}]
	\item 
	$\mathcal S$ is an $\A$-decorated skeleton of~$M$ with underlying skeleton~$S$;
	\item 
	if a switch of $(\mathcal S, \mathscr d)$ involves an $\mathcal X$-labelled ribbon of~$\mathcal R$ traversing an $\A_1$-labelled 1-stratum of~$\mathcal S$, then the switch is labelled by $\tau^X_1, \tau^X_2, \overline\tau^X_1$ or~$\overline\tau^X_2$ as dictated by orientations, cf.~\eqref{eq:tau_i_crossing-def-1}, \eqref{eq:tau_i_crossing-def-2}; 
	\item 
	over- and under-crossings in~$\mathscr d$ are replaced by coupons labelled with the corresponding braiding morphisms in~$\wa$. 
	Each braiding coupon is oriented such that on its source
	side the two ribbons involved in the crossing are pointing towards the coupon, and on its target side they are pointing away form the coupon.	
\end{enumerate}
The set of $\A$-decorated ribbon diagrams of the pair $(M,\mathcal R)$ is denoted $\mathscr S_\A(M,\mathcal R)$. 
\end{definition}

Given an $\A$-decorated ribbon diagram $(\mathcal S, \mathscr d)$ of a morphism 
\be 
(M,\mathcal R) \colon \Sigma \lra \Sigma'
\ee
in $\Bordribn{3}(\wa)$, we obtain a new defect bordism $F(\mathcal S, \mathscr d)$, which, in accordance with Section~\ref{subsubsec:OrbifoldTQFTs}, we call a \textsl{foamification} of $(M,\mathcal R)$ represented by $(\mathcal S, \mathscr d)$. 
Note that $F(\mathcal S, \varnothing) = F(\mathcal S)$. 

Viewed as a morphism in $\Borddefn{3}(\widehat\D)$, the defect bordism $F(\mathcal S, \mathscr d)$ has source and target objects which are $\widehat\D$-decorated defect surfaces whose 1-strata are labelled by~$\A_2$ and whose 0-strata precisely correspond to the endpoints of~$\mathcal R$ and of $\A_1$-lines on $\partial M$. 
We denote these source and target objects by $F(\Sigma, \mathcal G)$ and $F(\Sigma', \mathcal G')$, respectively, where $\mathcal G, \mathcal G'$ are the decorated skeleta of $\Sigma, \Sigma'$ induced from $(\mathcal S, \mathscr d)$. 
Thus 
\be
\label{eq:FoamificationMap}
F(\mathcal S, \mathscr d) \colon F(\Sigma, \mathcal G) \lra F(\Sigma', \mathcal G')
\ee 
in $\Borddefn{3}(\widehat\D)$. 

By definition, the evaluation of~$\zzhat$ on an $\A$-decorated ribbon diagram $(\mathcal S, \mathscr d)$ is given by $\zzhat(F(\mathcal S, \mathscr d))$. 
In particular, $\zzhat$ can be evaluated on defect 3-balls around each side of $\A$-decorated versions of the $\omega$-moves in Figure~\ref{fig:omegaMoves}. 

\begin{proposition}
\label{prop:SODandOmegaMoves}
Let~$\mathcal A$ be a special orbifold datum for a completed defect TQFT~$\zzhat$. 
Applying $\zzhat$ to $\A$-decorated $\omega$-moves gives identities in $\Vect$. 
\end{proposition}
\begin{proof}
Invariance under $\omega_0$-moves follows from Proposition~\ref{prop:SODBLT}. 
Invariance under moves of type $\omega_1, \omega_2, \omega_3$ follows from the fact that they are framed Reidemeister moves which hold in every ribbon category. 

Invariance under the remaining $\omega$-moves follows from the defining properties of the category~$\wa$ and the results of \cite{MuleRunk} which directly carry over to our more general setting: 
	for one choice of admissible orientations of 2-strata, invariance under~$\omega_{10}$ follows from~\eqrefT{5} together with the identity \cite[(T12$'$)]{MuleRunk}; 
	for~$\omega_9$, use \cite[(T13$'$)]{MuleRunk}; 
	for~$\omega_8$ and~$\omega_6$, use~\eqref{eq:M1} and~\eqref{eq:M2}; 
	for~$\omega_7$, use \cite[Lem.\,3.4]{MuleRunk}; 
	for~$\omega_5$, use~\eqrefT{4}. 
Showing invariance under the $\omega_4$-move, which we reformulate as the identity 
\be 
\label{eq:omega4again}
\includegraphics[scale=1.0, valign=c]{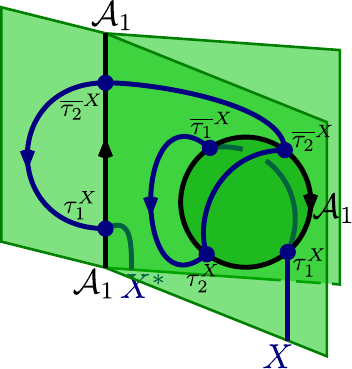}  
= 
\includegraphics[scale=1.0, valign=c]{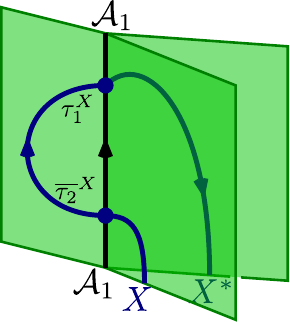}  
\, , 
\ee 
is more involved, and we give more details. 

Let us recall the identities \eqref{eq:aux_braiding_ids} which were used in the proof of Proposition \ref{prop:WA-ribbon-cat}.
Taking $X=Y$ and closing the left-most strands to a loop, we obtain
\be 
\label{eq:TwistNearT}
\includegraphics[scale=1.0, valign=c]{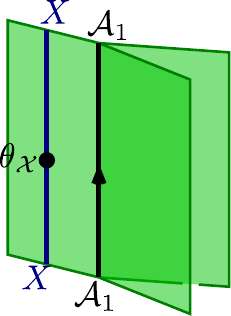}
=
\includegraphics[scale=1.0, valign=c]{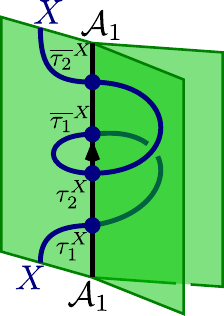}  
\, , \quad 
\includegraphics[scale=1.0, valign=c]{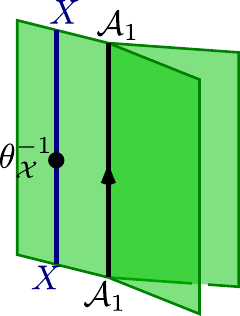}
=
\includegraphics[scale=1.0, valign=c]{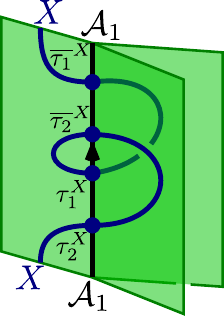}  
\, , 
\ee 
where we used the representations~\eqref{eq:TwistInWA}, \eqref{eq:TwistInvInWA} of the ribbon twist and its inverse. 
Hence the left-hand side of~\eqref{eq:omega4again} is 
\be 
\includegraphics[scale=1.0, valign=c]{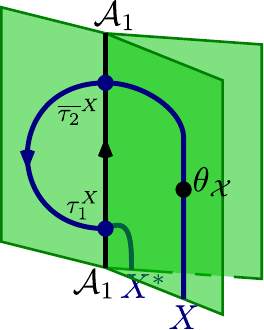}   
=
\includegraphics[scale=1.0, valign=c]{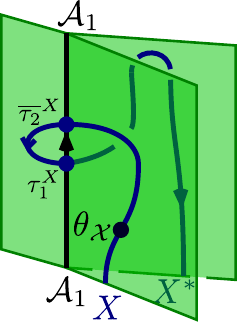}
= 
\includegraphics[scale=1.0, valign=c]{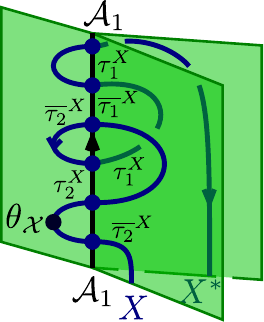}   
= 
\includegraphics[scale=1.0, valign=c]{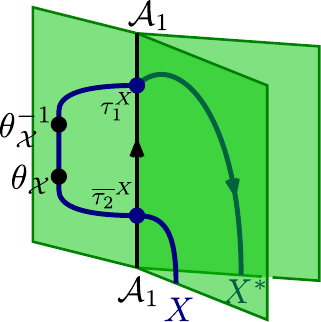}   
\, , 
\ee 
where in the first step we deformed the $X$-labelled line, in the second step we used \eqrefT{4}--\eqrefT{7} and a consequence of~\eqref{eq:M2}, and in the last step the second identity of~\eqref{eq:TwistNearT}.
\end{proof}

Together with Proposition~\ref{prop:omegaMoves}, the above result implies: 

\begin{corollary}
\label{cor:ZCAonSameBndManifolds}
Let $(\mathcal S, \mathscr d)$ and $(\mathcal S', \mathscr d')$ be $\A$-decorated representations of a $\wa$-coloured ribbon graph~$\mathcal R$ in a bordism~$M$. 
If $\mathcal S, \mathcal S'$ respectively $\mathscr d, \mathscr d'$ agree on the boundary, i.\,e.\ if $\mathcal S|_{\partial M} = \mathcal S'|_{\partial M}$ and $\mathscr d|_{\partial M}=\mathscr d'|_{\partial M}$, then 
\be
\zzhat\big(F(\mathcal S, \mathscr d)\big) = \zzhat\big(F(\mathcal S', \mathscr d')\big) . 
\ee
\end{corollary}

\subsubsection{Definition of the orbifold graph TQFT}

We finally define the orbifold graph TQFT $\zzwa \colon \Bordribn{3}(\wa) \lra \Vect$.  
According to Corollary~\ref{cor:ZCAonSameBndManifolds}, the linear map $\zzhat(F(\mathcal S, \mathscr d))$ depends only on the choice of decorated skeleton of the boundary of the bordism~$M$ with embedded ribbon graph~$\mathcal R$ represented by $(\mathcal S, \mathscr d)$. 

The dependence on the boundary decomposition is removed analogously to the construction of the orbifold TQFT~$\zza$ in Section~\ref{subsubsec:OrbifoldTQFTs}. 
Namely, let~$\Sigma$ be an object in $\Bordribn{3}(\wa)$. 
Each puncture~$p_i$ of~$\Sigma$ comes with a label $(\mathcal X_i,v_i,\varepsilon_i)$ as in Section~\ref{subsec:Review3dDefectTQFT}, where $\mathcal X_i \in \wa$. 
We view the cylinder $\Sigma \times [0,1]$ as a bordism $\mathcal C_\Sigma$ with embedded $\wa$-coloured ribbon graph~$\mathcal R_\Sigma$ that happens to consist only of straight ribbons labelled by the objects~$\mathcal X_i$. 
Any choice of $\A$-decorated ribbon diagram $(\mathcal S, \mathscr d)$ of $(\mathcal C_\Sigma, \mathcal R_\Sigma)$ gives rise to a linear map 
\be 
\Psi_{\mathcal G}^{\mathcal G'} := \zzhat\big( F(\mathcal S, \mathscr d) \big) \colon 
	\zzhat \big( F(\Sigma,\mathcal G) \big) \lra \zzhat \big( F(\Sigma,\mathcal G') \big) , 
\ee 
where $\mathcal G, \mathcal G'$ are the decorated skeleta of~$\Sigma$ induced by $(\mathcal S, \mathscr d)$ as in~\eqref{eq:FoamificationMap}. 

By Corollary~\ref{cor:ZCAonSameBndManifolds} the linear maps $\Psi_{\mathcal G}^{\mathcal G'}$ do not depend on the choice of $\A$-decorated ribbon diagram $(\mathcal S, \mathscr d)$ in the interior of~$\mathcal C_\Sigma$, and for arbitrary decorated admissible skeleta $\mathcal G, \mathcal G', \mathcal G''$ of~$\Sigma$ we have 
\be
\Psi_{\mathcal G'}^{\mathcal G''} \circ \Psi_{\mathcal G}^{\mathcal G'} = \Psi_{\mathcal G}^{\mathcal G''} \, . 
\ee 
In particular, each map $\Psi_{\mathcal G}^{\mathcal G}$ is an idempotent. 

\begin{construction}
\label{constr:OrbifoldGraphTQFT}
Let~$\mathcal A$ be a special orbifold datum for a completed defect TQFT~$\zzhat$. 
The \textsl{orbifold graph TQFT} 
\be 
\zzwa \colon \Bordribn{3}(\wa) \lra \Vect
\ee  
is defined as follows: 
\begin{enumerate}
	\item 
	For an object $\Sigma \in \Bordribn{3}(\wa)$, we set 
	\be 
	\zzwa (\Sigma) = \textrm{colim} \big\{ \Psi_{\mathcal G}^{\mathcal G'} \big\} , 
	\ee 
	where $\mathcal G, \mathcal G'$ range over all admissible $\A$-decorated skeleta of~$\Sigma$. 
	\item 
	For a morphism $(M,\mathcal R) \colon \Sigma \lra \Sigma'$ in $\Bordribn{3}(\wa)$, we set $\zzwa (M,\mathcal R)$ to be 
	\be 
	\label{eq:ZCAonMorphisms}
	\begin{tikzcd}
	\zzwa(\Sigma) \arrow[hook]{r} 
	& \zzhat \big( F(\Sigma, \mathcal G) \big) \arrow{rr}{\zzhat( F(\mathcal S, \mathscr d) )}
	& 
	& \zzhat\big(  F(\Sigma', \mathcal G') \big) \arrow[two heads]{r}
	& \zzwa(\Sigma') \, , 
	\end{tikzcd}  
	\ee 
	where $(\mathcal S, \mathscr d)$ is an arbitrary $\A$-decorated ribbon diagram representing $(M,\mathcal R)$. 
\end{enumerate}
\end{construction}

By Corollary~\ref{cor:ZCAonSameBndManifolds}
the definition of~$\zzwa(M,\mathcal R)$ in~\eqref{eq:ZCAonMorphisms} does not depend on the choice of admissible $\A$-decorated skeleton $(\mathcal S, \mathscr d)$. 
Moreover, by construction the state spaces of~$\zzwa$ are isomorphic to the images of the idempotents, 
\be 
\zzwa(\Sigma) \cong \textrm{Im} \big( \Psi_{\mathcal G}^{\mathcal G} \big) \, . 
\ee 

We have thus shown our main result, which is that~$\zzwa$ is indeed a graph TQFT: 

\begin{theorem}
\label{thm:MainResultForEulerComplete}
Let~$\mathcal A$ be a special orbifold datum for a completed defect TQFT~$\zzhat$. 
Then $\zzwa \colon \Bordribn{3}(\wa) \lra \Vect$ as in Construction~\ref{constr:OrbifoldGraphTQFT} is a symmetric monoidal functor. 
\end{theorem}

\medskip

There are few examples of special orbifold data for a generic 3-dimensional defect TQFT~$\zz$. 
However, if one passes to the so-called ``Euler completion''~$\zz^\odot$, one finds far more examples, cf.\ \cite{CRS2,CRS3,MuleRunk2}. 
In Appendix~\ref{app:ConstructionForNonEulerCompletedCase} we spell out the details of special orbifold data~$\A$ for~$\zz^\odot$ as well as $\A$-decorated skeleta and ribbon diagrams, the ribbon category~$\wa$, and the associated orbifold graph TQFT, all in terms of the non-completed TQFT~$\zz$. 
Conceptually, Appendix~\ref{app:ConstructionForNonEulerCompletedCase} offers nothing new, but the details are relevant for applications, in particular for the treatment in \cite{CMRSS2}.

\newpage
\appendix

\section{Proof of Theorem \ref{thm:ConnectSkeleta}} 
\label{app:A}

Here we prove that if two admissible skeleta of a 3-bordism agree on the boundary, then they are related by admissible BLT moves.

\subsection{Skeleta of 2-manifolds}

We use the notation introduced in Section~\ref{subsec:Skeleta} to define skeleta for closed 2-manifolds, in analogy to the 3-dimensional case. 

\begin{definition} \label{def:skeleta}
	Let $\Sigma$ be a closed 2-dimensional manifold. 
	A skeleton $S$ for $\Sigma$ is a stratification such that 
	the following additional requirements are satisfied:
	\begin{enumerate}
		\item 
		Every 2-stratum is an open disc.
		\item 
		Each $x \in \Sigma$ has a neighbourhood that is isomorphic to one of the following open stratified discs $B_x$: 
		\begin{enumerate}[label=\arabic*)]
			\item 
			If $x \in S^{(2)}$, then $B_x$ contains no 1- or 0-strata:
			\be	
			\includegraphics[scale=1.25, valign=c]{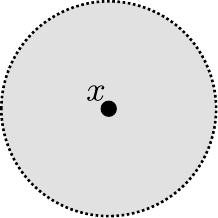}
			\, . 
			\ee 
			\item 
			If $x \in S^{(1)}$, then $B_x$ is given by
			\be	
			\includegraphics[scale=1.25, valign=c]{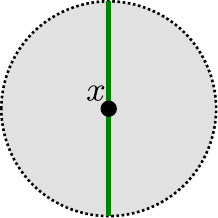}
			\, . 
			\ee 
			\item 
			If $x \in S^{(0)}$, then $B_x$ is given by
			\be
			\includegraphics[scale=1.25, valign=c]{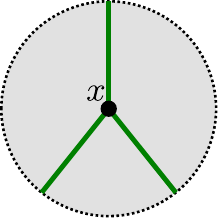}
			\, .
			\ee	
		\end{enumerate}	
	\end{enumerate}
\end{definition}	

Admissibility is defined via local orders,
analogously to the 3-dimensional case in Section~\ref{subsubsec:AdmissibleSkeleta}.
Moreover, the 2-dimensional analogues of the BLT moves are the bubble move and the dual of the 2-2 Pachner move, which we call the \textsl{b-move} and \textsl{l-move}, respectively:
\begin{align}
& \includegraphics[scale=1.25, valign=c]{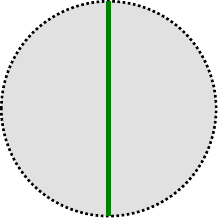} \longrightarrow
\includegraphics[scale=1.25, valign=c]{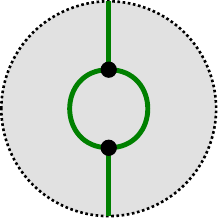}
\\ 
& \includegraphics[scale=1.25, valign=c]{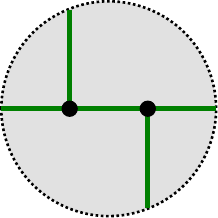} \longrightarrow
\includegraphics[scale=1.25, valign=c]{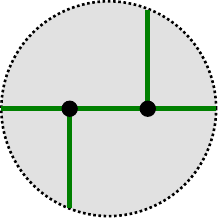}
\end{align} 
A bl move is called \textsl{admissible} if the skeleta on both sides are admissible.

The following theorem is the 2-dimensional version (without boundary) of the 3-dimensional statement that is the topic of this appendix. 
The 3-dimensional proof will be similar in structure. 

\begin{theorem}
	\label{thm:ClosedSurfacesAndblMoves}
	Let $\Sigma$ be a closed 2-manifold. 
	Then any two admissible skeleta~$S$ and~$S'$ of $\Sigma$ are connected by a finite sequence of admissible bl moves.
\end{theorem}
\begin{proof}
	We call two admissible skeleta equivalent if they are are connected by a finite sequence of admissible bl moves. 
	Given a globally ordered triangulation of $\Sigma$, its dual is a particular case of an admissible skeleton which we refer to as d-skeleton (short for ``skeleton dual to a globally ordered triangulation'').
	By \cite[Prop.\,3.3]{CRS1} any two d-skeleta are related by a finite sequence of moves that are dual to globally ordered 2-dimensional Pachner moves (see Section~\ref{subsubsec:BLTmoves}). 
	It is well known (see e.\,g.\ \cite{FHK}) that the admissible bl moves are stronger than the moves dual to globally ordered Pachner moves, thus any two d-skeleta are equivalent. 
	
	We are thus left with showing that an admissible skeleton $S$ is equivalent to a d-skeleton.
	Analogously to the 3-dimensional case in Lemma~\ref{lem:TriangCharact} below, one finds that a skeleton is dual to a triangulation iff
	every stratum~$s$ satisfies:
	(i) $s$ is contractible, 
	(ii) the germs of 2-strata adjacent to~$s$ belong to different 2-strata, and
	(iii) if the sets of germs of adjacent 2-strata agree for~$s$ and another stratum~$t$, then already $s=t$. 
	If conditions (i)--(iii) hold for an individual stratum~$s$, we say that~$s$ is locally dual to a triangulation.
	
	By the b-move the contractibility condition is easy to achieve starting from any admissible skeleton. 
	If the two germs of adjacent 2-strata of a 1-stratum belong to the same 2-stratum, by a sequence of one b- and two l-moves, we can create a ``copy of the 1-stratum'' such that in the resulting stratification the original 1-stratum and the newly created 1-stratum are locally dual to a triangulation, and all other strata remain locally dual to a triangulation if they are. 
	If a 0-stratum is not locally dual to a triangulation, it is straightforward to provide a similar combination of b- and l-moves to make it locally dual to a triangulation.
	
	We can thus assume that $S$ is dual to a triangulation and we now need to show it is equivalent to a d-skeleton, i.\,e.\ we need to fix the orientations. 
	We do this similar to the proof of Lemma~\ref{lemma:connect-adm-oriented} below: 
	Pass to the dual triangulation; carry out a 1-3 Pachner move on all triangles and orient the new edges towards the new vertices. 
	Now that both orientations of each old edge are allowed, reverse the orientations of the old edges as required (this is done via b- and t-moves in the dual picture); undo all the 1-3 Pachner moves.
\end{proof}

The next statement links the above discussion of the 2-dimensional skeleta to the 3-dimensional case that is the focus of this paper. 
Namely, it is straightforward to check that the following is true:

\begin{proposition}
	Any skeleton $S$ of a 3-manifold $M$ induces a skeleton $\del S$ of $\del M$, where the $i$-strata of $\del S$ are obtained by intersecting the $(i+1)$-strata of $S$ with $\del M$. 
	If $S$ is oriented or admissible then so is $\del S$. 
	\qed
\end{proposition}

\subsection{Pseudo-skeleta}

We define a slightly more general class of stratifications than skeleta that will be useful in the proof of Theorem \ref{thm:ConnectSkeleta}:
A \textsl{pseudo-skeleton} of a 3-manifold~$M$ is a stratification such that: 
		\begin{enumerate}
                \item 
                For each 3-stratum $s$ there are exactly three allowed cases: 
                if $s \cap \del M = \varnothing$, then~$s$ is a 3-ball; if~$s$ intersects exactly one of $\del_{\textup{in}}M$ and $\del_{\textup{out}} M$, then~$s$ is a half open ball; and if~$s$ intersects both $\del_{\textup{in}}M$ and $\del_{\textup{out}} M$, then~$s$ is the cylinder $D^{\circ} \times [0,1]$, where $D^\circ$ is the interior of the 2-disc.
			\item The same local conditions hold as in the definition of a skeleton.	
		\end{enumerate}
We say a stratified 3-bordism $(M,S)$ is \textsl{skeletal} if $S$ is a skeleton, and \textsl{pseudo-skeletal} if $S$ is a pseudo-skeleton.
A morphism of $\Bordstr_3$ (cf.\ Section~\ref{subsubsec:StratifiedBordisms}) is \textsl{\mbox{(pseudo-)}skeletal} if one (and thus all) of its representative stratified 3-manifolds are.

Note that in a pseudo-skeleton of~$M$, if a 3-stratum is a cylinder $D^{\circ} \times [0,1]$, its boundary-discs $D^{\circ} \times \{0\}$ and $D^{\circ} \times \{1\}$ are not allowed to both lie on $\del_{\textup{in}} M$ nor both on $\del_{\textup{out}} M$.

For pseudo-skeleta, 
Remark~\ref{rem:3CellsBoundary} gets replaced by the following observation:

\begin{lemma}
  A pseudo-skeleton $S$ of $M$ is a skeleton if and only if no 3-stratum of $S$ intersects  both the in- and out-boundary of $M$.  \qed
\end{lemma}

\subsection{Refinement to dual of a triangulation relative to boundary}

The main content of this section is Lemma \ref{lem:adm-triang}, which tells us that any admissible skeleton with sufficiently nice boundary can be refined to the dual of a triangulation using admissible BLT moves. 
We begin with a technicality that will be needed in its proof, namely how to delete a 2-stratum from a skeleton.

Recall the notation introduced in Section~\ref{sssec:stratMan}, and let~$M$ be a fixed 3-manifold from now on.
	Let $S$ be a stratification of $M$ with filtration $\varnothing\subset F^{(0)} \subset \ldots \subset F^{(3)}=M$, and let $s$ be a 2-stratum of $S$. We define the filtration $\varnothing\subset \tilde{F}^{(0)} \subset \ldots \subset \tilde{F}^{(3)}=M$ by $\tilde{F}^{(i)} := F^{(i)}\setminus(\bar{s} \cap S^{(i)}) = (F^{(i)} \setminus \bar{s}) \cup F^{(i-1)}$.
More concisely, $\tilde{F}$ is obtained from $F$ by deleting all strata that are contained in $\bar{s}$. 
In general, the filtration $\tilde F$ is not a stratification, but we have:

\begin{lemma} \label{lem:dels}
	Let $S$ be a skeleton of $M$.
	\begin{enumerate}
		\item 
		The filtration $\tilde{F}$ defines a stratification of~$M$, denoted $S \setminus \bar{s}$.
		\item 
		If in addition to (i), $s$ is contractible and
		different germs of 3-strata incident with $s$ are induced by different 3-strata such that not both of them intersect $\del_{\textup{in}}M$ or both $\del_{\textup{out}} M$, then $S \setminus \bar{s}$ is a pseudo-skeleton.    
		\item 
		If in addition to (i) and (ii), at least one of the 3-strata incident with $s$ does not intersect $\del M$, then $S \setminus \bar{s}$ is a skeleton.
	\end{enumerate}
\end{lemma}

\begin{proof}	
	For (i) we notice that clearly $\tilde{F}^{(i-1)} \subset \tilde{F}^{(i)}$ for all~$i$, so we only need to endow each $\tilde{S}^{(i)}$ with a smooth structure. 
	Since we are in dimension 3, a smooth structure on $\tilde{S}^{(i)}$ is uniquely determined once we understand $\tilde{S}^{(i)}$ to be a topological manifold.
	Elementary point-set manipulations show that $\tilde{S}^{(i)}= (S^{(i)}\setminus \bar{s}) \cup (S^{(i-1)}\cap \bar{s})$. 
	Now if $x \in S^{(i)}\setminus \bar{s}$, then~$x$ has a neighbourhood $U$ such that $U \cap \tilde{F}^{(i)}$ is homeomorphic to $\R^i$ since $S^{(i)}$ is a topological manifold. 
	If $x \in S^{(i-1)}\cap \bar{s}$, then the existence of such a neighbourhood follows from the existence of the neighbourhoods depicted in Figure \ref{fig:skeleta}.
	
	For (ii) we note that the additional assumption guarantees that after deleting~$s$, all 3-strata are still contractible.
	The other conditions necessary to make $S \setminus \bar{s}$ a pseudo-skeleton are checked straightforwardly.
	Part (iii) follows directly from the definitions. 
\end{proof}

If we want to delete a 2-stratum $s$ from an admissible skeleton, then there are restrictions on the orientations of strata of~$S$. 
Indeed, let~$t$ be an $(i-1)$-stratum in $\bar{s}$, $i\in \{1,2,3\}$. 
Then it is incident with two $i$-strata~$r_1$ and~$r_2$ of $S$ (with distinct germs at $t$, but~$r_1$ and~$r_2$ may be equal in $S$) that are not contained in $\bar{s}$. 
After deleting $\bar{s}$, $r_1$ and $r_2$ merge to form an $i$-stratum $r$ of $S\setminus \bar{s}$ and the local orders at $r_1$ and $r_2$ induce local orders at $r$.
If the two induced local orders at $r$ agree we say that the local orders at $r_1$ and $r_2$ are \textsl{compatible}. 
If this holds for all strata $t$ in $\bar{s}$ then we say that the orientations adjacent to $s$ are compatible. 
In that case we can canonically endow $S \setminus \bar{s}$ with an admissible orientation that is inherited from $S$.

Combining this with the conditions from Lemma \ref{lem:dels}, we arrive at the following definition:
Let~$S$ be an admissible skeleton of~$M$. 
A contractible 2-stratum~$s$ in~$S$ is called \textsl{superfluous} if 
\begin{enumerate}
		\item 
		different germs of 3-strata incident with $s$ are induced by different 3-strata,
		\item 
		at least one of the 3-strata incident with $s$ does not intersect $\del M$,
		\item 
		the orientations adjacent to $s$ are compatible.
\end{enumerate}
If $s$ is superfluous then the admissible structure of $S$ turns $S\setminus \bar{s}$ canonically into an admissible skeleton. 
The name ``superfluous'' is justified by the following fact.

\begin{lemma} \label{lem:delsBLT}
	Let $S$ be an admissible skeleton of $M$, and let~$s$ be a superfluous 2-stratum of~$S$. 
	Then there exists a sequence of admissible BLT moves between~$S$ and $S\setminus \bar{s}$.
\end{lemma}

\begin{proof}
	Let $B$ be a 3-stratum of $S$ that does not intersect $\del M$ and such that $s \subset \bar{B}$.
	Then the topological boundary $\Sigma$ of $B$ is a 2-sphere. 
	Noting that $s$ may originally be adjacent to several other 2-strata in~$\Sigma$ (see Figure~\ref{fig:flower} for an illustration), we use the L- and T- moves on $\bar{s}$, until $\bar{s}$ forms a bubble on a single 2-stratum of $\Sigma \setminus s$, at which point we can delete it with a B-move.
	We have to be careful that in each step the orientations adjacent to~$s$ remain compatible because this guarantees that after deleting~$s$ all the strata in $S \setminus \bar{s}$ carry the same orientation as they did originally.
	A quick check confirms that in each admissible BLT move that we apply, we can choose orientations such that this is indeed the case.
\end{proof}

\begin{figure}
	\centering
	\captionsetup[subfigure]{labelformat=parens}
	\begin{subfigure}[b]{0.48\textwidth}
		\centering
		\includegraphics[scale=1.25, valign=c]{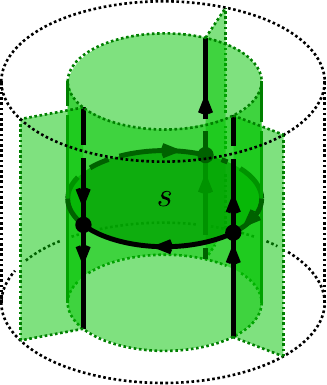}
		\caption{}
	\end{subfigure}
	\begin{subfigure}[b]{0.48\textwidth}
		\centering
		\includegraphics[scale=1.25, valign=c]{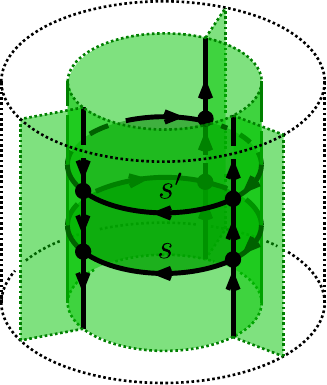}
		\caption{}
	\end{subfigure}
	\caption{(a) Neighbourhood of a 2-stratum~$s$ of an admissible skeleton~$S$.  
		(b) Neighbourhood of~$s$ and a shifted copy~$s'$ in the admissible skeleton~$S'$ (that agrees with~$S$ away from this neighbourhood).}
	\label{fig:flower}
\end{figure}

\begin{corollary}
	\label{cor:flower}
	Let~$S$ be an admissible skeleton of~$M$, let~$s$ be a contractible 2-stratum in~$S$ not intersecting $\partial M$, and let~$S'$ be the admissible skeleton which topologically differs from~$S$ only in that it has another copy~$s'$ of~$s$ which is connected to~$s$ by a cylinder (see Figure~\ref{fig:flower}). 
	Then~$S$ and~$S'$ are connected by admissible BLT moves. 
\end{corollary}

\begin{proof}
Note that $s'$ is superfluous. Start from the skeleton with $s'$ and use Lemma~\ref{lem:delsBLT}.
\end{proof}

\begin{remark}
	\label{rem:flip_orientation_flower}
	Note that if $s$ is superfluous then also the skeleton $S'$ obtained by reversing the orientation of $s$ is admissible (recall Remark~\ref{rem:ltovsorient}\,\ref{item:RemarkLocalOrder1}), and the orientation-reversed version of~$s$ is superfluous in $S'$.
	In this way, Lemma~\ref{lem:delsBLT} can be used to reverse the orientation of a superfluous $2$-stratum.

Combining this with Corollary~\ref{cor:flower}, it follows that if the skeleton $S'$ obtained from $S$ by flipping the orientation of $s$ is still admissible and $s$ does not touch $\del M$, then $S$ is connected to $S'$ by BLT moves even if $s$ is not superfluous (after making a copy $s'$ of $s$, $s$ becomes superfluous, and so its orientation can be flipped; then remove $s'$).
\end{remark}

We now give a characterisation of skeleta that are dual to triangulations.

\begin{lemma} 
	\label{lem:TriangCharact}
	A skeleton $S$ of $M$ is dual to a triangulation if and only if
	\begin{enumerate}
		\item every stratum of $S$ is contractible,
		\item for each stratum $s$ of $S$, the canonical map 
		\be 
		S_3(s) \lra \{\textrm{3-strata incident with } s\}
		\ee 
		is injective (and thus bijective), 
		\item if two strata $s$ and $t$ satisfy $S_3(s)=S_3(t)$, then $s=t$.
	\end{enumerate}
\end{lemma} 
\begin{proof}
	The conditions (i)--(iii) are satisfied for any dual of a triangulation.
	
	Conversely, from a skeleton~$S$ that satisfies these conditions, we obtain a simplicial complex, i.\,e.\ a set~$X$ (of vertices) together with a set $\Sigma \subset \mathcal{P}X$ (of simplices) such that for all $B \in \Sigma$ and $A \subset B$ we have $A \in \Sigma$ (face condition). 
	Indeed, first we set $X := S_3$, the set of 3-simplices of~$S$. 
	Using the local conditions from Figure~\ref{fig:skeleta} we see that each $i$-stratum~$s$ in~$S$ is incident to exactly $4-i$ germs of 3-strata around~$s$. 
	Thus we have maps $f_i\colon \{i\textrm{-strata of }\ S\} \too \{(4-i)\textrm{-element subsets of } X\}$ for each $i\in \{0,1,2,3\}$.
	The set~$\Sigma$ is the union of the images of the~$f_i$.
	Each~$f_i$ is injective, as follows directly from condition (iii).
	Because of this we obtain a map $|(X,\Sigma)| \too M$ which is a homeomorphism by definition of $(X,\Sigma)$. 
	This triangulation is dual to~$S$ by construction.
\end{proof}

\begin{lemma}
	\label{lem:adm-triang}
	Let $S$ be an admissible skeleton of $M$ such that $\del S$ is dual to a triangulation of $\del M$. Then there exists an admissible skeleton $S'$ of $M$ that is dual to a triangulation, satisfies $\del S' = \del S$, and is connected to $S$ by a sequence of admissible BLT moves.
\end{lemma}

\begin{proof}
	We show that we can pass to a skeleton satisfying the conditions from Lemma~\ref{lem:TriangCharact} using BLT moves that do not include the inverse lune move. 
	This allows us to always choose orientations of the targets such that the resulting moves will be admissible.
	
	To guarantee condition (i) we show that we can pass to sufficiently fine subdivisions of $S$. 
	We note that all 3-strata are contractible by definition of skeleta and 0-strata are so trivially.
	If $s$ is a 1-stratum we can cut it up by creating a bubble on a 2-stratum incident with it, and then sliding the bubble onto~$s$.     
	This can be implemented through a bubble and a lune move.
	By iterating this procedure we can guarantee that all 1-strata are contractible.
	Let now~$s$ be a 2-stratum.
	Since~$s$ is contained in the boundary of a 3-ball, it has genus~$0$.
	If~$s$ is closed then it is a sphere and we can decompose~$s$ into contractible pieces with a bubble move and the above argument for 1-strata.
	Otherwise the surface $\Sigma = \bar{s}$ is a sphere with a finite number of topological boundary components\footnote{Each such boundary component is an $S^1$ and can be made up from several 0- and 1-strata of $S$ or $\del S$.}.
	Each of the these boundary components contains some interval that is contained in the interior of $M$. 
	Indeed, suppose there was a boundary component $b$ with $b \subset \del M$. 
	Since all strata of $\del S$ are contractible, $b$ contains at least one 0-stratum $p$. 
	Then from the allowed neighbourhoods in Figure \ref{fig:skeleta} we see that there has to be a 1-stratum protruding from $p$ into the interior of $M$ that is part of the boundary of $\Sigma$. 
	This allows us to cut $s$ along a curve $\gamma$ connecting any two different boundary components by using a bubble and a lune move to create a bubble on $\gamma$ and stretching it along $\gamma$ using an isotopy. At the end we use two more lune moves to traverse the 1-strata connected by $\gamma$.
	This guarantees that~$s$ can be subdivided into contractible strata.
	All 1-strata created in the process are contractible.
        
	We now suppose that $S$ satisfies condition (i).
	Conditions (ii) and (iii) are automatically satisfied if any of the involved strata have non-trivial intersection with $\del M$ because of our assumption on $\del S$. 
	Hence in what follows we can assume that all the strata we consider do not intersect $\del M$.
	We first notice that condition (ii) is trivially satisfied for all 3-strata. 
	Furthermore if it is satisfied by all 2-strata, then also by all 1- and 0-strata.
	Indeed, this follows from the local conditions in Figure \ref{fig:skeleta}. 
	Thus let $s$ be a 2-stratum. 
	By Corollary \ref{cor:flower} we can make a copy $s'$ of $s$, guaranteeing that both $s$ and $s'$ satisfy condition (ii).
	
	We now suppose that $S$ satisfies conditions (i) and (ii), and we want to pass to a skeleton that additionally satisfies condition (iii).
	If $s,t$ are strata and $S_3(s) = S_3(t)$ then they have to be of the same dimension $i$, since for an $i$-stratum $s$ we have $|S_3(s)| = 4-i$.
	If $i=2$ and $S_3(s) = S_3(t)$ for $s\neq t$, then we can make a copy $s'$ of $s$ as in Corollary~\ref{cor:flower}, such that both $s$ and $s'$ are adjacent to the newly created 3-stratum while $t$ is not.
	A quick check confirms that this guarantees that after these moves, we have $S_3(s) \neq S_3(t)$, and that for all 2-strata $r\neq s'$, we have $S_3(s') \neq S_3(r)$. 
	If $i=1$ and $s\neq t$,  then~$s$ and~$t$ cannot be incident to a common 0-stratum because this 0-stratum would have a neighbourhood as depicted in Figure~\ref{fig:skeleta}, where all 1-strata have different sets of germs of 3-strata.
	We can then implement
	\be
	\includegraphics[scale=1.25, valign=c]{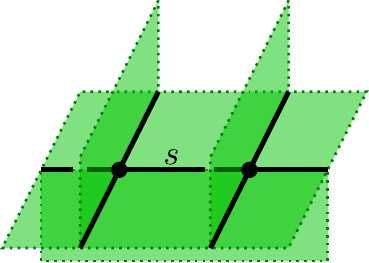} \longrightarrow
	\includegraphics[scale=1.25, valign=c]{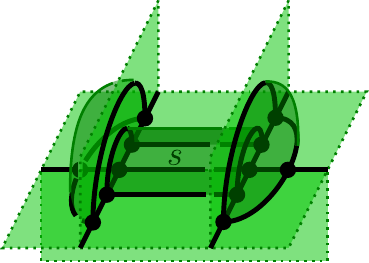}
	\ee
	around $s$ (which can be done without using inverse lune moves), guaranteeing that $S_3(s) \neq S_3(t)$. 
	Also none of the newly created 1- or 2-strata will violate condition (iii).
	Finally, if $i=0$ we use
	\be
	\includegraphics[scale=1.25, valign=c]{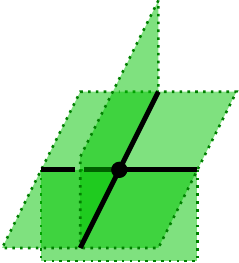} \longrightarrow
	\includegraphics[scale=1.25, valign=c]{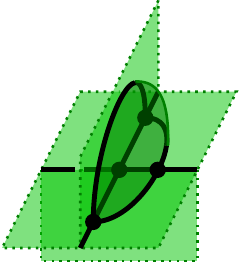} ~.
	\ee
	In none of these steps conditions (i) or (ii) are violated. 
\end{proof}

\subsection{Fixing the orientations}

Our next goal is to show that an admissible skeleton that is dual to a triangulation is connected to one that is dual to a globally ordered triangulation by admissible BLT moves.

\begin{lemma}
	\label{lemma:connect-adm-oriented}
	Let $S,S'$ be admissible skeleta of~$M$ such that
	\begin{enumerate}
		\item 
		the underlying unoriented skeleta agree, i.\,e.\ $\underline{S} = \underline{S'}$,
		\item 
		$S$ and $S'$ are dual to triangulations,
		\item 
		$S$ and $S'$ agree on the boundary of~$M$, i.\,e.\ $\del S = \del S'$.
	\end{enumerate}
	Then there is a sequence of admissible BLT moves from~$S$ to~$S'$.
\end{lemma}

\begin{proof}
	Let $T,T'$ be the triangulations dual to $S,S'$, respectively.
	By assumption we know that they agree as unoriented triangulations. Since they are admissibly oriented the orientations of all simplices are implied once the orientations of edges are fixed. The admissibility condition says that there is no closed loop of oriented edges that lies in the boundary of a single 3-simplex. 
	
	If~$e$ is an edge of~$T$ (or~$T'$), then we say that~$e$ can be \textsl{admissibly flipped} if the oriented triangulation $T(e^*)$ (or $T'(e^*)$) that is obtained by reversing the orientation of $e$ is admissible. By Remark~\ref{rem:flip_orientation_flower}, if~$e$ can be admissibly flipped then~$T$ and $T(e^*)$ (or~$T'$ and $T'(e^*)$) are connected by a sequence of admissible BLT moves.  
	
	We will construct triangulations $\widetilde T, \widetilde T'$ such that 
	\begin{itemize}
		\item 
		there are sequences of admissible BLT moves from~$T$ to~$\widetilde T$, and from~$T'$ to~$\widetilde T'$, 
		\item 
		the 1-skeleta of $T$ and~$T'$ embed into $\widetilde T$ and $\widetilde T'$, respectively (and we identify vertices and edges of~$T$ or~$T'$ with their images under these embeddings),
		\item 
		$\underline{\widetilde T} = \underline{\widetilde T'}$, and $\widetilde T$ differs from $\widetilde T'$ only in the orientations of edges of $\underline{T}=\underline{T'}$,
		\item 
		each edge of $T$ can be admissibly flipped in $\widetilde T$, and each edge of~$T'$ can be admissibly flipped in $\widetilde T'$. 
	\end{itemize}
	From this the claim follows because we then have a sequence of admissible BLT moves $T \too \widetilde T \too \widetilde T' \too T'$.
	
	By Lemma~\ref{lem:PachnerFromBLT} all admissible Pachner moves can be implemented via admissible BLT moves. 
	We first construct an intermediate triangulation $\bar T$ (and analogously $\bar T'$): 
	we apply an admissible 1-4 move to each 3-simplex of~$T$ and orient each newly created edge towards the newly created vertex it is incident with. 
	Then the 2-skeleton of $T$ embeds into $\bar T$. 
	Furthermore, if $s$ is any 2-simplex of $T$ in $\bar T$ that does not lie in the boundary, it is incident to exactly two 3-simplices all of whose other 2-faces are not contained in~$T$. 
	Thus we can apply 2-3 moves at each of the 2-simplices of~$T$ in~$\bar T$ (and the two adjacent 3-simplices in~$\bar T$) to arrive at the triangulation $\widetilde T$. 
	We orient the edges newly created by the 2-3 moves arbitrarily. 
	The choices of orientations in the construction of $\bar T$ (and $\bar T'$) guarantee that all these moves are admissible. 
	By design, all non-boundary 2-simplices of $T,T'$ have been erased in $\widetilde T, \widetilde T'$, and so all edges of~$T$ can be admissibly flipped in~$\widetilde T$, and analogously all edges of~$T'$ can be admissibly flipped in~$\widetilde T'$
	Thus all four conditions from above are satisfied.	
\end{proof}

\subsection{Fixing the boundary}

We provide the technical tools to deal with a boundary that is not dual to a triangulation.    

In the following we will denote (representatives of) morphisms of $\Bordstr_3$ by pairs $(M,S)$, or simply by~$M$, if there is no need to refer to the stratification~$S$ of the bordism~$M$. 
Composition in $\Bordstr_3$ is denoted by juxtaposition. 
           
\begin{lemma}
	\label{lem:compskel}
	Let $M,N$ be pseudo-skeletal morphisms of $\Bordstr_3$. Then the following statements hold whenever the respective compositions make sense.
	\begin{enumerate}
		\item $M N$ is pseudo-skeletal.
		\item If at least one of $M$, $N$ is skeletal then so is $M N$.
	\end{enumerate}
\end{lemma}

\begin{proof}
	Both claims follow straightforwardly from the definitions. 
	For part~(ii) we note that if a 3-stratum intersects both the in- and out-boundary of $MN$, then both of its restrictions to~$M$ and~$N$ must also intersect the respective in- and out-boundaries non-trivially.
\end{proof}

Two morphisms $(M,S)$ and $(M,S')$ of $\Bordstr_3$ with the same underlying bordism $M$ but possibly different stratifications $S$ and $S'$ are called \textsl{equivalent}, $(M,S)\sim(M,S')$, if $S$ and $S'$ are related by a sequence of admissible BLT moves. 

The resulting equivalence relation on the $\Hom$ sets of $\Bordstr_3$ is compatible with composition in the following way:

\begin{lemma}
	\label{lem:BLTEquivalenceAndComposition}
	If $M \sim N$ and $X \sim Y$, then $M X \sim N Y$. \qed
\end{lemma}

Let $\Sigma$ be a closed 2-manifold, $S$ a skeleton of $\Sigma$, and $s$ a 2-stratum of $S$. 
We construct a pseudo-skeletal 3-bordism $e_{\Sigma, S, s}\colon (\Sigma,S) \too (\Sigma,S)$ as follows:
Start with the cylinder $(\Sigma,S) \times I$. 
In the cylinder $I\times s$ we insert a single new 2-stratum $s'$ that is a copy of $s$, shifted away from $s$ along $I$ and a respective copy of each 0- and 1-stratum that is incident with $s$ in $\Sigma$. 
For example: 
\be
\label{eq:e_Sig-S-s}
	(\Sigma,S) = \includegraphics[scale=1.25, valign=c]{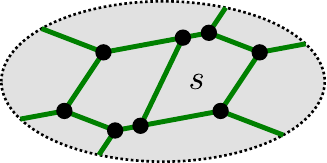} \;\Longrightarrow\;
	e_{\Sigma,S,s} = \includegraphics[scale=1.25, valign=c]{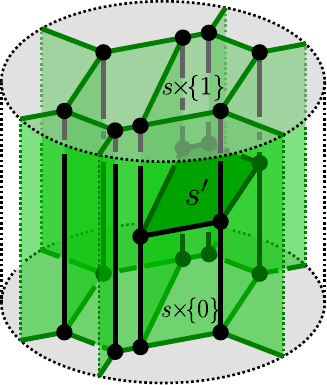} ~.
\ee

\begin{lemma} \label{lem:e}
	Let $\Sigma$ be a 2-manifold with skeleton $S$, let~$s$ be a 2-stratum of $S$, and let $M,N$ be skeletal 3-bordisms. 
	Then we have $e_{\Sigma,S,s}  M \sim M$ and $N  e_{\Sigma,S,s} \sim N$ whenever these compositions make sense.
\end{lemma}

\begin{proof}
	This is a direct application of Lemma \ref{lem:delsBLT}. 
\end{proof}

Note that this is in general not true if~$M$ or~$N$ are only pseudo-skeletal, as can be seen by taking them to be an identity in $\Bordstr_3$.

Any bl move for skeleta of 2-manifolds can be implemented via a pseudo-skeletal 3-bordism. 
More precisely: 
Let $\Sigma$ be a 2-manifold, $S$ a skeleton for $\Sigma$, and let $S'$ be obtained from $S$ via an application of a 2-dimensional bubble move in a disc $D  \subset \Sigma$. 
Then we define the pseudo-skeletal 3-bordism $M^{\textrm{b}}_{\Sigma,S,D}\colon (\Sigma, S) \too (\Sigma, S')$ by 
\be
M^{\textrm{b}}_{\Sigma,S,D} = 
\includegraphics[scale=1.25, valign=c]{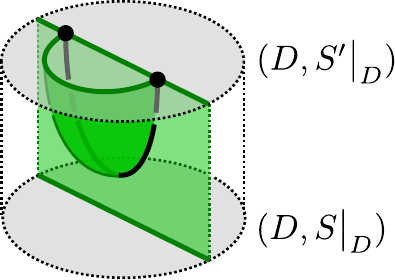}
\ee
which away from $D \times I$ is just the cylinder over $\Sigma \setminus D$.
Similarly, if $S'$ is obtained from $S$ via the dual of a 2-2 Pachner move we define the pseudo-skeletal 3-bordism $M^{\textrm{l}}_{\Sigma,S,D}\colon (\Sigma, S) \too (\Sigma, S')$ by
\be
	M^{\textrm{l}}_{\Sigma,S,D} =
	\includegraphics[scale=1.25, valign=c]{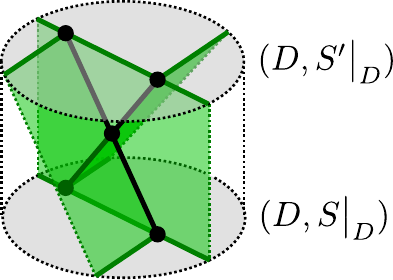}
\ee 
The bordisms $M^{\textrm{b}^{-1}}_{\Sigma,S,D}$ and $M^{\textrm{l}^{-1}}_{\Sigma,S,D}$ for the inverse moves $\textrm{b}^{-1}$ and $\textrm{l}^{-1}$ are defined analogously.

We can also implement all admissible 2-dimensional moves in this way. 
The admissible orientations on the corresponding 3-bordisms are uniquely determined by the ones on the boundary. 
More precisely, let $(M,S)$ be any of the 3-bordisms $M^{\textrm{b}^{\pm 1}}_{\Sigma,S,D}$, $M^{\textrm{l}^{\pm 1}}_{\Sigma,S,D}$ above,  and let $s$ be a 2-stratum of $S$. Then $S_3(s) \cong (\del S)_2(s \cap \del M)$.
The latter comes equipped with an order by assumption and this then defines the order on $S_3(s)$. 
A case-by-case check shows that this indeed defines a local order on $S$ in the sense of Definition~\ref{def:LocalOrder}.

\begin{lemma}
	\label{lem:Composing3dblMoves}
	Dropping sub indices we have
	\begin{enumerate}
		\item 
		$M^{\textrm{b}} \circ M^{{\textrm{b}}^{-1}} \sim e$, with $e$ as in \eqref{eq:e_Sig-S-s},	
		\item 
		$M^{{\textrm{b}}^{-1}} \circ M^{\textrm{b}} \sim 1$,
		\item 
		$M^{\textrm{l}} \circ M^{{\textrm{l}}^{-1}} \sim 1$,
		\item 
		$M^{{\textrm{l}}^{-1}} \circ M^{\textrm{l}} \sim 1$.
	\end{enumerate}
\end{lemma}

\begin{proof}
	The equivalence of part~(i) is implemented by a 3-dimensional lune move and an isotopy: 
	\be
		\includegraphics[scale=1.25, valign=c]{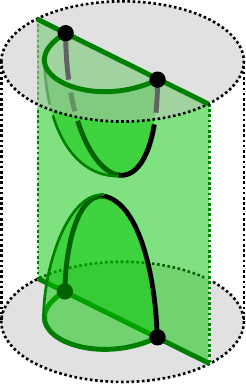} ~ \rightarrow ~
		\includegraphics[scale=1.25, valign=c]{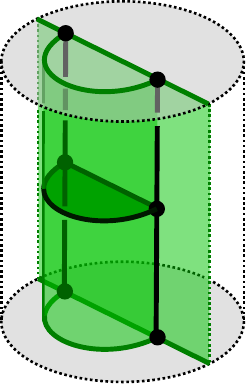} ~ .
	\ee
	Part~(ii) involves an inverse bubble move.
	Part~(iii) comes about with an inverse lune move, 
		\be
			\includegraphics[scale=1.25, valign=c]{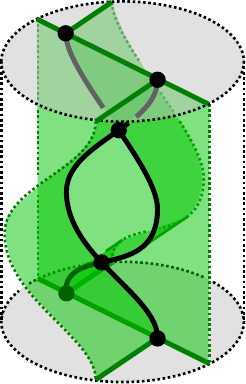} ~ \rightarrow ~
			\includegraphics[scale=1.25, valign=c]{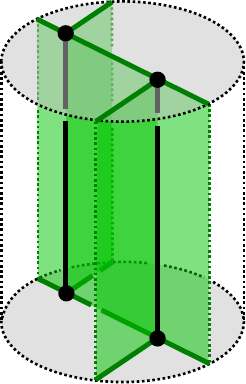} ~ ,
		\ee
	and similarly for part~(iv). 
\end{proof}

\subsection{The proof of Theorem \ref{thm:ConnectSkeleta}}

\begin{proof}
	Let $S,S'$ be admissible skeleta of~$M$ that agree on~$\partial M$. 
	Then we can find $X_1, \ldots, X_n$ and $Y_1,\ldots,Y_m$, where each of the $X_i, Y_j$ is one of the 3-bordisms $M^{\pm {\textrm{b}}}_{\cdots}, M^{\pm {\textrm{l}}}_{\cdots}$, 
	such that with $X:=X_1\cdots X_m$ and $Y:=Y_1\cdots Y_m$, $X (M,S) Y$ and $X (M,S') Y$ have the same boundary that is dual to a globally ordered triangulation
	(Theorem~\ref{thm:ClosedSurfacesAndblMoves}). 
	
	By Lemmas \ref{lem:adm-triang} and \ref{lemma:connect-adm-oriented} there are skeletal 3-bordisms $(M,R)\sim(M,S)$, $(M,R')\sim(M,S')$ such that $R,R'$ are dual to globally oriented triangulations.
	As described in \cite[Sect.\,3.1]{CRS1} $R$ and~$R'$ are connected by a sequence of globally oriented Pachner moves.
	By Lemma \ref{lem:PachnerFromBLT} these skeleta are a fortiori related by admissible BLT moves. 
	Hence with Lemma~\ref{lem:BLTEquivalenceAndComposition} we find $X (M,S) Y \sim X (M,S') Y$, and thus: 
	\begin{equation}
		(M,S) \overset{(*)}\sim X^{-}X (M,S) YY^{-}
		 \sim X^{-} X (M,S') Y Y^{-}   \overset{(*)}\sim (M,S') \, . 
	\end{equation}
	Here $X^{-}$ is obtained from $X$ by reversing order of composition while swapping~b with ${\textrm{b}}^{-1}$, and~l with ${\textrm{l}}^{-1}$.
	In $(*)$ we need Lemmas~\ref{lem:compskel} and~\ref{lem:e}:
	If $X_i^{-}X_i \sim 1$, there is nothing to do; but if $X_i^{-}X_i \sim e$, then we know that $X_{i+1}\cdots X_n (M,S)$ is skeletal, and thus together with Lemma~\ref{lem:Composing3dblMoves}: $X_i^{-}X_iX_{i+1}\cdots X_n (M,S) \sim e X_{i+1}\cdots X_n (M,S) \sim X_{i+1}\cdots X_n (M,S)$. 
	In this way we obtain $(X^{-}X) (M,S) \sim (M,S)$ by induction. 
	The same applies for $Y$.
\end{proof}

\section{Orbifold data for non-Euler-complete theories} 
\label{app:ConstructionForNonEulerCompletedCase}

We fix a defect TQFT $\zz\colon \Borddefn3(\mathds{D}) \lra \Vect$. 
Recall from \cite[Sect.\,2.5]{CRS1} the construction of its Euler completion $\zz^\odot \colon \Borddefn3(\mathds{D}^\odot) \lra \Vect$, which has the following properties: 
(i) $\zz$ naturally factors through~$\zz^\odot$; 
(ii) $(\zz^\odot)^\odot$ is equivalent to $\zz^\odot$; and 
(iii) the tensor product of~$\zz^\odot$ with the Euler defect TQFT $\zz^{\textrm{Eu}}_\Psi$ is equivalent to~$\zz^\odot$. 
Here for any list $\Psi = (\psi_1,\psi_2,\psi_3)$ of invertible scalars, $\zz^{\textrm{Eu}}_\Psi$ is the invertible defect TQFT which assigns~$\Bbbk$ to every surface, and 
$
\zz^{\textrm{Eu}}_\Psi(M) = \prod_{j=1}^{3} \prod_{s\in M_j} \psi_j^{\chi_{\textrm{sym}}(s)}
$ 
for a stratified bordism~$M$, where 
\be 
\chi_{\textrm{sym}}(-) := 2\chi(-) - \chi(\partial-) 
\ee 
is the \textsl{symmetric Euler characteristic} (see \cite[Ex.\,2.14]{CRS1} for details). 
For example, if $D$ is a 2-stratum consisting of a half-disc that intersects the boundary in an interval $I$, then $\chi(D)=\chi(I)=1$, and so $\chi_{\textrm{sym}}(D)=1$.

\subsubsection*{Special orbifold data}

As explained in \cite[Sect.\,3.4.1\,\&\,4.2]{CRS1}, a \textsl{special orbifold datum} for $\zz^\odot$ is a tuple 
\be 
\A = (\A_3, \A_2, \A_1, \A^\pm_0, \psi, \phi)
\, , 
\ee 
where $\A_3, \A_2, \A_1, \A^\pm_0$ are elements as in Section~\ref{subsubsec:SOD}, and 
\be 
\phi \in \textrm{Aut}_{\mathcal W(\A_3,\A_3)}(1_{1_{\A_3}}) 
\, , \qquad 
\psi \in \textrm{Aut}_{\mathcal W(\A_3,\A_3)}({1_{\A_3}}) \, , 
\ee 
such that the identities depicted in Figure~\ref{fig:SpecialOrbifoldDataPhiPsi} hold.

\begin{figure}
	\captionsetup[subfigure]{labelformat=empty}
	\centering
	\vspace{-50pt}
	\begin{subfigure}[b]{0.5\textwidth}
		\centering
		\includegraphics[scale=0.85, valign=c]{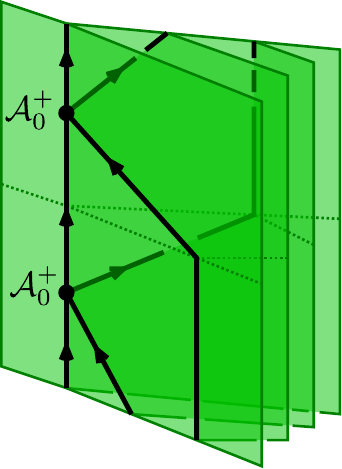} $=$
		\includegraphics[scale=0.85, valign=c]{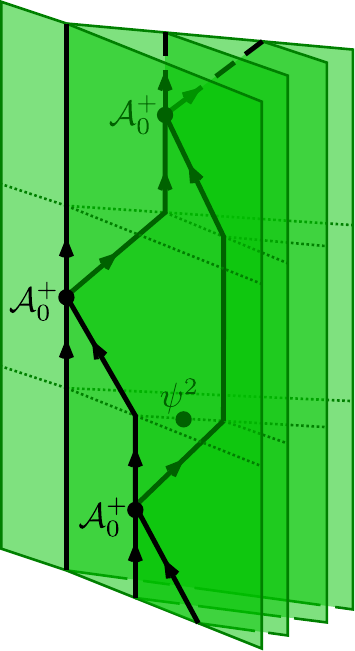}
		\caption{}
		\label{eq:OPSI1}
	\end{subfigure}\raisebox{8em}{(O${}_\psi$1)}\\
	\vspace{-15pt}
	\hspace{-60pt}
	\begin{subfigure}[b]{0.53\textwidth}
		\centering
		\includegraphics[scale=0.85, valign=c]{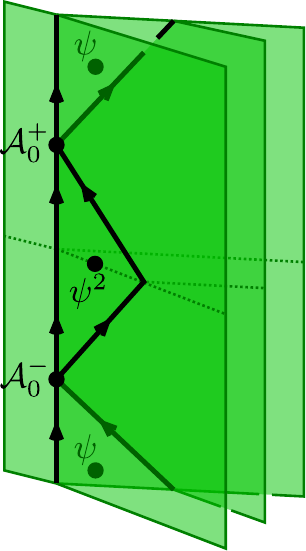} $=$
		\includegraphics[scale=0.85, valign=c]{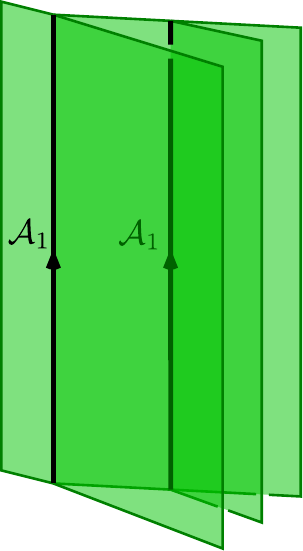}
		\caption{}
		\label{eq:OPSI2}
	\end{subfigure}\hspace{-2em}\raisebox{6.5em}{(O${}_\psi$2)}
	\hspace{-15pt}
	\begin{subfigure}[b]{0.53\textwidth}
		\centering
		\includegraphics[scale=0.85, valign=c]{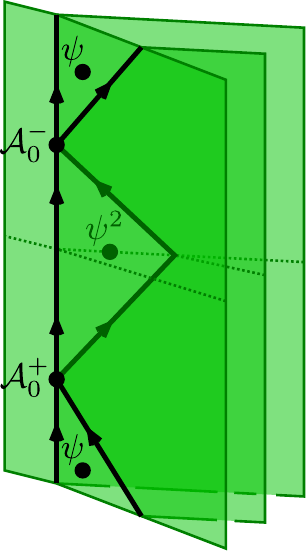} $=$
		\includegraphics[scale=0.85, valign=c]{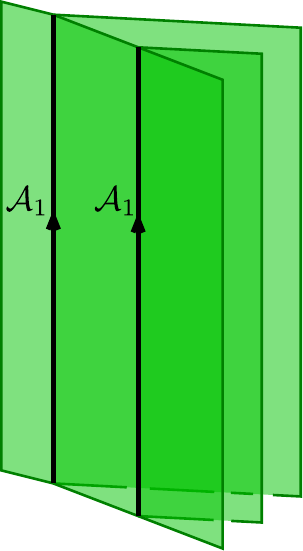}
		\caption{}
		\label{eq:OPSI3}
	\end{subfigure}\hspace{-2em}\raisebox{6.5em}{(O${}_\psi$3)}\\
	\vspace{-15pt}
	\hspace{-60pt}
	\begin{subfigure}[b]{0.53\textwidth}
		\centering
		\includegraphics[scale=0.85, valign=c]{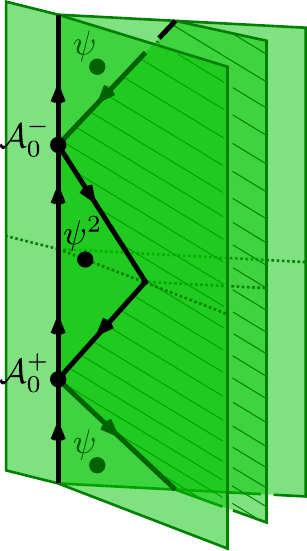} $=$
		\includegraphics[scale=0.85, valign=c]{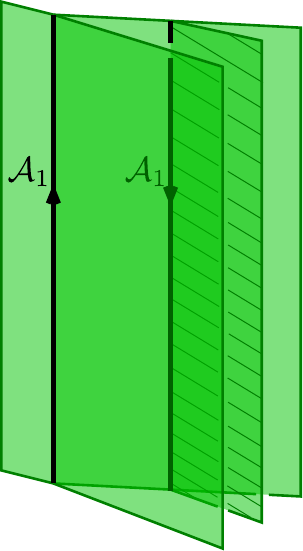}
		\caption{}
		\label{eq:OPSI4}
	\end{subfigure}\hspace{-2em}\raisebox{6.5em}{(O${}_\psi$4)}
	\hspace{-15pt}
	\begin{subfigure}[b]{0.53\textwidth}
		\centering
		\includegraphics[scale=0.85, valign=c]{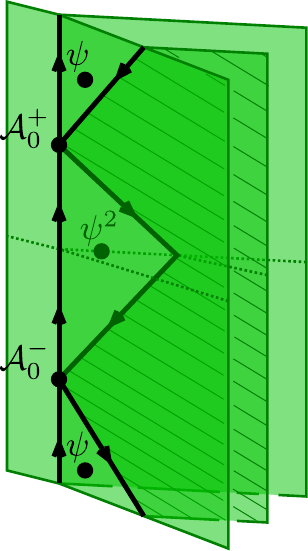} $=$
		\includegraphics[scale=0.85, valign=c]{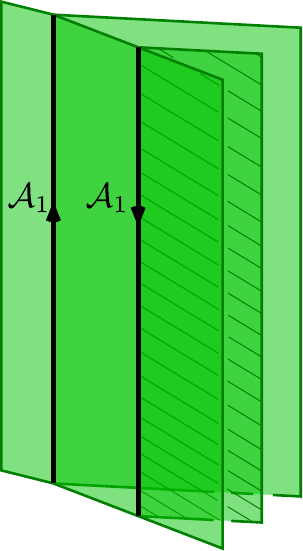}
		\caption{}
		\label{eq:OPSI5}
	\end{subfigure}\hspace{-2em}\raisebox{6.5em}{(O${}_\psi$5)}\\
	\vspace{-15pt}
	\hspace{-60pt}
	\begin{subfigure}[b]{0.53\textwidth}
		\centering
		\includegraphics[scale=0.8, valign=c]{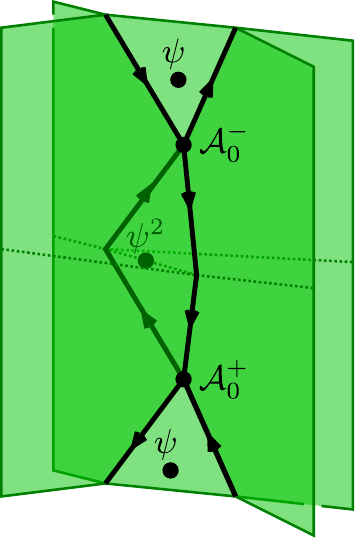} $=$
		\includegraphics[scale=0.8, valign=c]{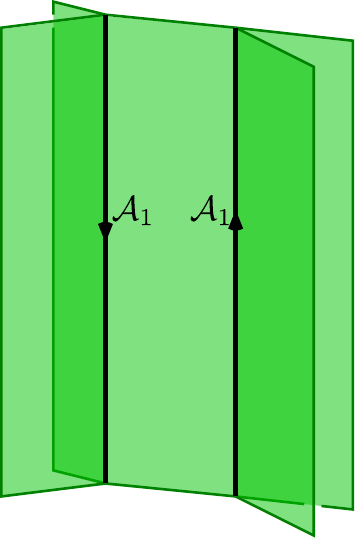}
		\caption{}
		\label{eq:OPSI6}
	\end{subfigure}\hspace{-1.5em}\raisebox{6em}{(O${}_\psi$6)}
	\hspace{-15pt}
	\begin{subfigure}[b]{0.53\textwidth}
		\centering
		\includegraphics[scale=0.8, valign=c]{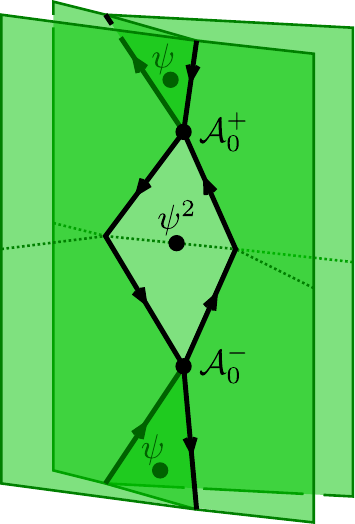} $=$
		\includegraphics[scale=0.8, valign=c]{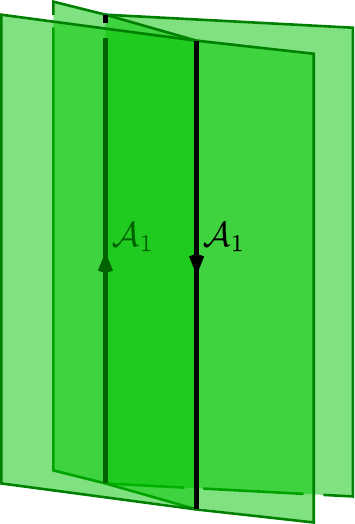}
		\caption{}
		\label{eq:OPSI7}
	\end{subfigure}\hspace{-1.5em}\raisebox{6em}{(O${}_\psi$7)}\\
	\vspace{-15pt}
	\hspace{-60pt}
	\begin{subfigure}[b]{1.0\textwidth}
		\centering
		\includegraphics[scale=0.8, valign=c]{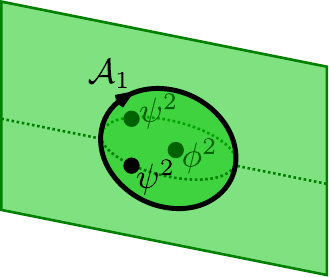} $=$
		\includegraphics[scale=0.8, valign=c]{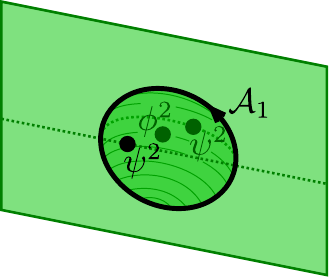} $=$
		\includegraphics[scale=0.8, valign=c]{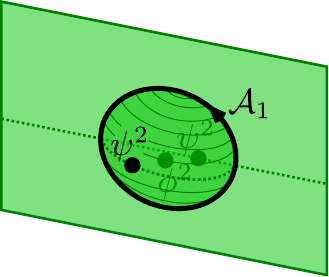} $=$
		\includegraphics[scale=0.8, valign=c]{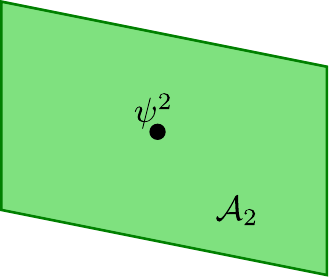}
		\caption{}
		\label{eq:OPSI8}
	\end{subfigure}\hspace{-2em}\raisebox{4em}{(O${}_\psi$8)}
	\vspace{-15pt}
	\caption{Defining conditions on special orbifold data~$\A$ with $\phi$ and $\psi$; the labels~$\A_j$ are suppressed for most $j$-strata.
	The application of $\zz$ on each side the equations is implied.	
}
	\label{fig:SpecialOrbifoldDataPhiPsi}
\end{figure}

\subsubsection*{$\boldsymbol{\A}$-decorated skeleta}

Let $\A = (\A_3, \A_2, \A_1, \A^\pm_0, \psi, \phi)$ be a special orbifold datum for~$\zz^\odot$. 
An \textsl{$\mathcal A$-decorated skeleton} $\mathcal S$ of a bordism~$M$ is an admissible skeleton~$S$ of~$M$ together with a decoration as follows: 
\begin{itemize}
	\item 
	each 3-stratum~$B$ of~$S$ is decorated by~$\A_3$ with an insertion of $\phi^{\chi_{\textrm{sym}}(B)}$, 
	\item 
	each 2-stratum~$F$ of~$S$ is decorated by~$\A_2$ with an insertion of $\psi^{\chi_{\textrm{sym}}(F)}$, 
	\item 
	each 1-stratum of~$S$ is decorated by~$\A_1$, 
	\item 
	for $\varepsilon \in \{+,-\}$, each $\varepsilon$-oriented 0-stratum of~$S$ is decorated by~$\A^\varepsilon_0$. 
\end{itemize}

\subsubsection*{Ribbon category $\boldsymbol{\wa}$}

The category~$\wa$ is defined as in Section~\ref{subsec:RibbonCategoriesFromSOD}, except for the following changes: 
\begin{itemize}
	\item 
	Objects of~$\wa$ are tuples $\mathcal X = (X,\tau_1^X, \tau_2^X, \overline{\tau}_1^X, \overline{\tau}_2^X)$, with $X\in\mathcal W$, where the crossings $\tau_1^X, \tau_2^X$ are as in~\eqref{eq:Tcrossings}, and where in addition
	\begin{align}
	\overline{\tau}_1^X & \in \zzhat \left( 
	\includegraphics[scale=1.0, valign=c]{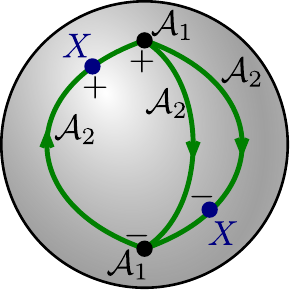}  
	\right) 
	,
	&& 
	\overline{\tau}_2^X  \in \zzhat \left( 
	\includegraphics[scale=1.0, valign=c]{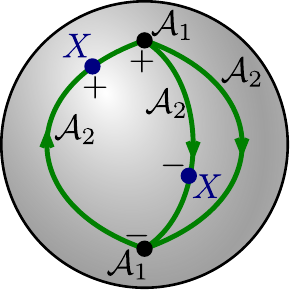}  
	\right) ~,
	\end{align}
	such that the identities in Figure~\ref{fig:CrossingIdentitiesPhiPsi} hold when $\zzhat$ is applied to both sides (viewed as defect 3-balls). 
	
	Note that as in the situation studied in \cite{MuleRunk}, the  \textsl{pseudo-inverses} $\overline{\tau}_1^X, \overline{\tau}_2^X$ are uniquely determined by $\tau_1^X, \tau_2^X$.\footnote{This may be easiest to see in the language of Gray categories discussed in Remark~\ref{rem:WAinGrayCat}.} 
	Hence we may, and will, shorten the notation to $\mathcal X = (X, \tau_1^X, \tau_2^X)$. 
	\item 
	The crossings in the tensor product 
	$
	(X, \tau_1^X, \tau_2^X\big) \otimes_\A \big(Y, \tau_1^Y, \tau_2^Y) 
	= 
	(X\otimes Y, \tau_1^{X,Y}, \tau_2^{X,Y})
	$ 
	involve additional $\psi$-insertions:
	\begin{align}
	\tau_1^{X,Y} & = 
	\includegraphics[scale=1.0, valign=c]{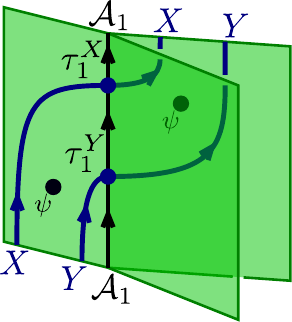} 
	\, , \quad 
	\tau_2^{X,Y}  = 
	\includegraphics[scale=1.0, valign=c]{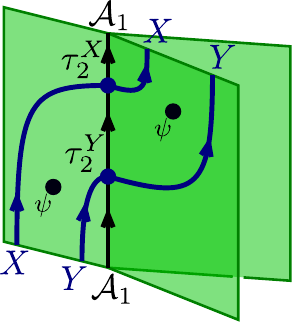} 
	\, . 
	\end{align}
	\item 
	The adjunction morphisms in~$\wa$ are
	\begin{align}
	\ev_{\mathcal X} & = 
	\includegraphics[scale=1.0, valign=c]{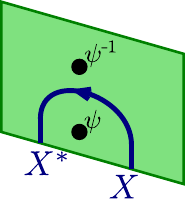}
	\, ,
	&
	\coev_{\mathcal X} & = 
	\includegraphics[scale=1.0, valign=c]{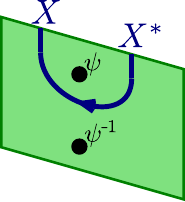}
	\, ,
	\\
	\tev_{\mathcal X} & = 
	\includegraphics[scale=1.0, valign=c]{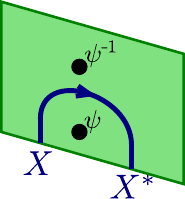}
	\, ,
	&
	\tcoev_{\mathcal X} & = 
	\includegraphics[scale=1.0, valign=c]{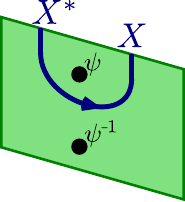}
	\, .
	\end{align}
	\item
	The braiding morpisms in~$\wa$ are
	\begin{align}
	c_{\mathcal X, \mathcal Y} & = 
	\zzhat \left( 
	\includegraphics[scale=1.0, valign=c]{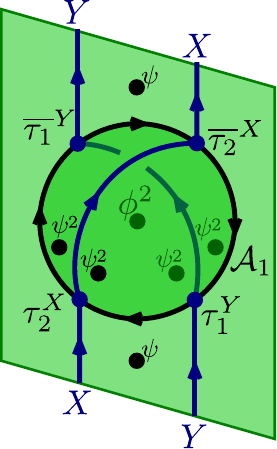}
	\right) 
	, \quad 
	c_{\mathcal X, \mathcal Y}^{-1} = 
	\zzhat \left( 
	\includegraphics[scale=1.0, valign=c]{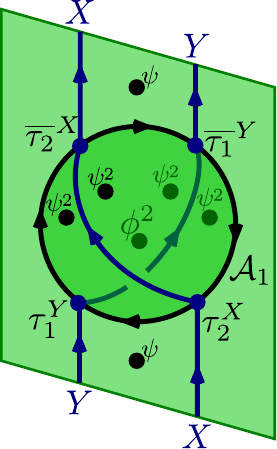}
	\right)  
	, 
	\end{align}
\end{itemize}

\begin{figure}
	\captionsetup[subfigure]{labelformat=empty}
	\centering
	\vspace{-15pt}
	\makebox[1.2\textwidth]{
		\hspace{-100pt}
		\begin{subfigure}[b]{0.6\textwidth}
			\centering
			\includegraphics[scale=1.0, valign=c]{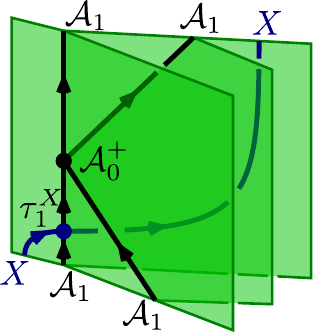} $=$
			\includegraphics[scale=1.0, valign=c]{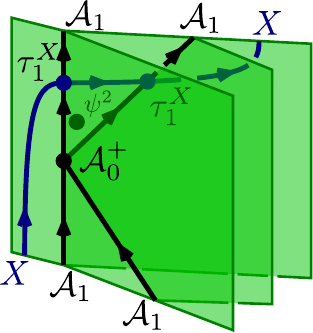}
			\caption{}
			\label{eq:TPSI1}
		\end{subfigure}\hspace{-2em}\raisebox{5.5em}{(T${}_\psi$1)}
		\hspace{-30pt}
		\begin{subfigure}[b]{0.6\textwidth}
			\centering
			\includegraphics[scale=1.0, valign=c]{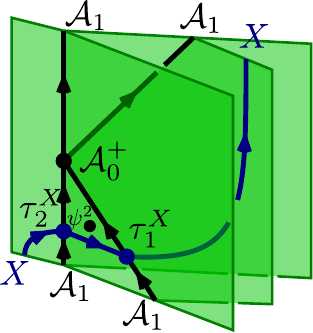} $=$
			\includegraphics[scale=1.0, valign=c]{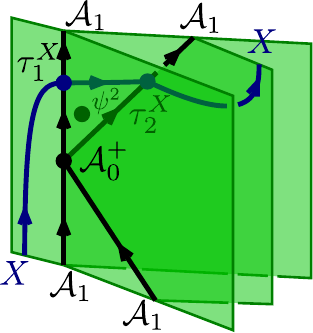}
			\caption{}
			\label{eq:TPSI2}
		\end{subfigure}\hspace{-2em}\raisebox{5.5em}{(T${}_\psi$2)}}\\
	\vspace{-10pt}
	\begin{subfigure}[b]{0.6\textwidth}
		\centering
		\includegraphics[scale=1.0, valign=c]{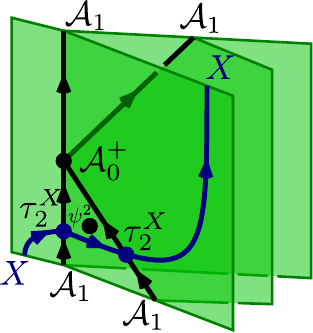} $=$
		\includegraphics[scale=1.0, valign=c]{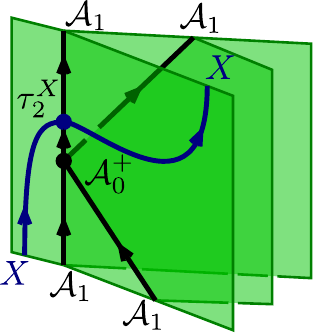}
		\caption{}
		\label{eq:TPSI3}
	\end{subfigure}\hspace{-2em}\raisebox{5.5em}{(T${}_\psi$3)}\\ 
	\vspace{-15pt}
	\begin{subfigure}[b]{0.95\textwidth}
		\centering
		\includegraphics[scale=1.0, valign=c]{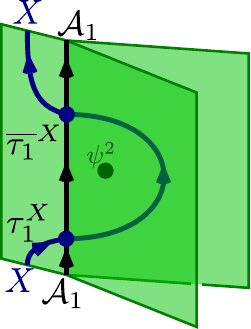} $=$
		\includegraphics[scale=1.0, valign=c]{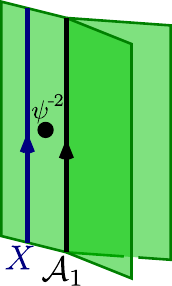}  , \quad
		\includegraphics[scale=1.0, valign=c]{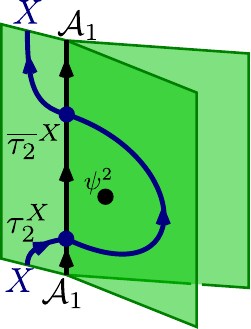} $=$
		\includegraphics[scale=1.0, valign=c]{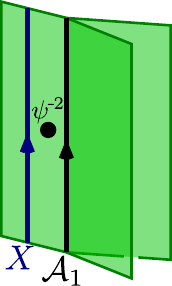}
		\caption{}
		\label{eq:TPSI4}
	\end{subfigure}\hspace{-1em}\raisebox{5.5em}{(T${}_\psi$4)}\\
	\vspace{-15pt}
	\begin{subfigure}[b]{0.95\textwidth}
		\centering
		\includegraphics[scale=1.0, valign=c]{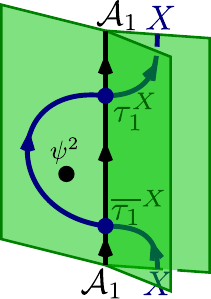} $=$
		\includegraphics[scale=1.0, valign=c]{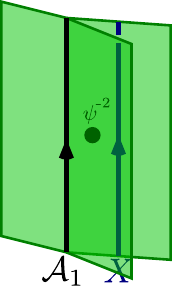}  , \quad
		\includegraphics[scale=1.0, valign=c]{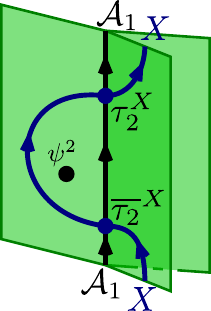} $=$
		\includegraphics[scale=1.0, valign=c]{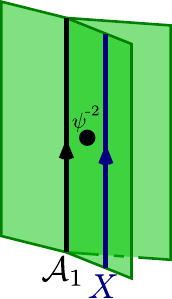}
		\caption{}
		\label{eq:TPSI5}
	\end{subfigure}\hspace{-1em}\raisebox{5.5em}{(T${}_\psi$5)}\\
	\vspace{-15pt}
	\begin{subfigure}[b]{0.95\textwidth}
		\centering
		\includegraphics[scale=1.0, valign=c]{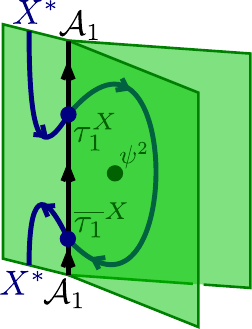} $=$
		\includegraphics[scale=1.0, valign=c]{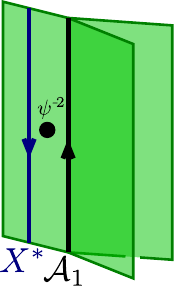}  , \quad
		\includegraphics[scale=1.0, valign=c]{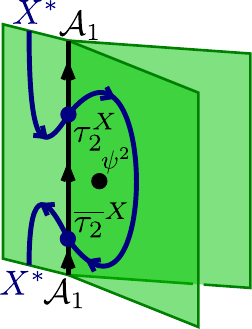} $=$
		\includegraphics[scale=1.0, valign=c]{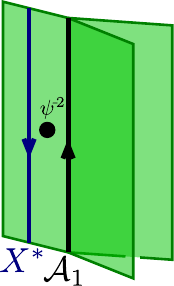}
		\caption{}
		\label{eq:TPSI6}
	\end{subfigure}\hspace{-1em}\raisebox{5.5em}{(T${}_\psi$6)}\\
	\vspace{-15pt}
	\begin{subfigure}[b]{0.95\textwidth}
		\centering
		\includegraphics[scale=1.0, valign=c]{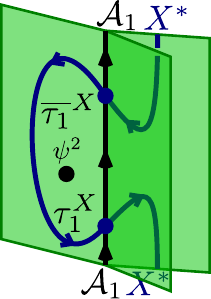} $=$
		\includegraphics[scale=1.0, valign=c]{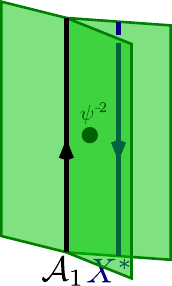}  , \quad
		\includegraphics[scale=1.0, valign=c]{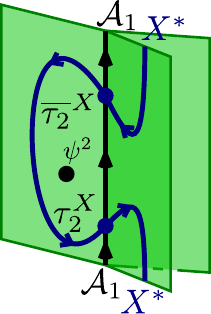} $=$
		\includegraphics[scale=1.0, valign=c]{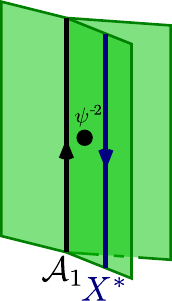}
		\caption{}
		\label{eq:TPSI7}
	\end{subfigure}\hspace{-1em}\raisebox{5.5em}{(T${}_\psi$7)}
	\vspace{-15pt}
	\caption{Defining conditions for objects in~$\wa$ with $\psi$}
	\label{fig:CrossingIdentitiesPhiPsi}
\end{figure}

\subsubsection*{$\boldsymbol{\A}$-decorated ribbon diagrams}

Now let $\A = (\A_3, \A_2, \A_1, \A^\pm_0, \psi, \phi)$ be a special orbifold datum for~$\zzhat^\odot$. 
An \textsl{$\mathcal A$-decorated ribbon diagram} $(\mathcal S, \mathscr d)$ of a bordism~$M$ with embedded $\wa$-coloured ribbon graph~$\mathcal R$ is an element $(S, \mathscr{d}) \in \mathscr S(M,\mathcal R)$ together with a decoration as follows: 
\begin{enumerate}[label={(\roman*)}]
	\item 
	$\mathcal S$ is an $\A$-decorated skeleton of~$M$ with underlying skeleton~$S$, except that $\psi$- and $\phi$-insertions are as described in parts~\ref{item:SdecoratedSkeleton} and~\ref{item:SdecoratedSkeletonPhiInsertion} below; 
	\item 
	if a switch of $(\mathcal S, \mathscr d)$ involves an $(X,\tau_1,\tau_2)$-labelled ribbon of~$\mathcal R$ traversing an $\A_1$-labelled 1-stratum of~$\mathcal S$, then the switch is labelled by $\tau_1, \tau_2, \overline\tau_1$ or~$\overline\tau_2$ as appropriate; 	
	\item 
	over- and under-crossings in~$\mathscr d$ are replaced by coupons labelled with the corresponding braiding morphisms in~$\wa$; 
	\item 
	\label{item:SdecoratedSkeleton}
	if a 2-stratum~$F$ of~$S$ is subdivided by strands of~$d$, the $\psi$-insertions pertaining to~$F$ are as follows: there is one $\psi^{\chi_{\textrm{sym}}(F_i)}$-insertion on every connected component~$F_i$ of $F\setminus d$; moreover, for each coupon~$c$ of~$\mathscr d$, the leftmost and rightmost 2-stratum components adjacent to~$c$ have one additional $\psi$-insertion each; 
	when computing $\psi^{\chi_{\textrm{sym}}(F_i)}$, boundary segments of coupons in~$\mathscr d$ are treated like boundary segments of $\partial M$ (see~\eqref{eq:PsiDistribution} below for an example); 
	\item 
	\label{item:SdecoratedSkeletonPhiInsertion}
	each 3-stratum in the interior of~$M$ obtains a $\phi^2$-insertion, and each 3-stratum adjacent to $\partial M$ obtains a $\phi$-insertion. 
\end{enumerate}

We remark that part~\ref{item:SdecoratedSkeleton} in the above definition is needed for compatibility with composition in the category~$\wa$. 
Indeed, if~$f$ and~$g$ are composable labels of two coupons in~$\mathscr d$, then~$\zzhata$ evaluates to the same vector on discs around either the two coupons or around one coupon labelled $g\circ f$: 
\be
\label{eq:PsiDistribution}
\zzhata \left( \includegraphics[scale=1.0, valign=c]{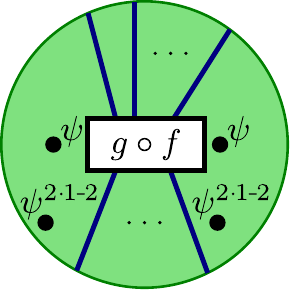} \right)
\;=\; 
\zzhata \left( \includegraphics[scale=1.0, valign=c]{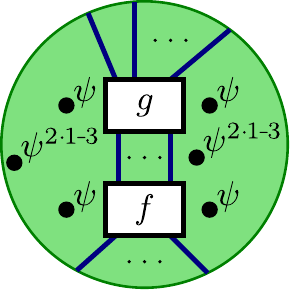} \right)
\ee 
Note that the above also illustrates the rules to have $\psi$-insertions to the left and right of coupons, and that boundary segments of coupon are treated like boundary segments of $\partial M$ when computing symmetric Euler characteristics $\chi_{\textrm{sym}}(F_i) = 2\chi(F_i) - \chi(\partial F_i)$ for $\psi$-insertions. 
In particular, $\chi_{\textrm{sym}}(F_i)=0$ for each rectangle bounded by two $\wa$-labelled strands. This is the reason that no $\psi$-insertions appear in between these strands.

\subsubsection*{Orbifold graph TQFT}

\begin{theorem}
	Let~$\mathcal A$ be a special orbifold datum for a completed defect TQFT~$\zzhat^\odot$. 
	Carrying out Construction~\ref{constr:OrbifoldGraphTQFT} for $\A$-decorated ribbon diagrams and the above ribbon category~$\wa$ with $\psi$- and $\phi$-insertions produces a symmetric monoidal functor 
	\be 
	(\zzhat_\A^\odot)^\Gamma \colon \Bordribn{3}(\wa) \lra \Vect \, . 
	\ee 
\end{theorem}

\end{document}